\documentclass{article}
\usepackage{xcolor, mathtools, amsmath, amssymb, amsthm, hyperref, appendix, aligned-overset, tikz, float, cite, siunitx}
\usepackage[textsize=tiny]{todonotes}

\usepackage{caption}
\usepackage{subcaption}
\newcommand{\figref}[1]{\figurename~\ref{#1}}
\usepackage{geometry}
\usepackage{cleveref,csquotes, algorithm2e, bm}
\usetikzlibrary{cd}
\crefname{algocf}{Algorithm}{Algorithms}

\graphicspath{{/imgs/}}

\let\P\relax
\DeclareMathOperator{\P}{\mathcal P}
\DeclareMathOperator{\C}{\mathcal C}
\DeclareMathOperator{\R}{\mathbb R}

\renewcommand{\d}{\mathop{}\!\mathrm{d}}
\DeclareMathOperator{\D}{\mathop{}\!\mathrm{D}}
\DeclareMathOperator{\vol}{vol}

\DeclareMathOperator{\KL}{KL}

\DeclareMathOperator{\grad}{grad}
\newcommand{\tT}{\mathrm{T}}
\DeclareMathOperator{\NN}{\mathcal N}
\DeclareMathOperator{\Sym}{Sym}
\DeclareMathOperator{\N}{\mathbb N}
\DeclareMathOperator{\Lip}{Lip}
\DeclareMathOperator{\id}{id}
\DeclareMathOperator{\tr}{tr}

\let\eps\varepsilon
\DeclareMathOperator{\E}{\mathbb E} 

\let\phi\varphi
\usepackage{dsfont}
\DeclareMathOperator{\1}{\mathds{1}}

\DeclareMathOperator{\diag}{diag}
\let\vec\relax
\DeclareMathOperator{\vec}{vec}

\DeclareMathOperator{\weak}{\rightharpoonup}

\DeclareMathOperator*{\argmax}{\arg\!\max}
\makeatletter
\newsavebox{\@brx}
\newcommand{\llangle}[1][]{\savebox{\@brx}{\(\m@th{#1\langle}\)}%
  \mathopen{\copy\@brx\kern-0.5\wd\@brx\usebox{\@brx}}}
\newcommand{\rrangle}[1][]{\savebox{\@brx}{\(\m@th{#1\rangle}\)}%
  \mathclose{\copy\@brx\kern-0.5\wd\@brx\usebox{\@brx}}}
\makeatother

\theoremstyle{plain}
\newtheorem{theorem}{Theorem}[section]
\newtheorem{lemma}{Lemma}[section]
\newtheorem{prop}{Proposition}[section]

\theoremstyle{definition}
\newtheorem{definition}{Definition}[section]
\newtheorem{example}{Example}[section]

\theoremstyle{remark}
\newtheorem{remark}{Remark}[section]
\AtEndEnvironment{remark}{\leavevmode\unskip\nobreak\hfill\(\blacktriangle\)}

\allowdisplaybreaks

\title{Towards understanding Accelerated Stein Variational Gradient Flow - Analysis of Generalized Bilinear Kernels for Gaussian target distributions}
\author{Viktor Stein\footnote{Institute of Mathematics,
	   Technische Universität Berlin,
	   Stra{\ss}e des 17.\ Juni 136, 
	   10623 Berlin, Germany,
	   {\ttfamily stein@math.tu-berlin.de},
      \url{https://tu.berlin/imageanalysis}.
	} \and Wuchen Li\footnote{
Department of Mathematics,
University of South Carolina,
Columbia. 1523 Greene St, Columbia, SC 29225, USA, {\ttfamily wuchen@mailbox.sc.edu}.}
}
\date{\today}

\begin{document}

\maketitle

\begin{abstract}
    Stein variational gradient descent (SVGD) is a kernel-based and non-parametric particle method for sampling from a target distribution, such as in Bayesian inference and other machine learning tasks.
    Different from other particle methods, SVGD does not require estimating the score, which is the gradient of the log-density.
    However, in practice, SVGD can be slow compared to score-estimation-based sampling algorithms.    
    To design a fast and efficient high-dimensional sampling algorithm with the advantages of SVGD, we introduce accelerated SVGD (ASVGD), based on an accelerated gradient flow in a metric space of probability densities following Nesterov's method. 
    We then derive a momentum-based discrete-time sampling algorithm, which evolves a set of particles deterministically. To stabilize the particles' position update, we also include a Wasserstein metric regularization. This paper extends the conference version \cite{SL2025}. For the bilinear kernel and Gaussian target distributions, we study the kernel parameter and damping parameters with an optimal convergence rate of the proposed dynamics. This is achieved by analyzing the linearized accelerated gradient flows at the equilibrium.
    Interestingly, the optimal parameter is a constant, which does not depend on the covariance of the target distribution. 
    For the generalized kernel functions, such as the Gaussian kernel, numerical examples with varied target distributions demonstrate the effectiveness of ASVGD compared to SVGD and other popular sampling methods.
    Furthermore, we show that in the setting of Bayesian neural networks, ASVGD outperforms SVGD significantly in terms of log-likelihood and total iteration times.
\end{abstract}

\section{Introduction}

Sampling from complicated and high-dimensional target distributions $\pi \propto e^{- f}$ is essential in many applications, including Bayesian inverse problems \cite{LBADDP2022}, Bayesian neural networks \cite{N2012}, and generative models \cite{AFHHSS2021}.
Classical sampling methods for the case of known $f$ are based on Markov Chain Monte Carlo (MCMC) methods \cite{H1970}. A famous example is the unadjusted Langevin algorithm (ULA), which is a time discretization of the overdamped Langevin dynamics. 
Score-based models \cite{SWMG2015} take a different approach and reformulate the overdamped Langevin dynamics into an ODE, which
can be interpreted as a particle discretization of the gradient flow \cite{AGS08} of the Kullback-Leibler (KL) divergence on the {Wasserstein-2} metric space of probability measures with finite second moment \cite{JKO1998}, under suitable regularity assumptions.
One way to approximate the score function is to use kernel density estimation (KDE). In practice, however, the difficulty is that score approximation based on KDE is often sensitive to the chosen kernel function and bandwidth size.

On the other hand, Stein variational gradient descent (SVGD) \cite{LW2016} is a score-estimation-free, kernel-based, deterministic interacting particle algorithm.
SVGD is the time and space discretization of the gradient flow of the KL divergence with respect to a kernelized metric in probability space, namely the Stein metric.
Since its update equation includes not only a drift vector field, but also an interaction term, often only a few particles suffice to attain good exploration of the target distribution.
Furthermore, SVGD can converge faster than Wasserstein gradient flow in the Gaussian distributions, see \cite[Thm.~3.1]{LGBP2024}.
Unfortunately, SVGD can converge quite slowly in practice. 

The idea of \enquote{accelerating} first-order optimization algorithms has a long tradition started by Nesterov \cite{N1983}, and a broad scope; for example in \cite{NNVOV2022} it used to improve on neural ODEs and in \cite{DCFS2025} it is used to propose a new type of generative model, which instead of diffusion is based on its accelerated counterpart, the damped wave equation.
There are different perspectives on this acceleration idea; an approach we do not take in this paper is to augment the space the probability measures are defined on, to the twice-as-large phase space, and then asking for convergence of the first marginal \cite{CLTW2025}.
This is closely related to Hamiltonian MCMC methods \cite{DKPR1987,MGA2024}.

In this paper, we analyze ASVGD, introduced by the authors in \cite{SL2025}, an accelerated sampling algorithm based on SVGD and Nesterov acceleration.
Like SVGD, the ASVGD algorithm avoids score estimation by using integration by parts to shift the derivative of the log-density onto the kernel function.
We show that the probability density function from the algorithm will stay in Gaussian distributions when considering an initial Gaussian distribution and a Gaussian target, while using the generalized bilinear kernel $K(x, y) \coloneqq x^{\tT} A y + 1$.
We generalize the results from \cite{LGBP2024} from $x^{\tT} y + 1$ to $x^{\tT} A y + 1$ in the SVGD case, and also investigate the accelerated SVGD in the non-centered case.
Our main theoretical contribution is analyzing which choice of parameter matrix $A$ and which damping parameter $\alpha$ gives the asymptotically optimal convergence rate in discrete time.
The overall message is that in SVGD, the convergence rates depend on $A$, and the optimal $A$ depends on the covariance matrix of the Gaussian target distribution, $Q$. While, for ASVGD, in \cref{thm:optimal_damping} we obtain $A$-independent geometric convergence rates depending on $\sqrt{\kappa(Q)}$, where $\kappa(Q)$ is the condition number of the covariance matrix $Q$.
Moreover, the optimal damping parameter $\alpha$ does not depend on the smallest eigenvalue of the matrix $Q$. 
Numerically, we show that ASVGD can perform much better than SVGD in the high-dimensional Bayesian neural network task.

\paragraph{Outline}
The plan of the paper is as follows: in Section 2, we review accelerated gradient methods in $\R^d$ and their extension to the density manifold of smooth positive probability densities in Sections 3 and 4.
In \cref{subsec:newParticleAlgo}, we then review the novel particle algorithm from \cite{SL2025}, that is score-estimation free, and realizes accelerated gradient descent in the density manifold equipped with the Stein metric.
We provide the proof for the central \cite[Lemma~1]{SL2025}, whose proof was omitted there due to lack of space.
For the (accelerated) Stein gradient flow with the bilinear kernel $x^{\tT} A y + 1$, we prove in \cref{sec:Bilinear_Analysis} that if the target measure and the initial measure are Gaussian, then the dynamics remain Gaussian for all time and provide convergence rates.
We can thus describe the pullback of the Stein metric tensor onto the Gaussian submanifold.
For the accelerated Stein gradient flow, we also prove the preservation of Gaussians and for its linearization we compute the asymptotically optimal value of $A$ and the damping parameter $\alpha$.
Lastly, in \cref{sec:numerics} we first demonstrate the advantages of ASVGD over SVGD and other popular first and second-order sampling methods using 2D toy target measures.
We then apply ASVGD to the high-dimensional problem of Bayesian neural networks and demonstrate the superiority of ASVGD over SVGD.
Lastly, we conclude and point out further directions for future research.

\paragraph{Notation}
The letter $\N$ denotes the set of positive integers.
The symmetrization of a square matrix $A \in \R^{n \times n}$ is denoted by $\Sym(A) \coloneqq \frac{1}{2} A^{\tT} + \frac{1}{2} A \in \Sym(\R; d)$.
The symmetric positive definite matrices are denoted by $\Sym_+(\R; d)$.
If not stated otherwise, $\langle \cdot, \cdot \rangle$ denotes the inner product in Euclidean space.
The derivative with respect to the time parameter $t$ is denoted by a dot, e.g., $\dot{V}(t) \coloneqq \frac{\d}{\d t} V(t)$.
The distribution of a Gaussian random variable with mean $b \in \R^d$ and covariance matrix $Q \in \Sym_+(\R; d)$ is denoted by $\NN(b, Q)$.
The set of eigenvalues of a matrix $B$ is denoted by $\sigma(B)$ and its maximal and minimal elements by $\lambda_{\max}(B)$ and $\lambda_{\min}(B)$, respectively.
The real part of a complex number $z$ is denoted by $\Re(z)$.
If not stated otherwise, $\| \cdot \|$ denotes the Euclidean norm of a vector or the Frobenius norm of a matrix.

\section{Accelerated gradient methods in Optimization} \label{sec:AcceleratedGradientMethods}
\subsection{Accelerated gradient methods on Hilbert spaces}
\paragraph{Gradient descent and gradient flow}
Consider a convex and continuously differentiable objective $f \colon \R^d \to \R$ with unique minimizer $x^* \in \R^d$.
Much of the following theory also works on real Hilbert spaces, but there are some subtle differences, see e.g. \cite{SW2024}.
We denote weak convergence by $\weak$.
The most common algorithm for finding $x^*$ is \emph{gradient descent}, first proposed by Cauchy in \cite{C1847},
\begin{equation} \label{eq:GD_Hilbert}
    x^{(k + 1)}
    = x^{(k)} - \tau_k \nabla f(x^{(k)}), \qquad x^{(0)} \in \R^d,
\end{equation}
with step size $\tau > 0$.
If the gradient $\nabla f$ is $L$-Lipschitz continuous on bounded sets (we say that $f$ is \enquote{$L$-smooth}), then for $\tau \in \left(0, \frac{2}{L}\right)$, the $x^{(k)} \weak x^*$ and $f(x^{(k)}) - f(x^*) \in O(k^{-1})$, as $k \to \infty$ \cite[Thm.~2.1.14]{N2018}.

The iteration \eqref{eq:GD_Hilbert} is the explicit-in-time Euler discretization of the \emph{gradient flow} evolution
\begin{equation} \label{eq:GF_Hilbert}
    \begin{cases}
        \dot{x}(t) + \nabla f(x(t)) = 0, & t > 0, \\
        x(0) = x^{(0)} \in \R^d.
    \end{cases}
\end{equation}
Unlike in the discrete case of \eqref{eq:GD_Hilbert} we do not need any further assumptions on $\nabla f$ to conclude $x(t) \weak x^*$ and $f(x(t)) - f(x^*) \in O(t^{-1})$ as $t \to \infty$.

In 1983, Nesterov \cite{N1983} improved upon gradient descent by introducing an extrapolation step.
The Nesterov accelerated gradient (NAG) algorithm
\begin{equation} \label{eq:NAG}
    \begin{cases}
        y^{(k)} = x^{(k)} + \frac{k - 1}{k + r - 1} (x^{(k)} - x^{(k - 1)}), \\
        x^{(k + 1)}
        = y^{(k)} - \frac{1}{L} \nabla f(y^{(k)}), \\
        x^{(0)} = y^{(0)},
    \end{cases}
\end{equation}
for $r \ge 3$ enjoys the improved quadratic rate $f(x^{(k)}) - f(x^*) \in O(k^{-2})$ for $k \to \infty$.
It was shown in \cite{SBC2016} that \eqref{eq:NAG} is a discretization of the second-order \enquote{Nesterov ODE}
\begin{equation} \label{eq:Nesterov-ODE}
    \ddot{x}(t)
    + \frac{r}{t} \dot{x}(t)
    + \nabla f(x(t))
    = 0,
\end{equation}
whose trajectories fulfill $f(x(t)) - f(x^*) = O(t^{-2})$ for $t \to \infty$ if $r = 3$ \cite{SBC2016}.
If $r > 3$, then we even have $f(x(t)) - f(x^*) = o(t^{-2})$ and $x(t) \weak x^*$ for $t \to \infty$ \cite{AP2016}.
For $r = 3$, the weak convergence of $x(t)$ is still open.

\begin{remark}
    Setting the momentum parameter $\frac{k - 1}{k + r - 1}$ to zero in Nesterov's accelerated gradient descent \eqref{eq:NAG}, recovers normal gradient descent \eqref{eq:GD_Hilbert}.
    In continuous time, a similar statement holds, see \cite[p.~4]{AGR2000}. 
\end{remark}

\begin{remark}
    The question of which damping functions $t \mapsto \gamma(t)$ instead of $\frac{r}{t}$ yield convergence and at which rates has been addressed, e.g., in \cite{CEG2009,AC2017}.
    So-called inertial dynamics with Hessian-driven damping, where the damping $\frac{r}{t}$ is replaced by $\frac{r}{t} + \beta \nabla^2 f(x(t))$, have first been considered in \cite{APR2016} and explored in subsequent works.
\end{remark}

\paragraph{The strongly convex case}
If additionally, $f$ is $\mu$-strongly convex for some $\mu > 0$, then we obtain better convergence rates, summarized in the following \cref{tab:convergence_rates_strongly_convex}.
\begin{table}[H]
    \centering
    \begin{tabular}{c|ccc}
                                & descent, $\tau = L^{-1}$                              & flow, $\alpha = \sqrt{2 \mu}$ \\  \hline
        gradient                & $O\left(( 1 - \kappa)^k\right)$                       & $O(e^{-\mu t})$  \\
        accelerated gradient    & $O\left( \left(1 - \sqrt{\kappa} \right)^k\right)$    & $O\left(\exp\left(- \frac{\sqrt{\mu}}{2} t\right)\right)$
    \end{tabular}
    \caption{Convergence rates for $k \to \infty$ resp. $t \to \infty$ of (accelerated) gradient flow and descent, see \cite[Thms.~2.1.14-15]{N2018} and \cite[Thm.~8]{SBC2016} and \cite{ACFR2022} for the last column.
    Here, $\kappa \coloneqq \frac{\mu}{L}$ denote the so-called \emph{condition number} of $f$. For equal guarantees under less restrictive assumptions, see \cite{KW2024}.}
    \label{tab:convergence_rates_strongly_convex}
\end{table}
Furthermore, in this case, the trajectory $x_{\text{r}}$ of the speed restart dynamics (see later) fulfills
$f(x_{\text{r}}(t)) - f(x^*) \in O(e^{-c t})$ for some $c$ depending only on $\kappa$, for $t \to \infty$ \cite[Thm.~10]{SBC2016}.

\paragraph{The Lagrangian/Hamiltonian perspective}

We consider an autonomous \emph{Hessian-type Lagrangian}
\begin{equation} \label{eq:Hessian-type_Lagrangian}
    L \colon \R^d \times \R^d \to \R, \qquad
    (x, v) \mapsto \frac{1}{2} \| v \|_{g(x)}^2 - f(x).
\end{equation}
Here, $\| v \|_{\nabla^2 h(x)}^2 = \langle g(x) v, v \rangle$ is the weighted inner product induced by a function $g \colon \R^d \to \Sym_+(\R; d)$.
One should think of $x$ as position and of $v$ as velocity $\dot{x}$, so that Lagrangian measures kinetic energy minus potential energy.
By minimizing the Lagrangian over absolutely continuous curves connecting two points in $\R^d$, we can define a \enquote{geodesic} distance on $\R^d$ as follows:
\begin{equation} \label{eq:geod_dist}
    \frac{1}{2} D_L(x_0, x_1)^2
    \coloneqq \inf_{\substack{x \colon [0, 1] \to \R^d \\ x(i) = x_i, \ i \in \{ 0, 1 \}}} \int_{0}^{1} L\big(x(t), \dot{x}(t)\big) \d{t}, \qquad \forall x_0, x_1 \in \R^d,
\end{equation}
see also \cite[Sec.~2.1.1]{LSSVW2025}.
Following \cite{MPTOD2018}, we can consider the \emph{damped Hamiltonian system} corresponding to the Lagrangian \eqref{eq:Hessian-type_Lagrangian}.
The autonomous Hessian-type Hamiltonian is
\begin{equation*} \label{eq:Hessian-type_Hamiltonian}
    H \colon \R^d \times \R^d \to \R, \qquad 
    (x, p) \mapsto \frac{1}{2} \langle p, [\nabla^2 h(x)]^{-1} p \rangle + f(x),
\end{equation*}
i.e., the total energy of the system is kinetic \emph{plus} potential energy.
One should think of $x$ as position and of $p$ as momentum.
In the Euclidean case of $g = \id$, we obtain $H(x, p) = \frac{1}{2} \| p \|^2 + f(x)$, and the Nesterov ODE \eqref{eq:Nesterov-ODE} can be equivalently written as a \emph{damped Hamiltonian flow}:
\begin{equation*}
    \begin{pmatrix} \dot{x} \\ \dot{v} \end{pmatrix}
    + \begin{pmatrix}
        0 \\ \alpha(t) v 
    \end{pmatrix}
    - \begin{pmatrix}
        0 & \id \\ - \id & 0
    \end{pmatrix}
    \begin{pmatrix}
        \nabla_x H(x, v) \\ \nabla_v H(x, v)
    \end{pmatrix}
    = 0.
\end{equation*}
In \cite{WL2022}, the damped Hamiltonian interpretation of accelerated gradient flow was used to generalize Nesterov's accelerated gradient descent \eqref{eq:NAG} from $\R^d$ to the space of positive smooth probability densities, which is also the perspective we will adopt in this paper.

\begin{remark}[Bregman geometry and accelerated mirror descent] \label{remark:Bregman}
    The accelerated dynamics \eqref{eq:Nesterov-ODE} have been generalized to non-Euclidean, so-called Bregman geometries in \cite{KBB2015,WWJ2016}, providing a continuous time perspective on the mirror descent algorithm \cite{NY1979,BT2003,ABB2004}, enabling the treatment of the minimization of $f$ over a closed convex subset $C \subset \R^d$.
    To this end, the authors of \cite{WWJ2016} introduced a time-dependent so-called Bregman Lagrangian \cite[Eqs.~(1),(12)]{WWJ2016} (and derived the associated time-dependent Bregman Hamiltonian \cite[Eq.~(S55)]{WWJ2016}), where the damping enters explicitly due to the time-dependence.
    For a comparison of the Hessian-type Lagrangian and the Bregman Lagrangian, see the discussion in \cite[Appendix~H]{WWJ2016}.
\end{remark}


\section{Metric Gradient Flows on the Density Manifold}

The time-dependent Bregman Lagrangian approach to accelerated schemes from \cite{WWJ2016} (see \cref{remark:Bregman}) has been generalized to probability densities in \cite{TM2019,CLTW2025}.
In the Euclidean case, the approach of \cite{TM2019} coincides with the accelerated Wasserstein gradient flow in \cite{WL2022}. 
In contrast, in \cite{WL2022}, the damped Hamiltonian approach described above is generalized to different types of kinetic energies defined via the metric tensor on the density manifold.
This approach will be the subject of the following two sections.

Let $\Omega$ be a compact connected $\C^{\infty}$ Riemannian manifold without boundary,
equipped with its volume form $\vol_{\Omega}$.
Instead of considering all probability measures on $\Omega$, we work exclusively on its subset of smooth positive probability densities,
\begin{equation*}
    \widetilde{\P}(\Omega) \coloneqq \left\{ \rho \in \C^{\infty}(\Omega): \rho(x) > 0 \ \forall x \in \Omega, \int_{\Omega} \rho(x) \d{\vol_{\Omega}}(x) = 1 \right\}.
\end{equation*}
We equip the space $\C^{\infty}(\Omega)$ with its Fréchet topology.
By \cite[Sec.~3]{L1988}, the space $\widetilde{\P}(\Omega)$ forms an infinite-dimensional $C^{\infty}$ Fréchet manifold.
This property no longer holds if we include densities with zeros ("nodes" in the language of \cite[pp.~731-732]{L1988}).
For more details on Fréchet manifolds, consult \cref{sec:Frechet} in the appendix.

To define metric gradient flows of functionals defined on $\widetilde{\P}(\Omega)$, we first recall the (co)tangent bundle and the metric on $\widetilde{\P}(\Omega)$, the latter of which is essential for defining the Hamiltonian.

\begin{definition}[(Co)tangent space to $\widetilde{\P}(\Omega)$]
    The (kinematic, see \cite{KM1997}) tangent space to $\widetilde{\P}(\Omega)$ at $\rho \in \widetilde{\P}(\Omega)$ is
    \begin{equation*}
        T_{\rho} \widetilde{\P}(\Omega)
        \coloneqq \left\{ \sigma \in \C^{\infty}(\Omega): \int_{\Omega} \sigma(x) \d{\vol_{\Omega}}(x) = 0 \right\},
    \end{equation*}
    and the cotangent space at $\rho$ is $T_{\rho}^* \widetilde{\P}(\Omega) \coloneqq \C^{\infty}(\Omega) / \R$.
\end{definition}

Hence, both the tangent $T \widetilde{\P}(\Omega)$ and the cotangent bundle $T^* \widetilde{\P}(\Omega)$ are \emph{trivial} - they factor into the product of the manifold $\widetilde{\P}(\Omega)$ and a Fréchet space.

\begin{definition}[Metric tensor field $G$ on $\widetilde{\P}(\Omega)$]
    A \textit{metric tensor field on $\widetilde{\P}(\Omega)$} is a smooth map $G \colon \rho \mapsto G_{\rho}$ on $\widetilde{\P}(\Omega)$ such that, for each $\rho \in \widetilde{\P}(\Omega)$, $G_{\rho} \colon T_{\rho} \widetilde{\P}(\Omega) \to T_{\rho}^* \widetilde{\P}(\Omega)$ is a smooth and invertible map.    
\end{definition}

\begin{remark}[Onsager operator, mobility function]
    The mapping $G_{\rho}^{-1}$ is referred to as \textit{Onsager} operator \cite{LM2013}.
    If there is a nonnegative functional $M \colon \widetilde{\P}(\Omega) \to [0, \infty)$ such that $G_{\rho}^{-1}[\Phi] = \nabla \cdot (M(\rho) \nabla \Phi)$ (and that quantity exists) for all $\Phi \in T_{\rho}^* \widetilde{\P}(\Omega)$, then we call $M$ the \textit{mobility function}, see \cite[Sec.~2.3]{LWL2024} and \cite[Rem.~12]{DNS2023}.
\end{remark}

A metric tensor field yields a \textit{metric on $\widetilde{\P}(\Omega)$} via the assignment $\widetilde{\P}(\Omega) \ni \rho \mapsto g_{\rho}$, where
\begin{equation*}
     g_{\rho} = \langle \cdot, \cdot \rangle_{\rho} \colon T_{\rho} \widetilde{P}(\Omega) \times T_{\rho} \widetilde{P}(\Omega) \to \R, \qquad (\sigma_1, \sigma_2) \mapsto \int_{\Omega} \sigma_1(x) \left(G_{\rho}[\sigma_2]\right)(x) \d{\vol_\Omega}(x).
\end{equation*}

While there are also the Fisher-Rao metric \cite{CWDS2024}, the Kalman-Wasserstein metric and the Hessian transport metric and affine-invariant modifications of some of these metrics, see, e.g. \cite{WL2022,LY2019,CHHRS2023}, we now focus on the Wasserstein-2 metric \cite{V2003,V2008} and
the Stein metric \cite{LW2016,L2017,NR2023,DNS2023,N2024}.

\begin{example}[Wasserstein-2 metric]
    The Wasserstein metric uses the inverse metric tensor field
    \begin{equation*}
        [G_\rho^{W}]^{-1} \colon T_{\rho}^* \widetilde{\P}(\Omega) \to T_{\rho} \widetilde{\P}(\Omega), \qquad
        \Phi \mapsto - \nabla \cdot (\rho \nabla \Phi), \qquad \rho \in \widetilde{\P}(\Omega).
    \end{equation*}
\end{example}

\begin{example}[Stein metric] \label{example:SteinMetric}
    For a symmetric, positive definite and smooth kernel $K \colon \Omega \times \Omega \to \R$, define the inverse metric tensor field
    \begin{equation*}
        \left[G_{\rho}^{(K)}\right]^{-1} \colon T_{\rho}^* \widetilde{\P}(\Omega) \to T_{\rho} \widetilde{\P}(\Omega), \qquad 
        \Phi \mapsto \left( x \mapsto - \nabla_x \cdot \left( \rho(x) \int_{\Omega} K(x, y) \rho(y) \nabla \Phi(y) \d{y} \right) \right)
    \end{equation*}
    Here and in the following we assume that $\Omega$ is regular enough for $(G_{\rho}^{(K)})^{-1}$ to be invertible and map $T_{\rho}^* \widetilde{\P}(\Omega)$ to $T_{\rho} \widetilde{\P}(\Omega)$ by the divergence theorem.
\end{example}

In order to introduce gradient flows on $\widetilde{\P}(\Omega)$, we first need define a suitable notation of differential $\delta E$ (i.e., a covector field) for functionals $E$ on $\widetilde{\P}(\Omega)$, which uses the linear structure of the ambient vector space $\C^{\infty}(\Omega)$.
We can then conveniently express the gradient (i.e., (tangent) vector field) of such functionals as $G_{\rho}^{-1}[\delta E(\rho)]$, recovering the familiar formula $g(\grad(E), s) = \langle \delta E, s \rangle$ for any tangent vector field $s$.

\begin{definition}[First linear functional derivatives]
    Consider a functional $E \colon \widetilde{\P}(\Omega) \to \R$.
    The first variation (or: linear first order functional derivatives or: $L^2(\Omega)$-gradient) of $E$, if it exists, is the one-form $\delta E \colon \widetilde{\P}(\Omega) \to \C^{\infty}(\Omega)$ fulfilling for all $\rho \in \widetilde{\P}(\Omega)$ that 
    \begin{gather*}
        \langle \delta E(\rho), \phi \rangle_{L^2(\Omega)}
        = \frac{\d}{\d{t}} E(\rho + t \phi) \bigg|_{t = 0} \quad \forall \phi \in \C^{\infty}(\Omega) \text{ s.t. } \rho + t \phi \in \widetilde{\P}(\Omega) \text{ for sufficiently small } t \in \R.
    \end{gather*}
\end{definition}
In the following, we assume that the functional derivative exists and is unique, which is the case for our functional of interest, the KL-divergence, if $\Omega$ is bounded.

We summarize some examples of functional derivatives in the following \cref{table:functional_derivatives}.
\begin{table}[H]
    \centering
    \begin{tabular}{c|c|c}
        Functional                      & $E(\rho)$                                                                                                     & $\delta E(\rho)$  \\ \hline
        Dirichlet energy                & $\frac{1}{2} \| \nabla \rho \|_{L^2(\Omega)}^2$                                                               & $-\Delta \rho$          \\
        Entropy                         & $\int_{\Omega} f(\rho(x)) \d{x}$                                                                              & $f' \circ \rho$        \\
        Interaction                     & $\int_{\Omega} \int_{\Omega} W(x, y) \rho(x) \rho(y) \d{x} \d{y}$                                             & $\int_{\Omega} W(x, \cdot) \rho(x) \d{x}$ \\
        Potential                       & $\int_{\Omega} V(x) \rho(x) \d{x}$                                                                            & $V$   \\
        $D_f(\cdot \mid \rho^*)$        & $\int_{\Omega} f\left(\frac{\rho(x)}{\rho^*(x)}\right) \rho^*(x) \d{x}$                                       & $f'\left(\frac{\rho}{\rho^*}\right)$    \\
        MMD$_{K, \rho^*}$               & $\frac{1}{2} \int_{\Omega} K(x, y) \big(\rho(x) - \rho^*(x) \big) \big(\rho(y) - \rho^*(y) \big) \d{x} \d{y}$ & $\int_{\Omega} K(x, \cdot) (\rho(x) - \rho^*(x)) \d{x}$ \\
        $R_{\alpha}(\cdot \mid \rho^*)$ & $\frac{1}{\alpha - 1} \ln\left(\int_{\Omega} \rho(x)^{\alpha} \tilde{\rho}(x)^{1 - \alpha} \d{x}\right)$      & $\frac{\alpha}{\alpha - 1} \frac{1}{I_{\alpha}(\rho \mid \rho^*)} \left(\frac{\rho}{\rho^*}\right)^{\alpha - 1}$   \\
    \end{tabular}
    \caption{First variation of popular energy functionals, including the maximum mean discrepancy (MMD), the $f$-divergence $D_f$, and the $\alpha$-Rényi divergence $R_{\alpha}$ for $\alpha \in (0, 1) \cup (1, \infty)$, with target density $\rho^* \in \widetilde{\P}(\Omega)$. Here, $I_{\alpha}$ is the $\alpha$-mutual information.}
    \label{table:functional_derivatives}
\end{table}

\begin{definition}[Metric gradient flow on $\widetilde{\P}(\Omega)$]
    We say that a smooth curve $\rho \colon [0, \infty) \to \widetilde{\P}(\Omega)$, $t \mapsto \rho_t$ is a $(\widetilde{\P}(\Omega), G)$-gradient flow of a functional $E \colon \widetilde{\P}(\Omega) \to \R$ starting at $\rho(0)$ if 
    \begin{equation} \label{eq:metric_GF}
        \partial_t \rho_t
        = - G_{\rho_t}^{-1}[\delta E(\rho_t)] \qquad \forall t > 0.
    \end{equation}
\end{definition}

\section{Accelerated Gradient Flows}
As described in \cref{sec:AcceleratedGradientMethods}, generalizing the continuous limit of Nesterov's accelerated gradient descent \cite{N1983} to the density manifold $\widetilde{\P}(\Omega)$ yields the following definition.

\begin{definition}[Accelerated $(\widetilde{\P}(\Omega), G)$-gradient flow {\cite[Sec.~3]{WL2022}}]
    Let $\alpha \colon [0, \infty) \to [0, \infty)$ be a function, referred to as \enquote{damping}.
    The $\alpha$-accelerated $(\widetilde{\P}(\Omega), G)$-gradient flow of $E \colon \widetilde{\P}(\Omega) \to \R$ is the curve $(\rho_t)_{t > 0}$ solving the linearly-damped Hamiltonian flow
    \begin{equation} \label{eq:acc_Hamiltonian_flow}
        \partial_t \begin{pmatrix}
            \rho_t \\ \Phi_t
        \end{pmatrix}
        + \begin{pmatrix}
            0 \\ \alpha_t \Phi_t
        \end{pmatrix}
        - \begin{pmatrix}
            0 & 1 \\ -1 & 0
        \end{pmatrix} 
        \begin{pmatrix}
            \delta_1 H(\rho_t, \Phi_t) \\
            \delta_2 H(\rho_t, \Phi_t)
        \end{pmatrix}
        = 0,
        \qquad \rho(0) = \rho_0, \Phi(0) = 0,
    \end{equation}
    where 
    \begin{equation*}
        H \colon T \widetilde{\P}(\Omega) \to \R, \qquad 
        (\rho, \Phi) \mapsto \frac{1}{2} g_{\rho}(\Phi, \Phi) + E(\rho)
    \end{equation*}
    is the Hamiltonian, and $\delta_i$ for $i \in \{ 1, 2 \}$ denotes the functional derivative with respect to the $i$-th component.
    
    By \cite[Prop.~1]{WL2022}, the damped Hamiltonian flow \eqref{eq:acc_Hamiltonian_flow} can be rewritten as
    \begin{equation} \label{eq:Accelerated_GF}
        \begin{cases}
            \partial_t \rho_t - G_{\rho_t}^{-1}[\Phi_t] = 0, \\
            \partial_t \Phi_t + \alpha_t \Phi_t 
            + \frac{1}{2} \delta \left\{ \rho \mapsto g_{\rho}(\Phi_t, \Phi_t) \right\}(\rho_t) + \delta E(\rho_t)
            = 0,
        \end{cases}
        \qquad \rho(0) = \rho_0, \Phi(0) = 0.
    \end{equation}
\end{definition}

\begin{remark}
    A big difference to other accelerated gradient flows in spaces of probability measures like \cite{TM2019,CLTW2025} is that our Hamiltonian does not separate into a position-dependent term and a velocity-dependent term. In particular, the mobility function depends non-linearly on the density, which makes analysis more complicated.
\end{remark}

\begin{example}[Accelerated Stein flow {\cite[Ex.~8]{WL2022}}]
    The accelerated Stein metric gradient flow of some functional $E \colon \widetilde{\P}(\Omega) \to \R$ is
    \begin{equation} \label{eq:Accelerated_Stein_Flow}
        \begin{cases}
            \partial_t \rho_t + \nabla \cdot \left(\rho_t \displaystyle\int_{\Omega} K(\cdot, y) \rho_t(y) \nabla \Phi_t(y) \d{y}\right)
            = 0, \\
            \displaystyle\partial_t \Phi_t + \alpha_t \Phi_t 
            + \int_{\Omega} K(y, \cdot) \langle \nabla \Phi_t(y), \nabla \Phi_t(\cdot) \rangle_{\R^n} \rho_t(y) \d{y} + \delta E(\rho_t)
            = 0.
        \end{cases}
        \quad \rho_0 \in \widetilde{\P}(\Omega), \ \Phi_0 = 0.
    \end{equation}
\end{example}

Since Wasserstein gradient flows are generally stable, whereas Stein metric flows can degenerate - at least numerically - due to (near-)singularity of the kernel Gram matrix, we will consider a mixed Stein-Wasserstein metric to increase stability.

\begin{example}[Stein-Wasserstein metric] \label{example:S-WS}
    For $\eps > 0$ and a kernel like in \cref{example:SteinMetric}, the \emph{Stein-Wasserstein metric} is induced by the metric tensor field
    \begin{equation} \label{eq:SWSmetricTensorField}
        (G_{\rho}^{(W,K)})^{-1} \coloneqq (G_{\rho}^{(K)})^{-1} + \eps (G_{\rho}^{W})^{-1}.
    \end{equation}

    By a simple computation (compare \cite[Exs.~6,~8]{WL2022}), the accelerated gradient flow of some functional $E \colon \widetilde{\P}(\Omega) \to \R$ with respect to the Stein-Wasserstein metric is
    \begin{equation} \label{eq:S_WS}
        \begin{cases}
            \partial_t \rho_t + \nabla \cdot \left(\rho_t \left( \displaystyle\int_{\Omega} K(\cdot, y) \rho_t(y) \nabla \Phi_t(y) \d{y}+ \eps \nabla \Phi_t\right)\right) = 0, \\
            \displaystyle\partial_t \Phi_t + \alpha_t \Phi_t 
            + \int_{\Omega} K(y, \cdot) \langle \nabla \Phi_t(y), \nabla \Phi_t(\cdot) \rangle_{\R^n} \rho_t(y) \d{y}
            + \frac{\eps}{2} \| \nabla \Phi_t \|_2^2  + \delta E(\rho_t)
            = 0.
        \end{cases}
    \end{equation}
\end{example}
Assuming that the dynamics \eqref{eq:S_WS} conserve the unit mass of $\rho_t$, a time-continuous particle formulation of \eqref{eq:Accelerated_Stein_Flow} with $V_t \coloneqq \nabla \Phi_t(X_t)$ (compare with \cite[Eqs.~(10),(57)]{WL2022}) is
\begin{equation} \label{eq:WLparticle}
    \begin{cases}
        \displaystyle\dot{X}_t = \int_{\Omega} K(X_t, y) \nabla \Phi_t(y) \rho_t(y) \d{y} \\
        \displaystyle\dot{V}_t = - \alpha_t V_t
        - \displaystyle\int_{\Omega} \nabla_1 K(X_t, y) \langle \nabla \Phi_t(y), V_t \rangle \rho_t(y) \d{y}
        - \nabla \delta E(\rho_t)(X_t)
    \end{cases}
    \quad X_0 \in \Omega, \ V_0 = 0.
\end{equation}
We defer an informal justification of \eqref{eq:WLparticle} to \cref{subsec:particleFormulation}.

\begin{remark}[Choice of energy]
    In this work we focus on $\Omega = \R^d$ and $E = \KL(\cdot \mid Z^{-1} e^{-f})$ being the KL divergence, whose target is the Gibbs-Boltzmann distribution $Z^{-1} e^{-f}$ with differentiable \textit{potential} $f \colon \Omega \to \R$ such that the normalization constant $Z \coloneqq \int_{\Omega} e^{-f} \d{\vol_{\Omega}}(x)$ is finite.
    This loss function is widely used in applications like variational inference and Bayesian statistics.
    Another reason for focusing on this specific energy is that, as demonstrated in \cite{CHHRS2023}, it is the only $f$-divergence (up to rescaling) with continuously differentiable $f$ that is invariant under scaling of the target measure.
    By \cite{C2025}, it is also the only Bregman divergence with this property.
    This invariance property can simplify the numerical implementation of particle methods \cite{CHHRS2023}.
\end{remark}

\begin{remark}[Particle methods and score estimation]
    The disadvantage of the particle method \eqref{eq:WLparticle} is that we have to estimate the density $\rho_t$ in order to evaluate the term $\nabla \delta E(\rho)$; for $E = \KL(\cdot \mid Z^{-1} e^{- f})$, we have $\nabla \delta E(\rho) = \nabla \log(\rho) + \nabla f$, where the first summand is called the \emph{score of $\rho$}.
    However, all other terms involving $\rho_t$ are expectations against $\rho_t$ and can thus be efficiently estimated using i.i.d. samples from $\rho_t$.
\end{remark}


\subsection{A score-estimation-free particle algorithm} \label{subsec:newParticleAlgo}
In this subsection, we review the particle formulation of \eqref{eq:S_WS}, which was introduced \cite{SL2025} and, in contrast to the particle formulation \cite[Eq.~(58)]{WL2022}, does not need any score estimation.
Furthermore, we provide a full proof of \cref{lemma:central_Lemma}, which was omitted in \cite{SL2025} due to space constraints.

Let $(\rho_t, \Phi_t)_{t > 0}$ fulfill the Stein-Wasserstein accelerated gradient flow \eqref{eq:S_WS}, and $X_t \sim \rho_t$ be a particle.
Then, we have
\begin{equation} \label{eq:Acc_Stein}
    \dot{X}_t
    = \int_{\Omega} K(X_t, y) \nabla \Phi_t(y) \rho_t(y) \d{y} + \eps V_t, \qquad t > 0.
\end{equation}
We define velocity of the particle $X$ over time, $Y \colon (0, \infty) \to \R^d$, $t \mapsto \dot{X}_t$.
In the following lemma, we calculate the particle's acceleration $\dot{Y}_t$.

\begin{lemma}[Acceleration in the Stein-Wasserstein flow] \label{lemma:central_Lemma}
    Let $(\rho_t, \Phi_t)_{t > 0}$ solve \eqref{eq:S_WS}, $X_t \sim \rho_t$ and $Y_t \coloneqq \dot{X}_t$.
    For all $t > 0$ we have
    \begin{equation} \label{eq:dot_Yt}
    \begin{aligned}
        \dot{Y}_t
        & = - \alpha_t Y_t - \int_{\Omega} K(X_t, y) \nabla \delta E(\rho_t)(y) \rho_t(y) \d{y} 
        + \int_{\Omega} \int_{\Omega} \rho_t(y) \rho_t(z) \langle \nabla \Phi_t(z), \nabla \Phi_t(y) \rangle_{\R^n} \\
        & \qquad \cdot \bigg[K(y, z) (\nabla_2 K)(X_t, y) + K(X_t, z) (\nabla_1 K)(X_t, y) - K(X_t, y) (\nabla_2 K)(z, y) \bigg] \d{y} \d{z} \\
        & \qquad + \eps \bigg( \int_{\Omega} \left( (\nabla_1 K)(X_t, y) \langle V_t, \nabla \Phi_t(y) \rangle_{\R^n} + (\nabla_2 K)(X_t, y) \| \nabla \Phi_t(y) \|_2^2\right) \rho_t(y) \d{y}
        + \alpha_t V_t
         + \dot{V}_t \bigg).
    \end{aligned}
    \end{equation}
\end{lemma}

\begin{proof}
    In \cref{subsec:proof}.
\end{proof}

\begin{remark}
    For $\eps = 0$ the right-hand side of \eqref{eq:dot_Yt} consists of the \emph{damping term} $- \alpha_t Y_t$, the \emph{energy term} $\int_{\Omega} K(X_t, y) \nabla \delta E(\rho_t)(y) \rho_t(y) \d{y}$ and an \emph{interaction term}.
\end{remark}

\begin{remark}
    If $K(x, y) = \phi(x - y)$ for some $\phi \colon \R^d \to \R$, that is, $K$ is \textit{translation invariant}, then $\nabla_2 K = - \nabla_1 K$ and so the term in the square brackets in the second line of \eqref{eq:dot_Yt} simplifies a bit.
\end{remark}

\begin{remark}[Simplifying the KL energy term] \label{remark:Simplied_Energy_Term_KL}
    We can use integration by parts to avoid having to estimate the score $\delta E(\rho) = \ln\left(\frac{\rho}{\pi}\right) + 1$.
    Indeed,
    \begin{equation*}
        \nabla \delta E(\rho) = \nabla \ln \circ \rho - \nabla \ln \circ \pi = \nabla \ln \circ \rho + \nabla f, \qquad \rho \in \widetilde{\P}(\Omega),    
    \end{equation*}
    and thus integration by parts under sufficient regularity of the boundary yields that the \textit{energy term} in \eqref{eq:dot_Yt} is, for $x \in \Omega$, equal to
    \begin{align*}
        - \int_{\Omega} K(x, y) \nabla \delta E(\rho)(y) \rho_t(y) \d{y}
        & = - \int_{\Omega} K(x, y) \nabla \rho_t(y) + K(x, y) \rho_t(y) \nabla f(y) \d{y} \\
        & = \int_{\Omega} \left( \nabla_2 K(x, y) - K(x, y) \nabla f(y)\right) \rho_t(y) \d{y}.
    \end{align*} 
\end{remark}

In summary, for $\eps = 0$, our deterministic interacting particle scheme associated to \eqref{eq:Accelerated_Stein_Flow} is
\begin{equation} \label{eq:Accelerated_Stein_Flow_2}
    \begin{cases}
        \dot{X}_t = Y_t, \\
        \dot{Y}_t = - \alpha_t Y_t 
        + \displaystyle\int_{\Omega} \left( \nabla_2 K(X_t, y) - K(X_t, y) \nabla f(y)\right) \rho_t(y) \d{y}
        + \displaystyle\int_{\Omega} \int_{\Omega} \rho_t(y) \rho_t(z) \langle \nabla \Phi_t(z), \nabla \Phi_t(y) \rangle \\
        \qquad \qquad \cdot \bigg[K(y, z) (\nabla_2 K)(X_t, y) + K(X_t, z) (\nabla_1 K)(X_t, y) - K(X_t, y) (\nabla_2 K)(z, y) \bigg] \d{y} \d{z}.
    \end{cases}
\end{equation}

Discretizing explicitly in time yields the following update equations.
\begin{lemma}[ASVGD update equations] \label{lemma:ASVGD_fully_discrete}
    The full discretization in time and space of \eqref{eq:Accelerated_Stein_Flow_2} is
    \begin{equation*}
        \begin{cases}
            X_{k + 1}
            \gets X_{k} + \tau_k Y_{k} \\
            V_{k + 1}
            \gets N K_{k + 1}^{\dagger} Y_{k} \\
            Y_{k + 1}
            \gets \alpha_{k} Y_{k}
            + \displaystyle\frac{\tau}{N} \left( \left(\sum_{j = 1}^{N} A_{i, j}^{k + 1, (2)}\right)_{i = 1}^{N} - K_{k + 1} G_{k + 1} \right) \\
            \qquad \qquad + \displaystyle\frac{\tau}{N^2} \left[ 
            \left(\sum_{i = 1}^{N} A_{j, i}^{k + 1, (2)} (V_k V_k^{\tT} \odot K_k) \1_N)_i
            + (V_k V_k^{\tT} K_k)_{i, j} A_{j, i}^{k + 1, (1)} \right)_{j = 1}^{N}
            - K_{k + 1} r_{k + 1} 
            \right],
        \end{cases}
    \end{equation*}
    where $\alpha_k = \frac{k - 1}{k + 2}$ or $\alpha_k$ is constant, $\dagger$ denotes the pseudo-inverse, and for $k \in \N$ we set
    \begin{gather*}
        X_{k}
        \coloneqq \begin{pmatrix} X_{k}^{(1)} & \ldots & X_{k}^{(N)} \end{pmatrix}^{\tT} \in \R^{N \times d}, \quad 
        \text{and } Y_{k}, 
        V_{k}
        \in \R^{N \times d} 
        \text{ analogously} \\
        G_{k}
        \coloneqq \begin{pmatrix}
            \nabla f(X_{k}^{(1)}) & \ldots & \nabla f(X_{k}^{(N)})
        \end{pmatrix}^{\tT} \in \R^{N \times d}, \quad
        K_{k}
        \coloneqq \left( K\left(X_{k}^{(i)}, X_{k}^{(j)}\right) \right)_{i, j = 1}^{N} \in \R^{N \times N}, \\
        A^{k, (m)}
        \coloneqq \big(\nabla_m K(X_k^{(i)}, X_k^{(j)})\big)_{i, j} \in \R^{N \times N \times d}, \quad m \in \{ 1, 2 \}, \qquad
        (r_k)_i \coloneqq \sum_{\ell = 1}^{N} \langle V_k^{(i)}, V_k^{(\ell)} \rangle A_{j, i}^{k, (2)} \in \R^{N}.
    \end{gather*}
\end{lemma}

\begin{proof}
    See \cref{subsec:proof_full_discretization}.
\end{proof}

For the Gaussian kernel and the generalized bilinear kernel, the particle updates from \cref{lemma:ASVGD_fully_discrete} simplify.
We summarize them in \cref{algo:ASVGD}.
\begin{algorithm}
    \caption{Accelerated Stein variational gradient descent} \label{algo:ASVGD}
    \KwData{Number of particles $N \in \N$, number of steps $M \in \N$, step size $\tau > 0$, target score function $\nabla f \colon \R^d \to \R^d$. Either $A \in \Sym(\R; d)$ or a bandwidth $\sigma^2 > 0$, regularization parameter $\eps \ge 0$, (constant damping $\beta \in (0, 1)$).}
    \KwResult{Matrix $X^{M}$, whose rows are particles that approximate the target distribution $\pi \propto \exp(-f)$.}
    \textbf{Step 0.} Initialize $Y^0 = 0 \in \R^{N \times d}$ and  \texttt{restart\_count}$=\1_N$. \\
    \For{$k=0,\ldots, M - 1$}{
    \emph{\textbf{Step 1.} Update particle positions using particle momenta.}
    $$\displaystyle X^{k + 1}
    \gets X^{k} + \sqrt{\tau} Y^{k}.$$
    \emph{\textbf{Step 2.} Form kernel matrix and update momentum in density space.}
    $$
    K^{k + 1}
    = \big(K(X_i^{k + 1}, X_j^{k + 1}) \big)_{i, j = 1}^{N},
    \qquad
    V^{k + 1}
    \gets N (K^{k + 1} + \eps \id_N)^{-1} Y^{k}.$$\\
    \emph{\textbf{Step 3.} Update damping parameter using speed and gradient restart.} \\
    \For{$i = 1, \ldots, N$}
    {\eIf{$\| X_{i}^{k + 1} - X_{i}^{k} \|_2 < \| X_{i}^{k} - X_{i}^{k - 1} \|_2$}
    {\texttt{restart\_count}$_i= 1$}
    {\texttt{restart\_count}$_i += 1$}
    } 
    Only for the Gaussian kernel: \\
    \If{$\tr\left((V^{k + 1})^{\tT} (K^{k + 1} \nabla f(X^{k + 1}) + \big(K^{k + 1} - \diag(K^{k + 1} \1_N)\big) X^{k + 1})\right) < 0$,}{\texttt{restart\_count} $= \1_N$}
    $\alpha_{i}^k = \displaystyle\frac{\texttt{restart\_count}_i - 1}{\texttt{restart\_count}_i + 2}$, $i \in \{ 1, \ldots, N \}$. \\
    Alternatively, set a \textit{constant damping} for each particle: $\alpha_i^k = \beta$. \\
    \emph{\textbf{Step 4.} Update momenta.} \\
    For the bilinear kernel:
    \begin{align*}
        Y^{k + 1}
        & \gets \alpha^{k} Y^{k}
        - \frac{\sqrt{\tau}}{N} K^{k + 1} \nabla f(X^{k + 1})
        + \sqrt{\tau} \left(1 + N^{-2} \tr\left((V^{k + 1})^{\tT} K^{k + 1} V^{k + 1}\right)\right) X^{k + 1} A.   
    \end{align*}
    For the Gaussian kernel: 
    \begin{align*}
        W^{k + 1}
        & \gets N K^{k + 1} + K^{k + 1} (V^{k + 1} (V^{k + 1})^{\tT}) \circ K^{k + 1} - K^{k + 1} \circ (K^{k + 1} V^{k + 1} (V^{k + 1})^{\tT}), \\
        Y^{k + 1}
        & \gets \alpha^{k} Y^{k} - \frac{\sqrt{\tau}}{N} K^{k + 1} \nabla f(X^{k + 1}) + \frac{\sqrt{\tau}}{N^2 \sigma^2} \left(\diag(W^{k + 1} \1_N) - W^{k + 1}\right) X^{k + 1}.
    \end{align*}
    }
\end{algorithm}

\begin{remark}[Wasserstein metric regularization]
    The Gaussian kernel matrix is invertible if the particles are distinct.
    Since its inversion can become ill-conditioned, in \cref{algo:ASVGD} we instead invert $K^{k + 1} + \eps \id_N$ for some small $\eps > 0$.
    By \cref{example:S-WS}, this exactly corresponds to adding the Wasserstein metric regularization $\eps [G_{\cdot}^{(W)}]^{-1}$ to the Stein metric and adding $\eps V_t$ to the right side of \eqref{eq:Acc_Stein}.
    To make the resulting algorithm tractable, we ignore the corresponding $\eps$-terms in the $Y$-update.
    Note that for the generalized bilinear kernel $K + \eps \id$ is a low rank matrix, than can be easily pseudo-inverted by the Woodbury formula.
\end{remark}

\section{Non-accelerated and Accelerated Stein Gradient Flow with Generalized Bilinear Kernel on Gaussians} \label{sec:Bilinear_Analysis}

Let us now focus on the generalized bilinear kernel $K(x, y) \coloneqq x^{\tT} A y + 1$.
It is one of the most simple kernels and appears as first-order approximation of other standard kernels such as the (anisotropic) Gaussian kernel: $e^{- \frac{1}{2} (x - y)^{\tT} A (x - y)} \approx 1 + x^{\tT} A y + O(x^2, y^2)$.
For $A = \id_d$, the authors of \cite{LGBP2024} demonstrate that this kernel is important for Gaussian variational inference.
As a novel contribution, we generalize their results \cite[Thms.~3.1,~3.5]{LGBP2024} to arbitrary $A \in \Sym_+d(\R; d)$, by mimicking their proofs.
For this, we first prove that for a Gaussian target and initialization, the flow stays a Gaussian for all times.

For this, consider the inclusion map
\begin{equation*}
    \phi \colon \Theta \to \widetilde{\P}(\R^d), \qquad
    \theta \mapsto \rho_{\theta},
\end{equation*}
which is a smooth immersion \cite[Def.~I.4.4.8]{H1982} into the density manifold by \cite[Ex.~I.4.4.9]{H1982}.
We describe the pullback of the Stein metric onto the submanifold $\phi(\Theta) \subset \widetilde{\P}(\R^d)$ of Gaussian densities, defined on the tangent space at $\theta \in \Theta$ by
\begin{equation*}
    \tilde{g}_{\theta}
    \coloneqq (\phi^* g)_{\theta} \colon T_{\theta} \Theta \times T_{\theta} \Theta \to \R, \qquad 
    (\xi, \eta)
    \mapsto g_{\rho_{\theta}}(\d \phi_{\theta}(\xi), \d \phi_{\theta}(\eta)).
\end{equation*}

\begin{lemma}[Gaussian-Stein metric tensor] \label{lemma:Gaussian-Stein_metric_tensor}    
    Let $\theta = (\mu, \Sigma) \in \R^d \times \Sym_+(\R; d)$.
    For $(\tilde{\mu}_1, \tilde{\Sigma}_1), (\tilde{\mu}_2, \tilde{\Sigma}_2) \in T_{\theta} \Theta$ the pulled back Riemannian metric is
    \begin{equation*}
        \tilde{g}_{\theta}((\tilde{\mu}_1, \tilde{\Sigma}_1), (\tilde{\mu}_2, \tilde{\Sigma}_2))
        = \tr(S_1 S_2 \Sigma A \Sigma) + (b_1^{\tT} S_2 + b_2^{\tT} S_1) \Sigma A \mu + K(\mu, \mu) b_1^{\tT} b_2,
    \end{equation*}
    where
    \begin{equation*}
        (b_i, \frac{1}{2} S_i)
        = \tilde{G}_{\theta}(\tilde{\mu}_i, \tilde{\Sigma}_i), \qquad i \in \{ 1, 2 \},
    \end{equation*}
    and the tangent-cotangent automorphism associated to $\tilde{g}_{\theta}$ is
    \begin{equation} \label{eq:Stein-Gaussian_metric}
        \tilde{G}_{\theta}^{-1}
        \colon \R^d \times \Sym(d) \to \R^d \times \Sym(d), \qquad
        (\nu, S) \mapsto \begin{pmatrix}
            2 S \Sigma A \mu + K(\mu, \mu) \nu \\
            2 \Sym\left(\Sigma A (2 \Sigma S + \mu \nu^{\tT})\right) 
        \end{pmatrix}.        
    \end{equation}
\end{lemma}

\begin{proof}
    Analogous to \cite[Thm.~A.3]{LGBP2024}, see \cref{subsec:proof_Gauss_Stein_tensor}.
\end{proof}

\begin{remark}[SVGD with time-dependent kernel recovers Wasserstein and natural gradient]
    In the zero-mean case, the automorphism \eqref{eq:Stein-Gaussian_metric} becomes
    \begin{equation*} 
        \tilde{G}_{\Sigma}^{-1}
        \colon \Sym(d) \to \Sym(d), \qquad
        S \mapsto 4 \Sym\left(\Sigma A \Sigma S\right).      
    \end{equation*}
    Formally, the time-dependent choice $A = \Sigma^{-1}$ recovers the Wasserstein case since then $\tilde{G}_{\Sigma}^{-1}(S) = 4 \Sym(\Sigma S)$, see \cite[Thm.~A.6]{LGBP2024}.
    Further, the time-dependent choice $A = \Sigma^{-2}$ yields $\tilde{G}_{\Sigma}^{-1}(S) = 4 S$ and thus natural gradient descent \cite{CL2020}.

    This motivates searching for a fixed matrix $A \in \Sym_+(d)$ for which the convergence of the Stein metric gradient flow with generalized bilinear kernel is the best.
\end{remark}

\begin{theorem}[Preservation of Gaussians] \label{thm:3.1}
    Let $\rho_0 \sim \NN(\mu_0, \Sigma_0), \rho^* \sim \NN(b, Q)$ be Gaussian distributions, with means $\mu_0, b \in \R^d$ and covariance matrices $\Sigma_0, Q \in \Sym_+(d)$.
    Further, let $(\rho_t)_{t \ge 0}$ be the Stein gradient flow with respect to the generalized bilinear kernel, of $\KL(\cdot \mid \rho^*)$, starting at $\rho_0$.
    \begin{enumerate}
        \item 
        Then, $\rho_t \sim \NN(\mu_t, \Sigma_t)$ for all $t \ge 0$, where
        \begin{equation} \label{eq:Stein_gradient_flow_Gaussians}
            \begin{cases}
                \dot{\mu}_t
                = (\id - Q^{-1} \Sigma_t) A \mu_t
                - K(\mu_t, \mu_t) Q^{-1}(\mu_t - b), \\
                \dot{\Sigma}_t
                = 2 \Sym(\Sigma_t A)
                - 2 \Sym\left(\Sigma_t A \big( \Sigma_t + \mu_t (\mu_t - b)^{\tT} \big) Q^{-1}\right).
            \end{cases}
        \end{equation}

        \item 
        For any $\mu_0 \in \R^d$ and $\Sigma_0 \in \Sym_+(d)$, the system \eqref{eq:Stein_gradient_flow_Gaussians} has a unique solution on $[0, \infty)$.
        For $t \to \infty$ we have $\rho_t \weak \rho^*$ and $\| \mu_t - b \|, \| \Sigma_t - Q \| \in O\left(\exp\left( - 2 (\gamma - \eps) t\right)\right)$ for any $\eps > 0$, where
        \begin{equation*}
            \gamma 
            \coloneqq \lambda_{\min}\left(
                \begin{pmatrix}
                A \otimes \id_{d \times d} & \frac{1}{\sqrt{2}} (A b) \otimes Q^{-\frac{1}{2}} \\
                \frac{1}{\sqrt{2}} (b^{\tT} A) \otimes Q^{-\frac{1}{2}} & \frac{1}{2} K(b, b) Q^{-1}
            \end{pmatrix} \right)
            > \frac{1}{K(b, b) + 2 \lambda_{\min}(A) \lambda_{\max}(Q)}.
        \end{equation*}

        \item 
        If $\mu_0 = b = 0$, then \eqref{eq:Stein_gradient_flow_Gaussians} reduces to
        \begin{equation} \label{eq:centered_nonacc_Stein_GF}
            \dot{\Sigma}_t
            = 2 \Sym\big(\Sigma_t A (\id - \Sigma_t Q^{-1})\big),
        \end{equation}
        which we call the {\em Stein variational gradient Lyapunov equation}.
        
        If additionally $\Sigma_0 Q = Q \Sigma_0$ and $A$ is a diagonal matrix with $A V = V A$, where $V$ simultaneously diagonalizes $\Sigma_0$ and $Q$, then $\Sigma_t Q = Q \Sigma_t$ for all $t \ge 0$ and thus we have the closed form solution $\Sigma_t^{-1} = Q^{-1} + e^{-2 t A} \big(\Sigma_0^{-1} - Q^{-1}\big)$.

        \item 
        Now, let $(\rho_t, \Phi_t)_{t > 0}$ be the accelerated Stein gradient flow.
        Then, $\rho_t \sim \NN(\mu_t, \Sigma_t)$, where
        \begin{equation} \label{eq:ASVGD_Gaussians}
            \begin{cases}
                \dot{\mu}_t
                = 2 S_t \Sigma_t A \mu_t + (\mu_t^{\tT} A \mu_t + 1) \nu_t, & \mu_t|_{t = 0} = \mu_0, \\
                \dot{\Sigma}_t
                = \nu_t \mu_t^{\tT} A \Sigma_t 
                + \Sym(\Sigma_t A (2 \Sigma_t S_t + \mu_t \nu_t^{\tT}))
                + 2 \Sym(\Sigma_t A \Sigma_t S_t), & \Sigma_t|_{t = 0} = \Sigma_0, \\
                \dot{\nu}_t
                = - \alpha_t \nu_t - 2 A \Sigma_t S_t \nu_t - A \mu_t \| \nu_t \|_2^2 - Q^{-1} (\mu_t - b), & \nu_0 = 0, \\
                \dot{S}_t
                = - \alpha_t S_t - 2 S_t \nu_t \mu_t^{\tT} A - 4 \Sym(S_t^2 \Sigma_t A) - \frac{1}{2} (Q^{-1} - \Sigma_t^{-1}), & S_0 = 0.
            \end{cases}
        \end{equation}
    \end{enumerate}
\end{theorem}

\begin{proof}
    Analogous to the proofs of \cite[Thms.~3.1,~3.5]{LGBP2024}, see \cref{subsec:Thm3.1_Proof}.
\end{proof}

\begin{remark}
    The first statement from \cref{thm:3.1} also holds without the affine term $+1$ in the definition of the kernel.
\end{remark}

\begin{remark}[Comparison to Wasserstein geometry]
    For the Wasserstein-2 gradient flow of $\KL(\cdot \mid \rho^*)$ we instead have
    \begin{equation} \label{eq:WGF_KL_Gaussians_centered}
        \dot{\Sigma}_t = 2 \id_d - 2 \Sym(\Sigma_t Q^{-1}),    
    \end{equation}
    in the centered case, so compared to \eqref{eq:centered_nonacc_Stein_GF}, every summand has one factor $\Sigma_t$ less.
    By \cite{WL2022}, the accelerated version of \eqref{eq:WGF_KL_Gaussians_centered} is 
    \begin{equation*}
        \begin{cases}
            \dot{\Sigma}_t = 4 \Sym(S_t \Sigma_t), \\
            \dot{S}_t = - \alpha_t S_t - 2 S_t^2 + \frac{1}{2}(\Sigma_t^{-1} - Q^{-1}).
        \end{cases}
    \end{equation*}%
\end{remark}

Lastly, note that \eqref{eq:centered_nonacc_Stein_GF} is more complicated than the GAUL equation \cite[Eq.~(3.6)]{ZOL2024}, which is linear in $\Sigma_t$, while our right-hand side is quadratic.

\subsection{Optimal selection of the kernel and damping parameters for the generalized bilinear kernel}
Under the assumptions of \cref{thm:3.1} 3., the solution of \eqref{eq:centered_nonacc_Stein_GF} fulfills $\| \Sigma_t^{-1} - Q^{-1} \| \le \| e^{-2 A t} \| \| \Sigma_0^{-1} - Q^{-1} \|$, so one could assume that $A \to \infty$ would give the fastest convergence rate.
However, when discretizing in time, one should balance this observation with the numerical stability of the algorithm, which can be approximated by the condition number of the associated linearized ODE system.
This is made precise in the following remark.
\begin{remark}[Convergence rates of explicit Euler discretization of linear dynamical systems] \label{remark:convergence_rate_kappa}
    In both of the special cases discussed in \cref{lemma:optimalA_1D,thm:optimal_damping} below, linearizing the gradient flow \eqref{eq:Stein_gradient_flow_Gaussians} at equilibrium yields a dynamical system $\dot{x}(t) = - B x(t)$ for some non-normal, but diagonalizable matrix $B = V \Lambda V^{-1}$ with positive eigenvalues, where $x(t) = \big(\mu_t, \vec(\Sigma_t)\big)$ or $x(t) = \big(\mu_t, \vec(\Sigma_t), \nu_t, \vec(S_t)\big)$.
    In continuous time, we have the estimate $\| x(t) \|^2 \le e^{-2 s(B) t} \| x(0) \|^2$, where $s(B) \coloneqq \min_{\lambda \in \sigma(B)} \Re(\lambda)$ is the \emph{spectral abscissa} of $-B$.
    
    The explicit Euler discretization in time with step size $h > 0$ is $x^k = (I - h B)^k x^0$.
    We have
    \begin{equation*}
        \| (I - h B)^k \|_2
        \le \| V \|_2 \| V^{-1} \|_2 \| (I - h \Lambda)^k \|_2
        \le \rho(B)^k \kappa(V),  
    \end{equation*}
    where $\rho(B) \coloneqq \max_{\lambda \in \sigma(B)} | 1 - h \lambda |$ is the \emph{spectral radius} of $B$ and $\kappa(V) \coloneqq \| V \|_2 \| V^{-1} \|_2$ is the \emph{condition number} of $V$.

    While the factor $\kappa(V)$ can be large, the \emph{asymptotic} convergence rate for $k \to \infty$ is governed by the spectral radius.
    If we choose a sufficiently small step size, parametrized as $h = \frac{\alpha}{\lambda_{\max}(B)}$ for some $\alpha \in (0, 1]$, then $\rho(B) = 1 - \alpha\left(\kappa(B)\right)^{-1}$, where $\kappa(B) \coloneqq \frac{\lambda_{\max}(B)}{\lambda_{\min}(B)}$.
    Hence, in order to get the best \emph{asymptotic} convergence rate, we minimize $\rho^*$ and thus try to get $\kappa(B) \in [1, \infty)$ to be as small as possible. 
\end{remark}

In simple cases, we can explicitly compute an optimal choice of $A$ for the linearized system.
As recounted in 

\begin{lemma}[Optimal $A$ for linearized SVGD with bilinear kernel] \label{lemma:optimalA_1D}
    \begin{enumerate}
        \item 
        For $d = 1$,
        the scaling $A$ minimizing the condition number of the system matrix of the linearized version of the Stein gradient flow in the Gaussian submanifold, \eqref{eq:Stein_gradient_flow_Gaussians},
        is given by
        \begin{equation*}
            A = \frac{1}{2 Q + b^2}
        \end{equation*}
        and the associated optimal step size is $h^* = Q > 0$.
        If furthermore we have $b = 0$, then we obtain the continuous time convergence rate $\| (\mu_t, \Sigma_t) \| \le e^{- \frac{1}{Q} t} \| (\mu_0, \Sigma_0)\|$.
    
        \item 
        If $A$ and $Q$ are simultaneously diagonalizable and $\mu_0 = b = 0$, then the optimal $A$ (in the above sense) is $A = \frac{1}{2} Q^{-1}$.
    \end{enumerate}
\end{lemma}

\begin{proof}
    See \cref{subsec:ProofLemma5.1}.
\end{proof}

As our main contribution, we now prove which choice of constant damping parameter yields the best asymptotic convergence rate.
In contrast to the non-accelerated SVGD, the best $A$ does not depend on $Q$.

\begin{theorem}[Asymptotically optimal damping constant for centered linearized ASVGD] \label{thm:optimal_damping}
        If $A$ and $Q$ commute and $\mu_0 = b = 0$, then the jointly optimal (in the sense of \cref{remark:convergence_rate_kappa}) damping parameter $\alpha_t \equiv \alpha$ and matrix $A$ yielding the optimal asymptotic continuous time rate are given by
        \begin{equation*}
            \alpha
            = \sqrt{8 \lambda_{\min}(A)},
        \end{equation*}
        and $A = \theta \id_{d \times d}$ for some $\theta > 0$.
        The continuous-time convergence rate is $\| (\Sigma_t, S_t) \| \le e^{-\sqrt{8 A} t} \| (\Sigma_0, S_0) \|$ and the optimal discrete-time convergence rate is
        \begin{equation*}
            \| (\Sigma_k, S_k) \|_2 \le \rho^k \kappa(V) \| (\Sigma_0, S_0) \|_2,    
        \end{equation*}
        where $A = V D_A V^{\tT}$ is the orthogonal diagonalization of $A$, and
        \begin{equation*}
            \rho
            \coloneqq \frac{\sqrt{\frac{1}{2}\left(\kappa(Q) + \frac{1}{\kappa(Q)}\right)} - 1}{\sqrt{\frac{1}{2}\left(\kappa(Q) + \frac{1}{\kappa(Q)}\right)} + 1}
            < \frac{\sqrt{\kappa(Q)} - 1}{\sqrt{\kappa(Q)} + 1},
        \end{equation*}
        when using the optimal step size
        \begin{equation*}
            h^*
            \coloneqq \frac{2}{\sqrt{\max\limits_{1 \le i, j \le d} \mu_{i, j}} + \sqrt{2 \lambda_{\min}(A)}}.
        \end{equation*}
\end{theorem}
\begin{proof}
    See \cref{subsec:proof_optimal_damping}.
\end{proof}

\section{Numerical results} \label{sec:numerics}
\subsection{Toy problems}

We choose $\Omega = \R^2$ and compare our algorithm, ASVGD\footnote{The Python code for reproducing these experiments is available online: \url{https://github.com/ViktorAJStein/Accelerated_Stein_Variational_Gradient_Flows}.}, with ULA, SVGD, and MALA (Metropolis-adjusted Langevin algorithm), which augments ULA by a Metropolis-Hastings acceptance step to remove its asymptotic bias \cite{B1994}.
Since our method is of second order in time, we also compare it to underdamped Langevin dynamics (ULD) with unit mass and unit friction.

Note that there are other, unbiased alternatives for sampling, not even requiring differentiability of the potential, like the proximal sampler \cite{LST2021}.
However, the proximal sampler is not a sequential particle method.

\paragraph{Setup.}
In all our plots, the initial particles are blue circles, the red squares are the final particles, and the colored lines represent the trajectories between them.
The black lines represent the level lines of the target density $f$.
In all experiments, we choose $N = 500$ particles and perform 1000 steps with step size, bandwidth, and regularization parameter all equal to $\tau = \sigma = \eps = 0.1$.
If not stated otherwise, we always use the gradient restart and the speed restart from \cite{OC2015,SBC2016} with damping parameter $\alpha_k = \frac{k}{k + 3}$.
Throughout, we observe that ASVGD converges faster than SVGD.

\paragraph{Bilinear kernel.}
For this kernel, we only consider Gaussian targets, since from \cref{thm:3.1} we know that when initializing with a Gaussian distribution, the distribution will remain Gaussian.
On the other hand, we observe that, as supported by the proof of \cref{thm:3.1}, for arbitrary target densities, the overall shape of the particles is always approximately a linear transformation of the initial particles, which is in line with our theoretical analysis of the (co)tangent space.

\figref{fig:bilinear_KL} shows that ASVGD recovers Gaussian targets much more rapidly than the other sampling algorithms.
\begin{figure}[H]
    \centering
    \includegraphics[width=.25\textwidth]{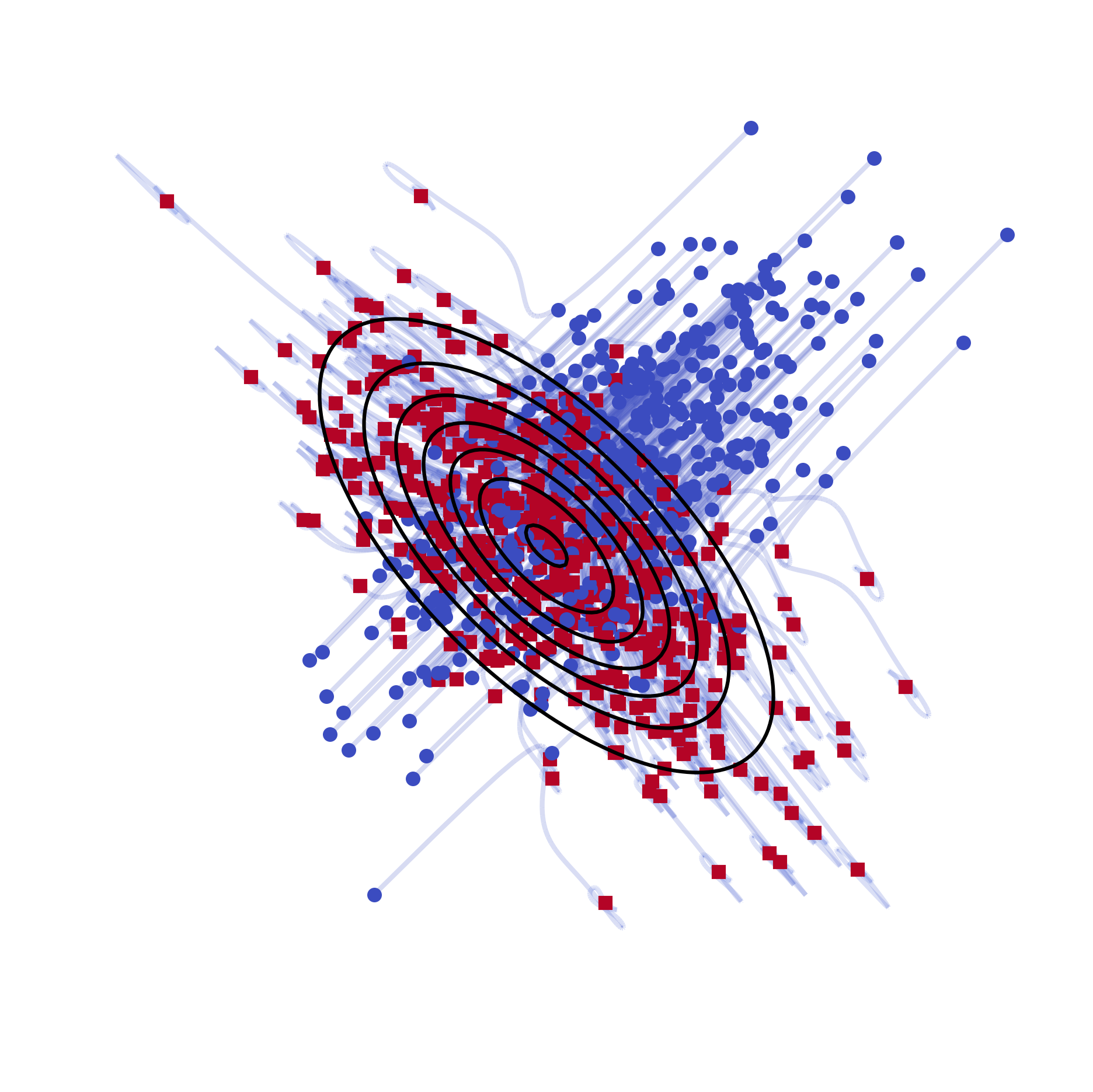}%
    \includegraphics[width=.25\textwidth]{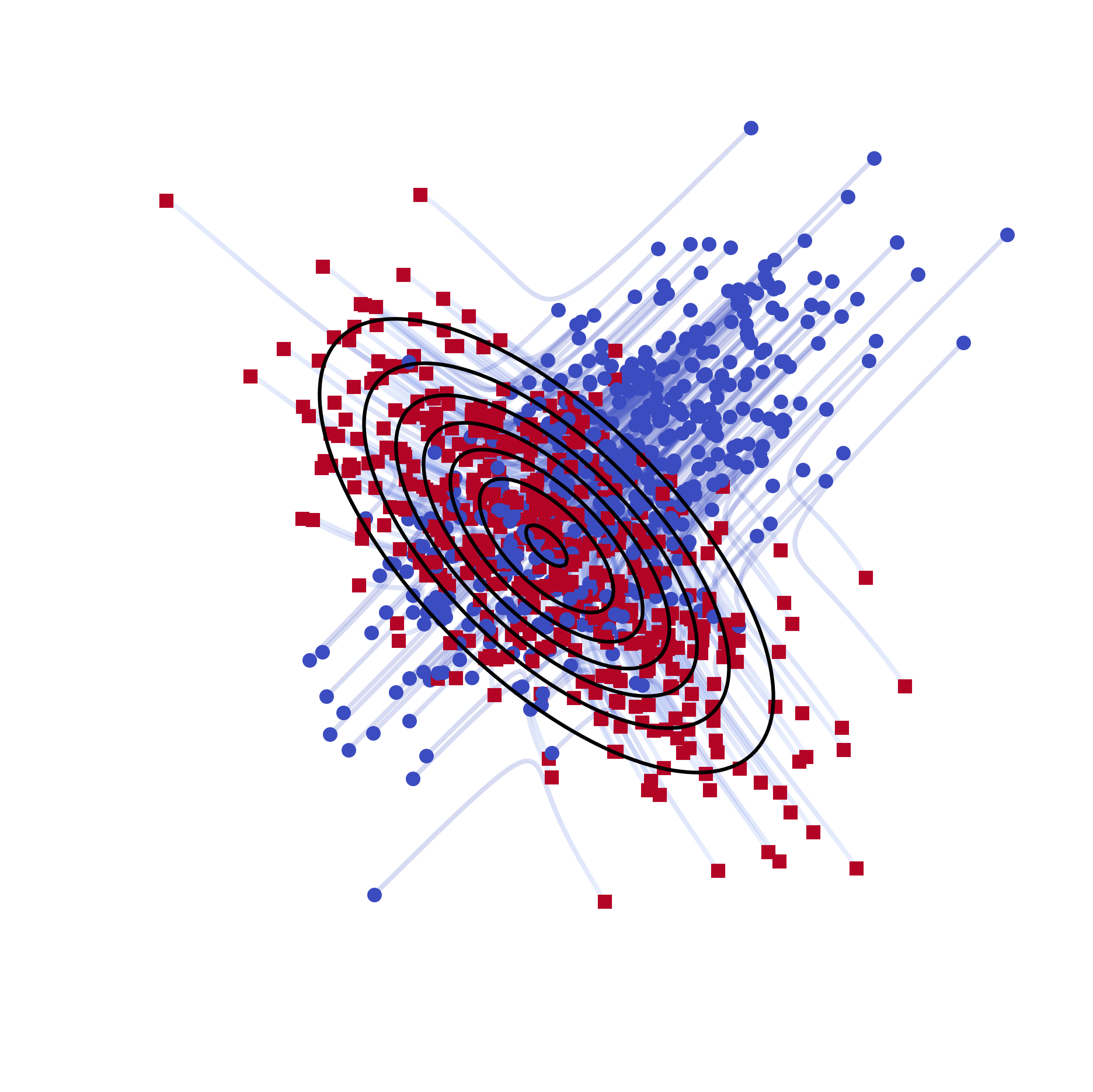}%
    \includegraphics[width=.25\textwidth]{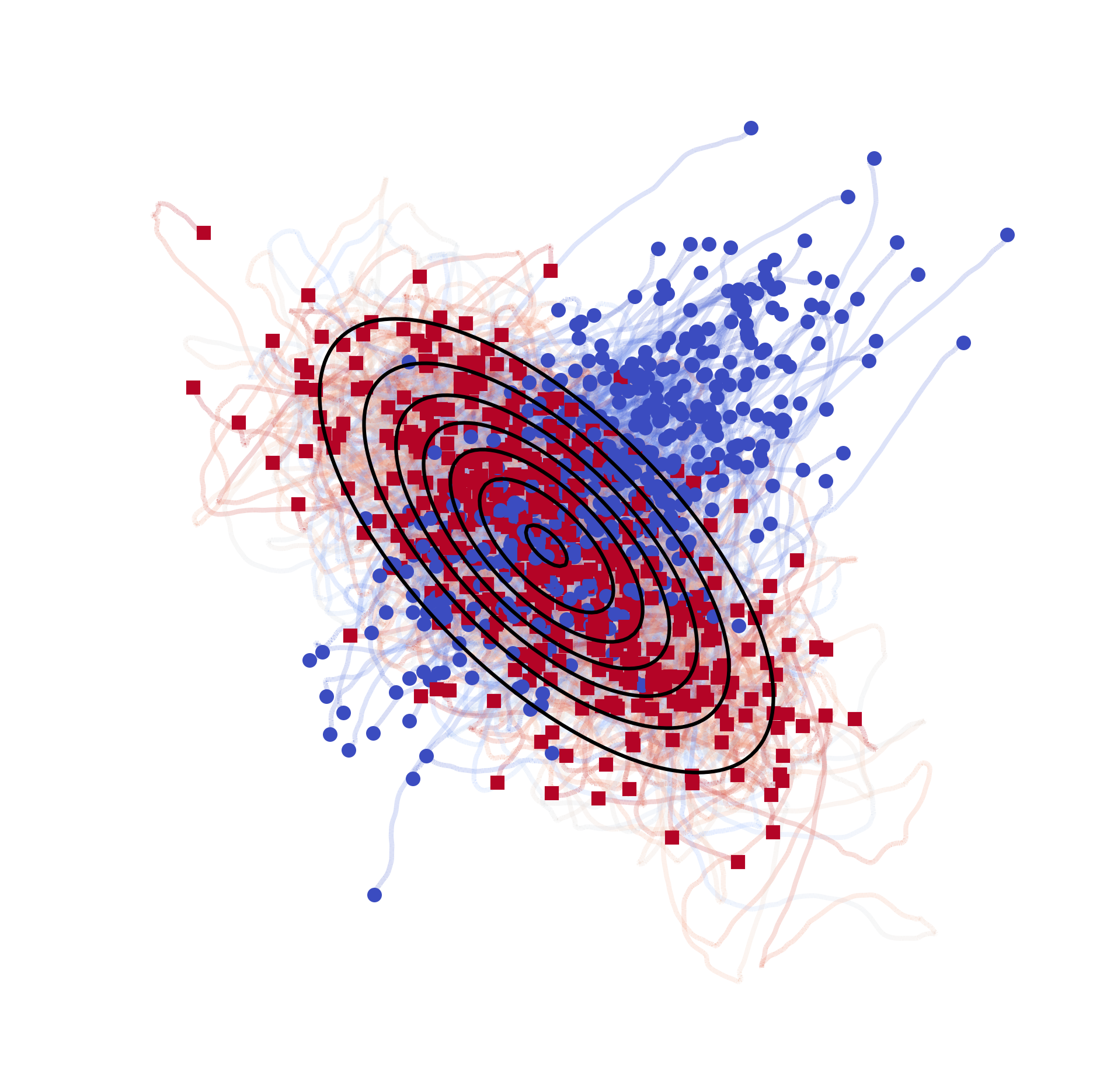}%
    \includegraphics[width=.25\textwidth]{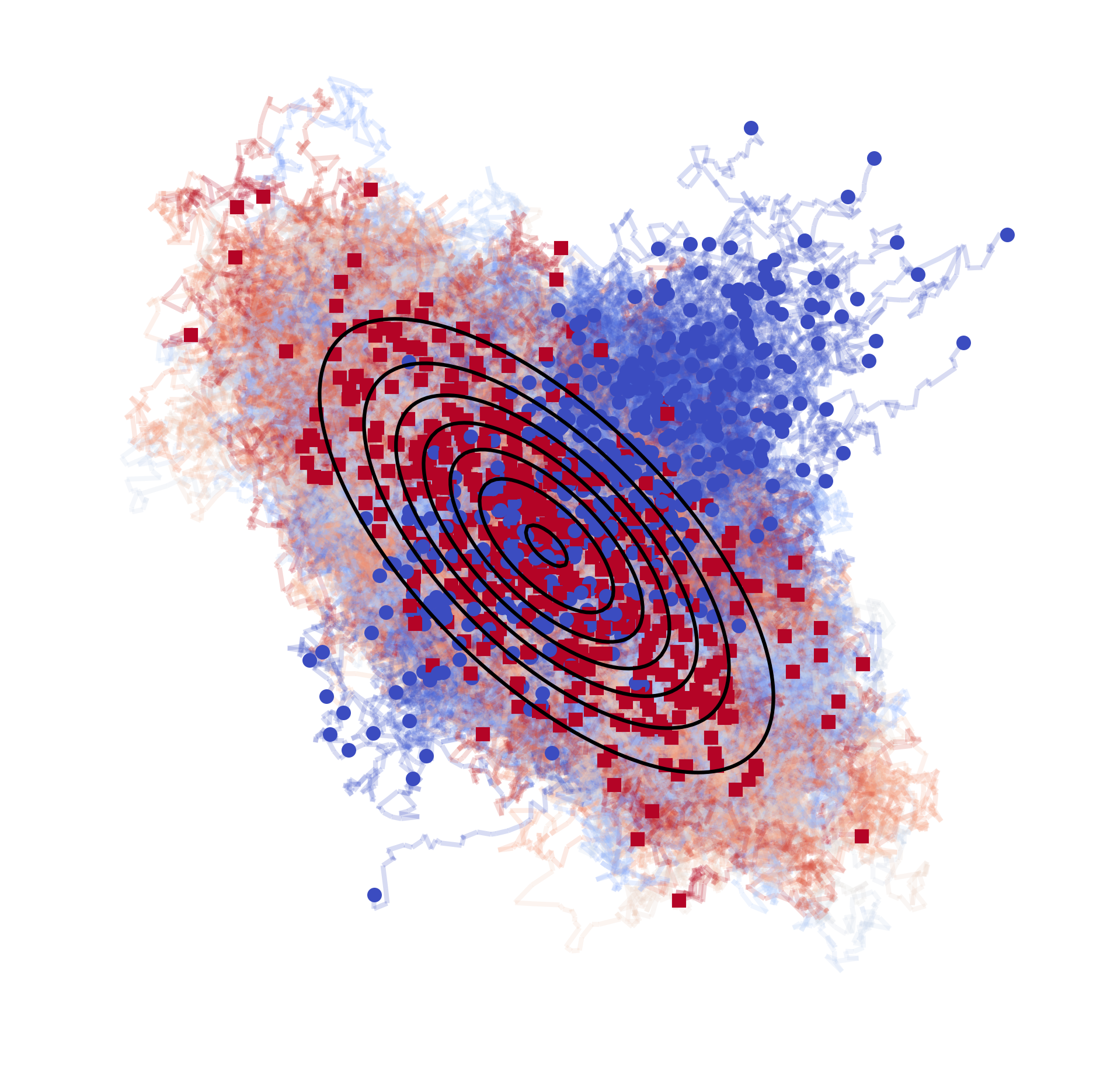}
    \caption{Particle trajectories of ASVGD, SVGD, with the generalized bilinear kernel, MALA, and ULD (from left to right). The potential is $f(x) = \frac{1}{2} x^{\tT} Q x$, with $Q = [[3, -2], [-2, 3]]$ and we initialize the particles from a Gaussian distribution with mean $[1, 1]^{\tT}$ and covariance $[[3, 2], [2, 3]]$.}
    \label{fig:Gaussian}
\end{figure}%
\begin{figure}[H]
    \centering
    \includegraphics[width=.55\textwidth]{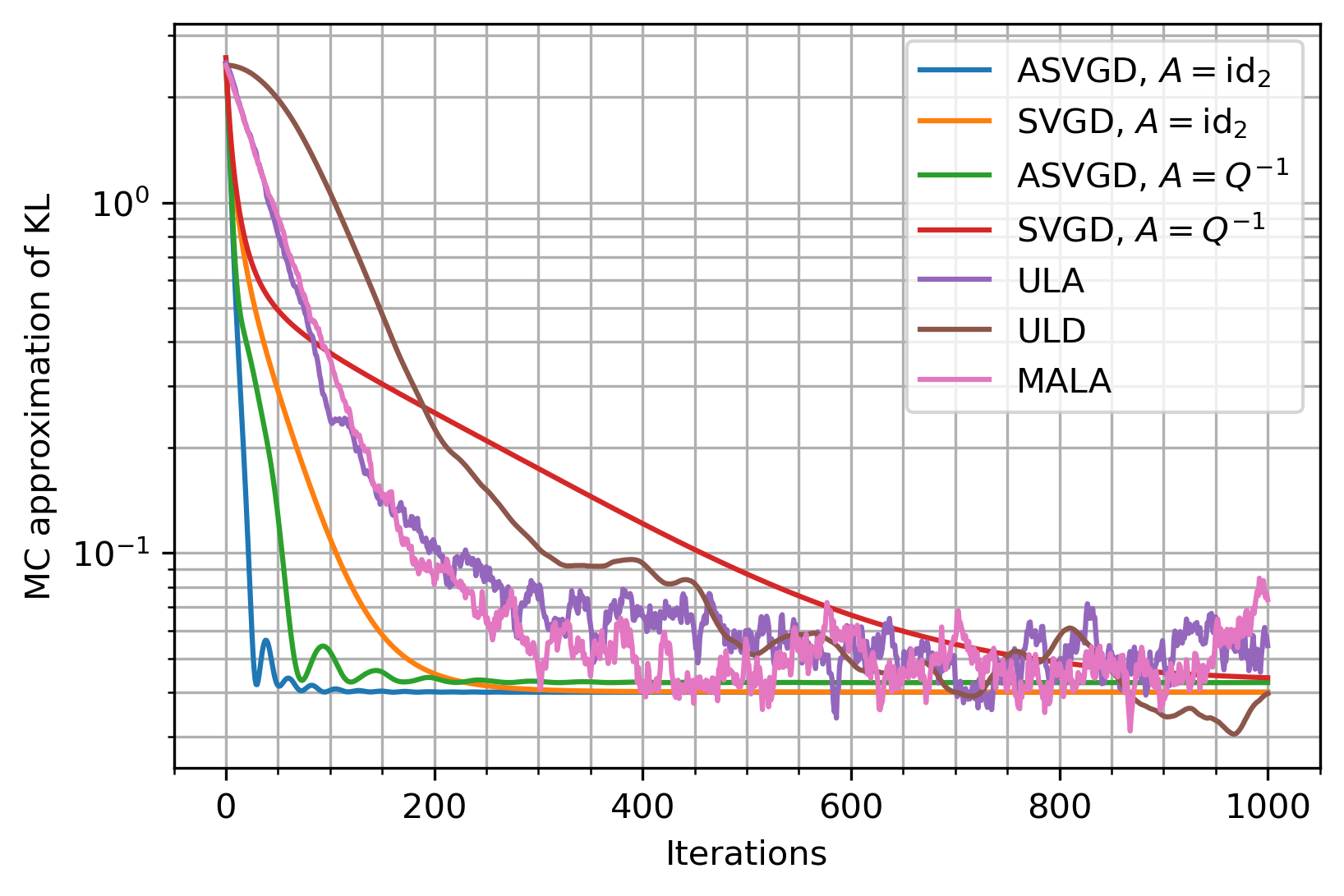}
    \caption{Monte-Carlo-estimated KL divergence for two different choice of $A$ for the particle evolutions from \figref{fig:Gaussian}. We see that for ASVGD $A = \id_2$ performs better than aligning $A$ with the target.} \label{fig:bilinear_KL}
\end{figure}%

\paragraph{Gaussian kernel.}
Since SVGD with the Gaussian kernel can match any target density, we consider the following differentiable target densities $f$, where we set $L \coloneqq \Lip(f) = \| \nabla f \|_{\infty}$ and $m \coloneqq \lambda_{\min}(\nabla^2 f)$.
The condition number of the potential $f$ is $\kappa \coloneqq \frac{L}{m}$.
For two of the considered targets, we collect these parameters in \cref{tab:targets}.
\begin{table}[H]
    \centering
    \begin{tabular}{c|cccc}
        target name             & potential $f(x)$              & $L$       & $m$              & $\kappa$ \\ \hline
        non-Lipschitz           & $\frac{1}{4} (x_1^4 + x_2^4)$ & $\infty$  & $0$              & $\infty$ \\
        anisotropic Gaussian    & $x^{\tT} \diag(10, 0.05) x$   & $\approx 20$    & $\frac{1}{20}$   & $\approx 401$
    \end{tabular}
    \caption{Smoothness and convexity constants and condition number for two target distributions, whose potentials have a simple closed form.}
    \label{tab:targets}
\end{table}

We observe that SVGD and ASVGD have lower variance.
Another advantage of our algorithm being deterministic is that particles arrange neatly along the level lines of the density, especially for small bandwidths, as opposed to the randomness of ULA, ULD, and MALA.
This \textit{structured behavior} of ASVGD is favorable since it respects symmetries of the target.
This is most pronounced in the first example in \figref{fig:nonLip}.
Furthermore, a smaller Gaussian kernel width makes the particle cluster closer to the modes, making the bandwidth an interpretable parameter.
We also observe that the potential not being Lipschitz continuous leads to a high rejection rate, which slows down MALA considerably.
As observed for the anisotropic Gaussian target, ULA (and to some extend MALA and ULD, whose particle overshoot at the ends) need small step sizes to avoid degenerating, while ASVGD, as SVGD, works with large step sizes.
This is an advantage of the particle-wise restart in ASVGD, in contract ULA needs large steps in one coordinate direction, which are unsuitable for the other direction.
Similarly, for the non-Lipschitz target, the ULA particles diverge almost immediately when using larger step sizes.

Lastly, the momentum of ASVGD can improve exploration: only the second-order in time methods, ASVGD and ULD, explore both modes and do not just concentrate on one banana of the target distribution. Note that the potential of the double bananas target is non-convex and non-smooth.
Even though ULD outperforms ASVGD in terms of the approximated KL along the flow, it matches the geometry of the target better (see anisotropic Gaussian example) and has less variance (non-Lipschitz potential example).

\begin{figure}[H]
    \begin{subfigure}{\textwidth}
    \includegraphics[width=0.25\linewidth]{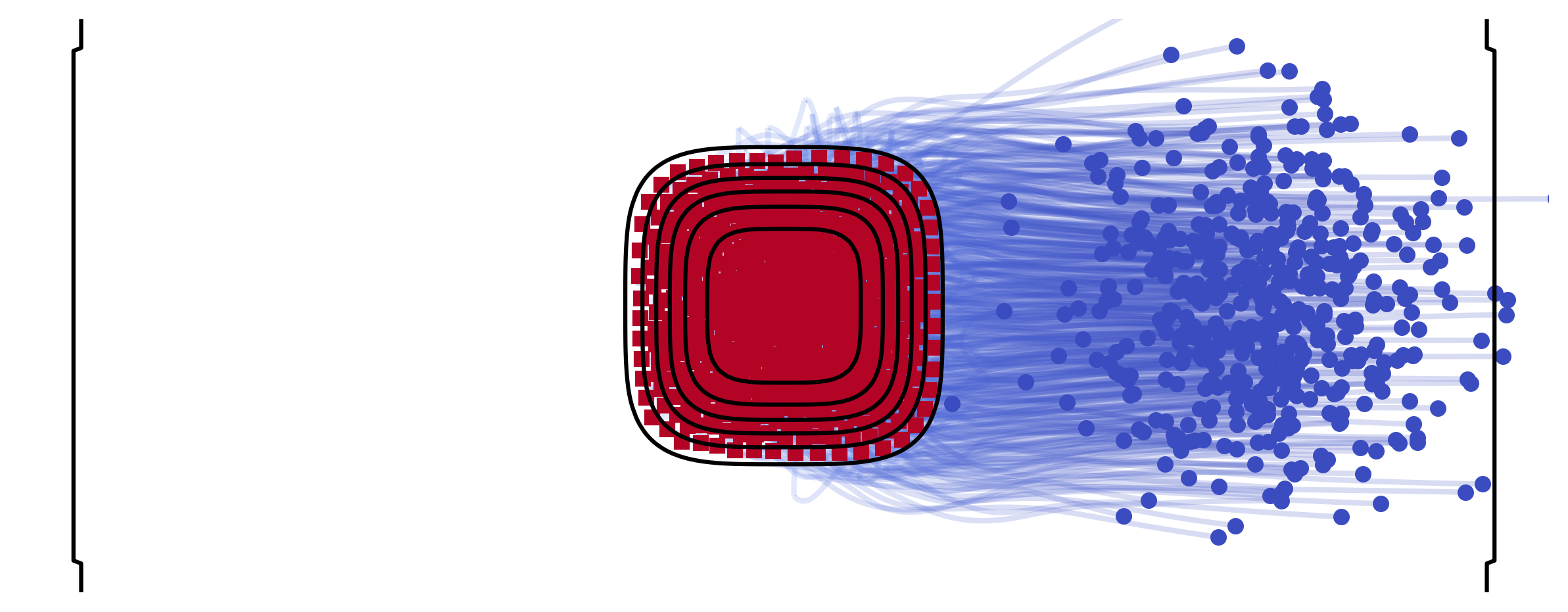}%
    \includegraphics[width=0.25\linewidth]{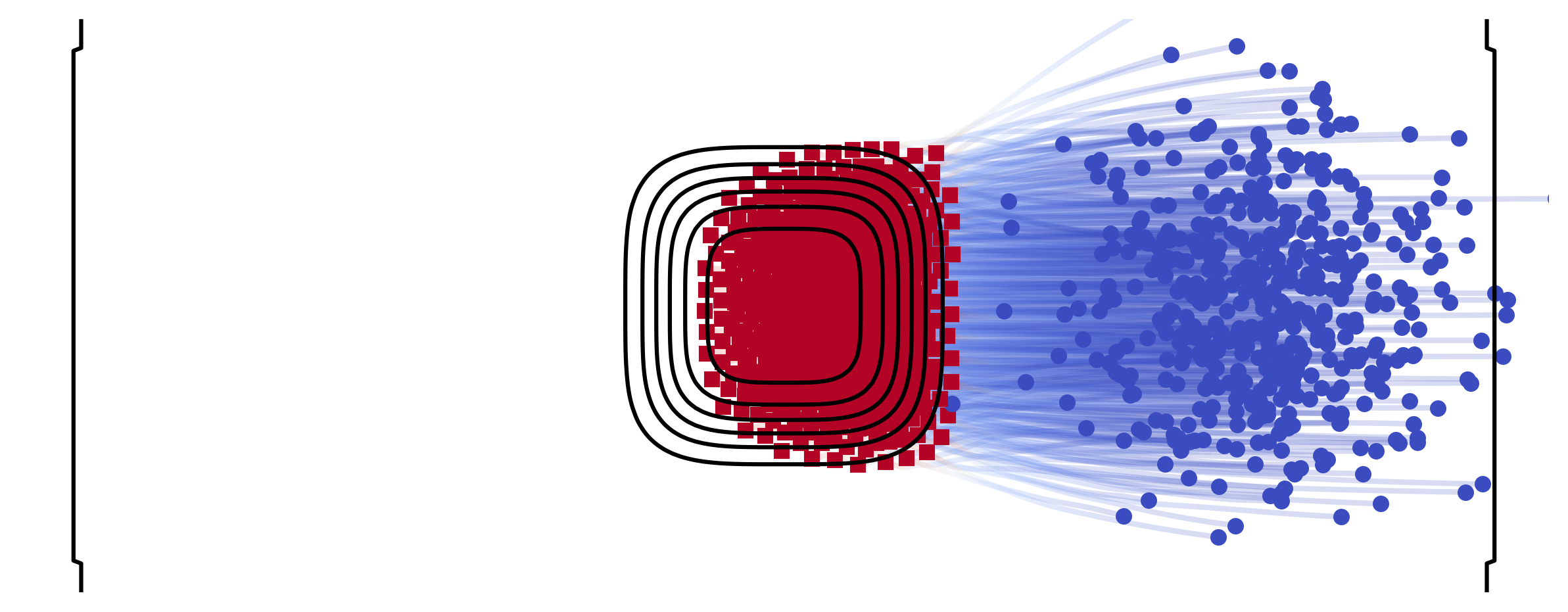}%
    \includegraphics[width=0.25\linewidth]{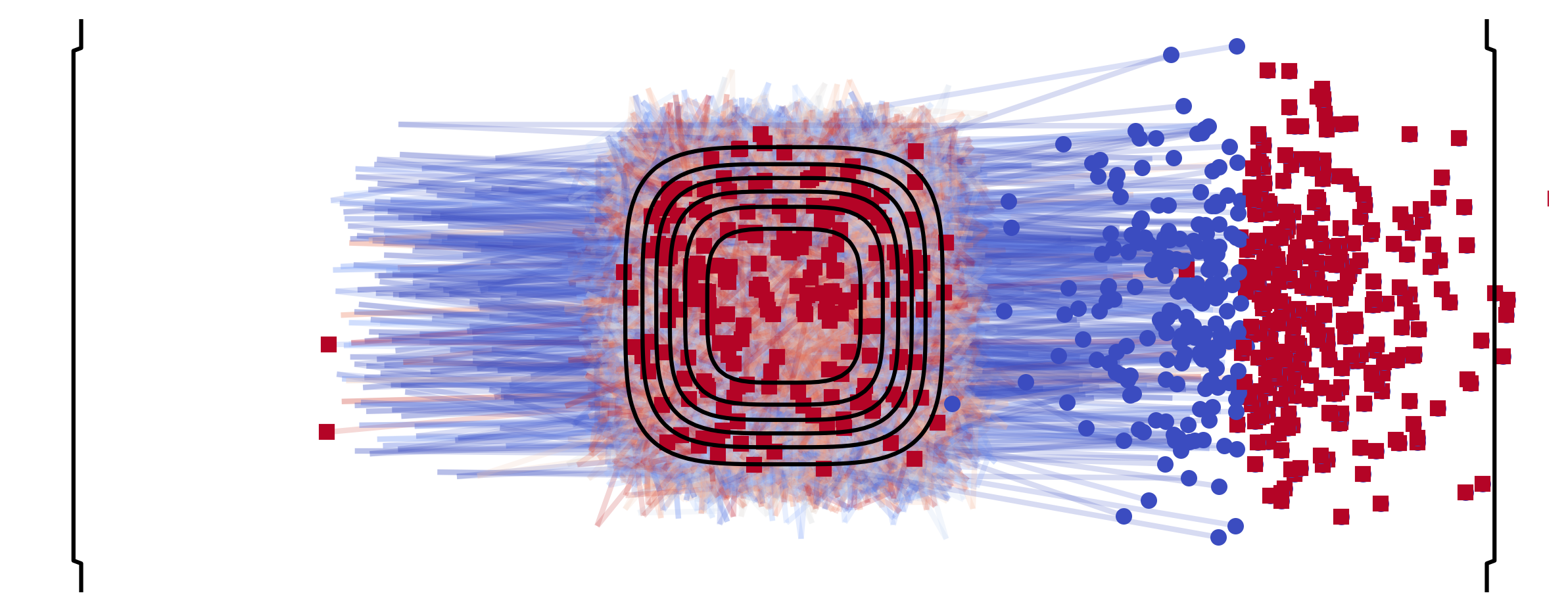}%
    \includegraphics[width=0.25\linewidth]{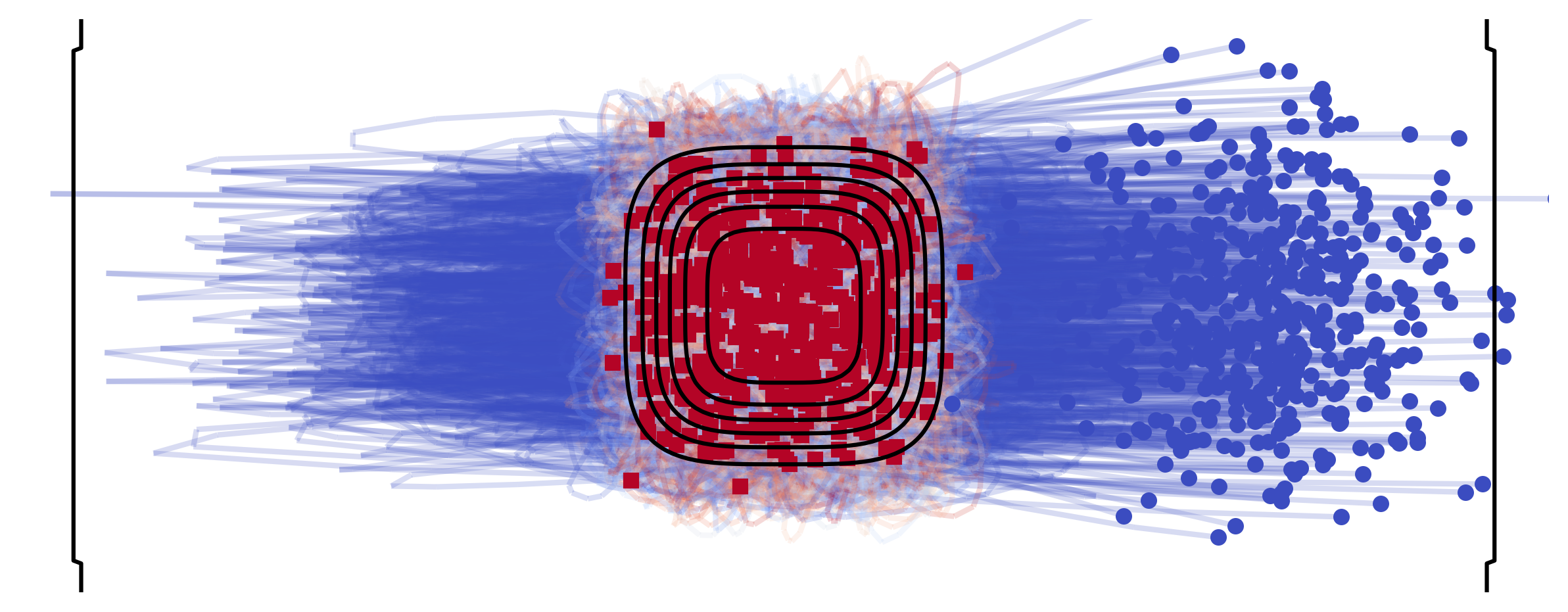}
    \includegraphics[width=0.25\linewidth]{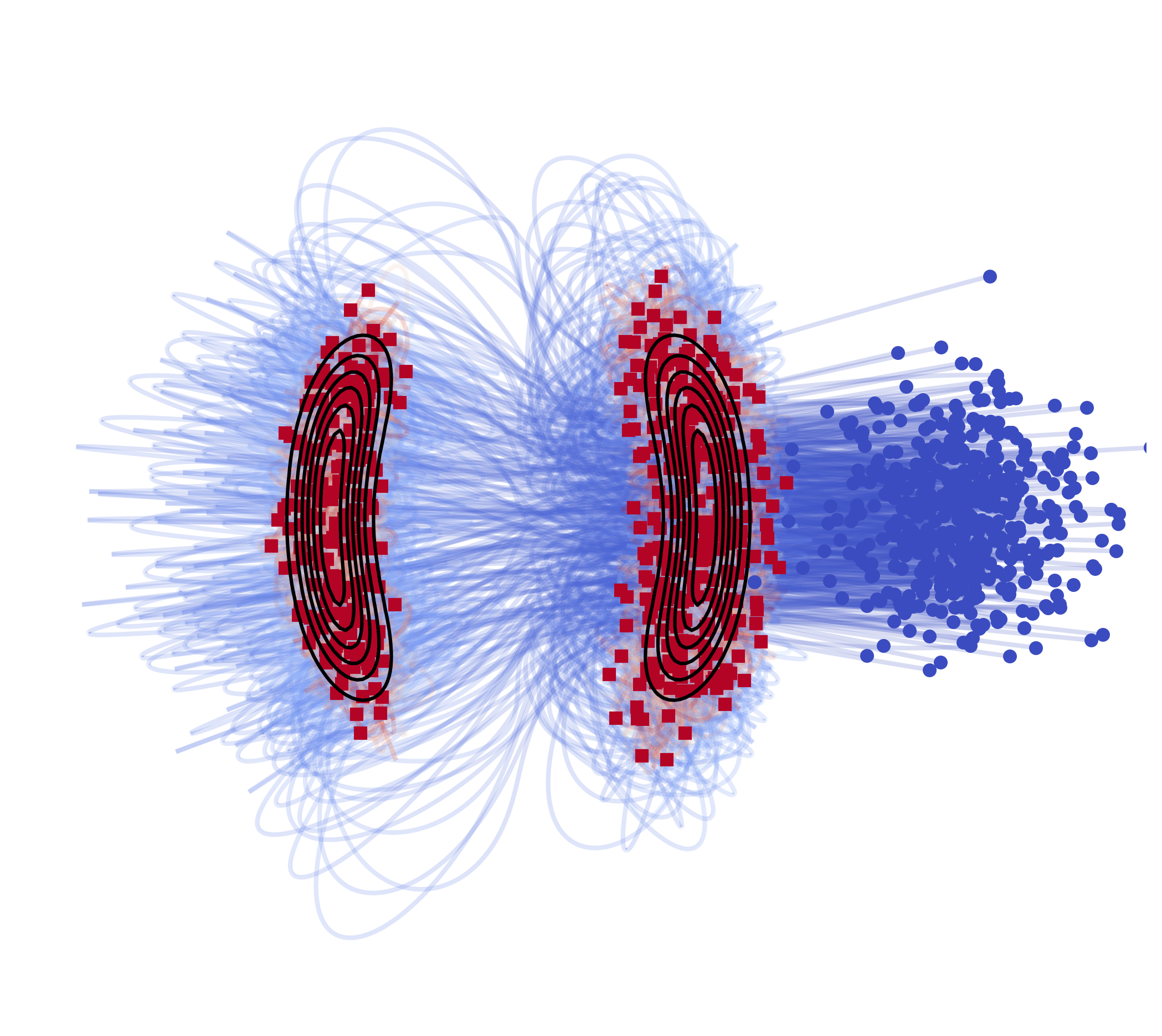}%
    \includegraphics[width=0.25\linewidth]{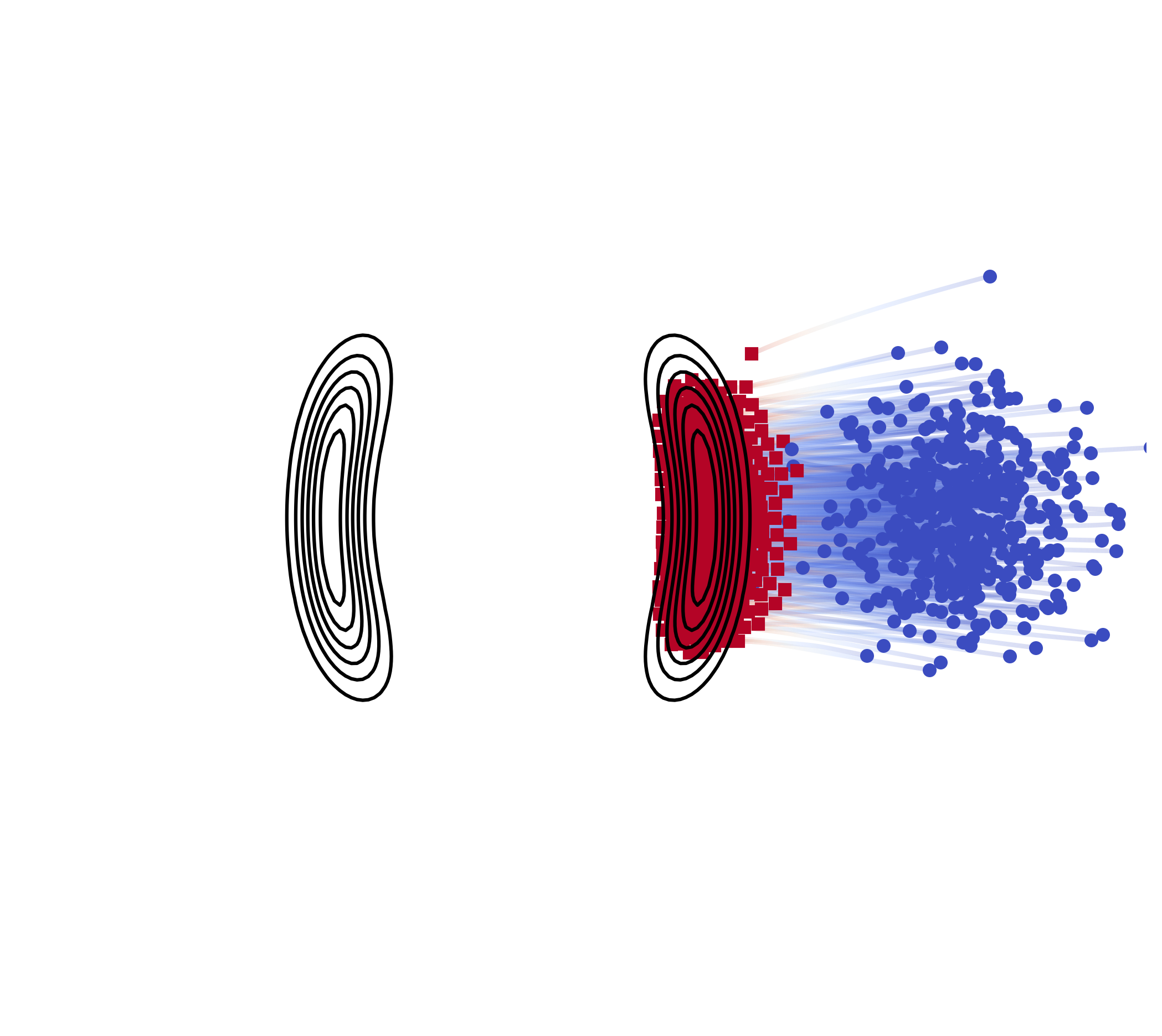}%
    \includegraphics[width=0.25\linewidth]{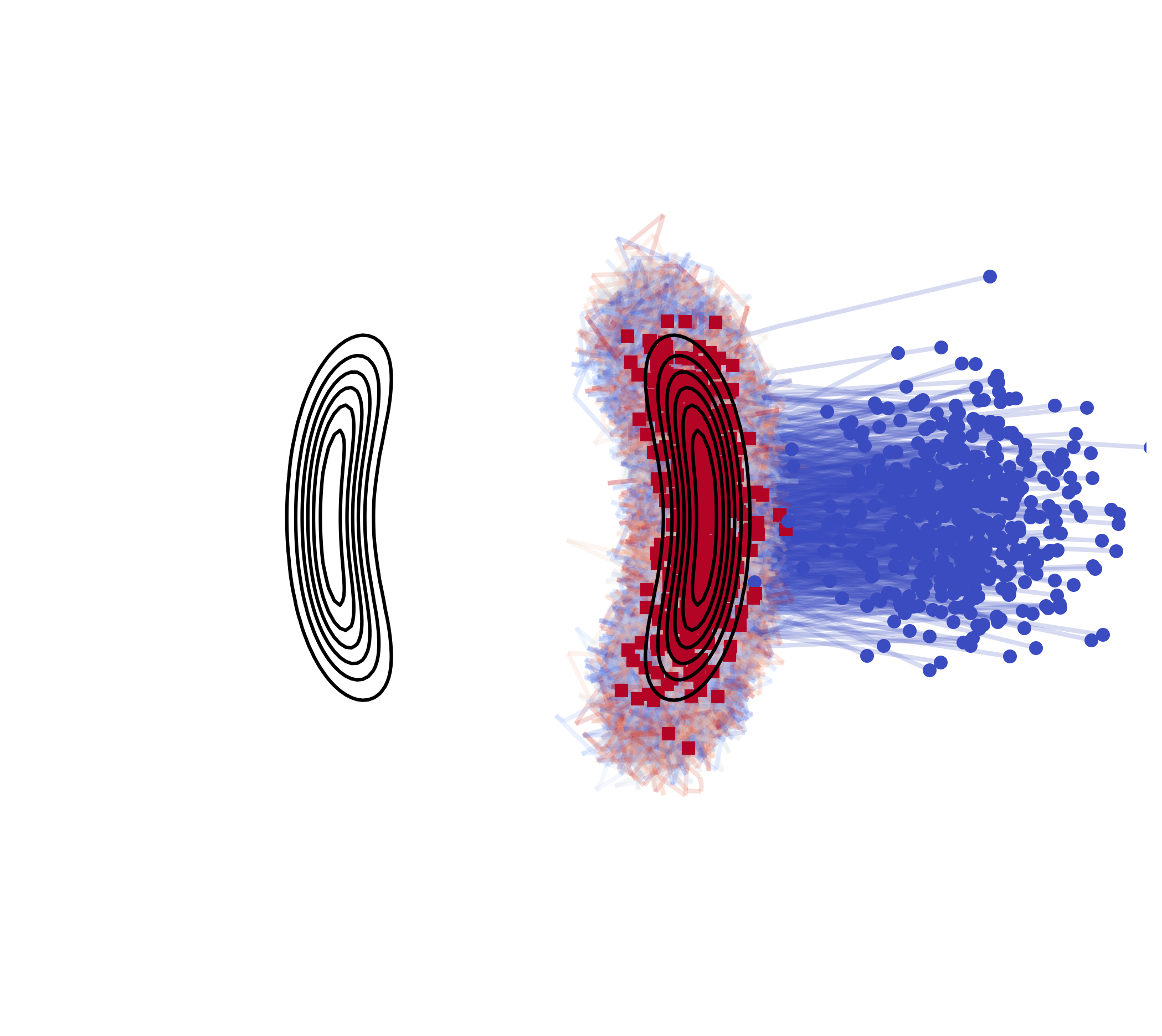}%
    \includegraphics[width=0.25\linewidth]{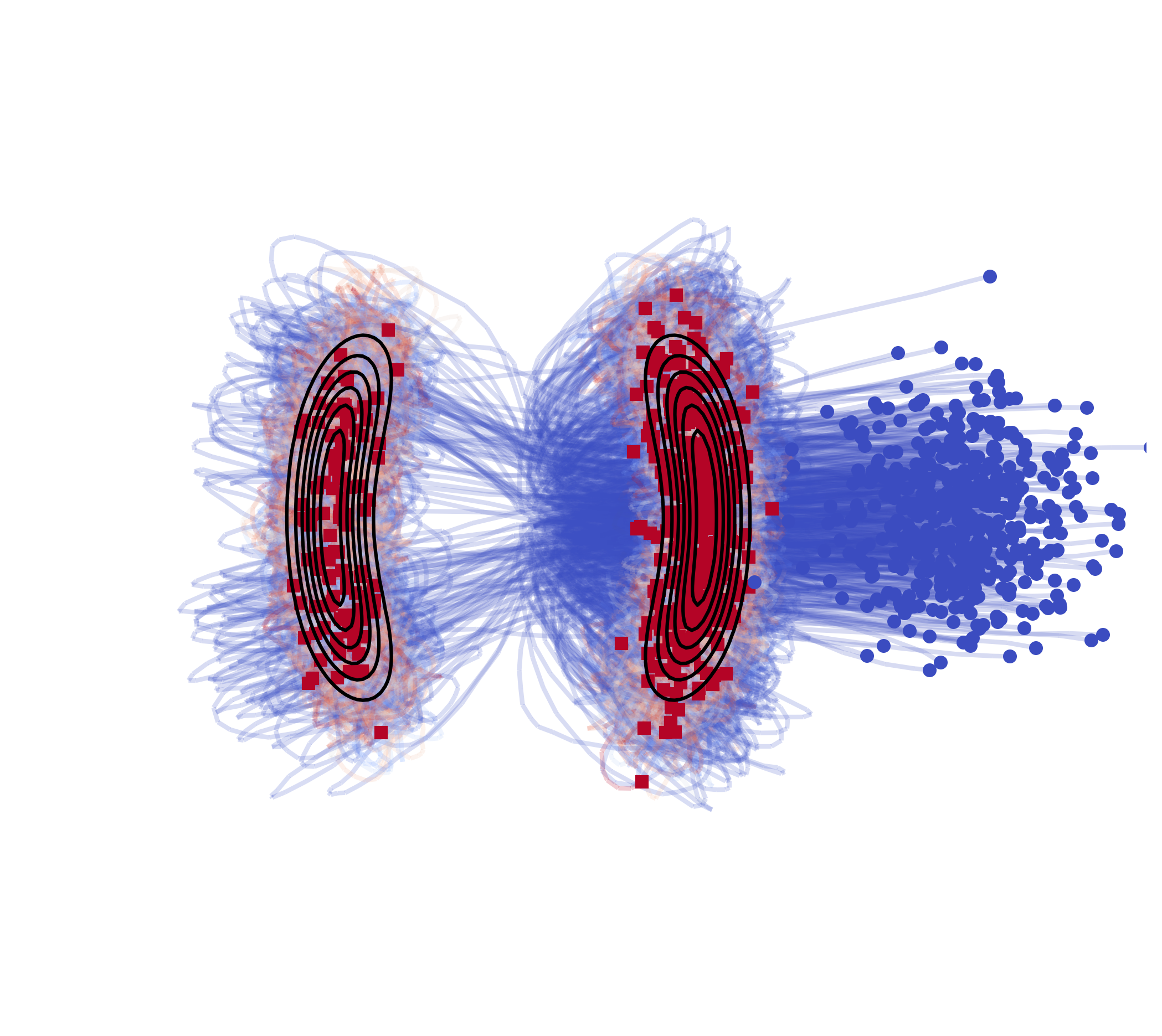}
    \includegraphics[width=0.25\linewidth]{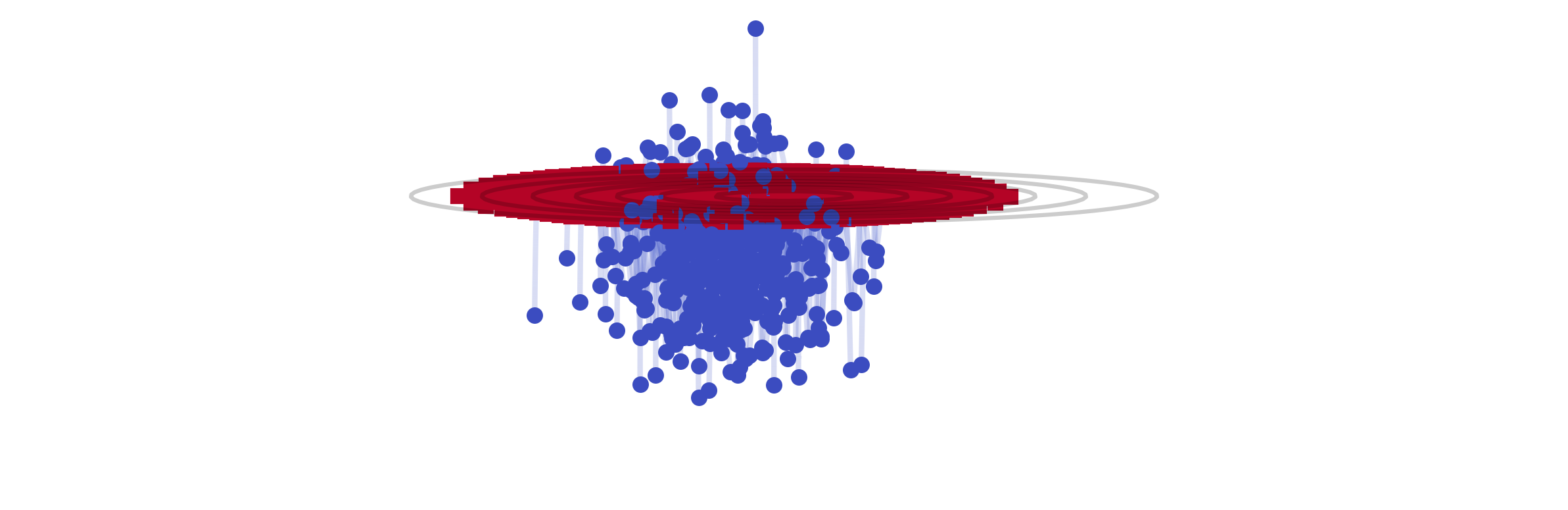}%
    \includegraphics[width=0.25\linewidth]{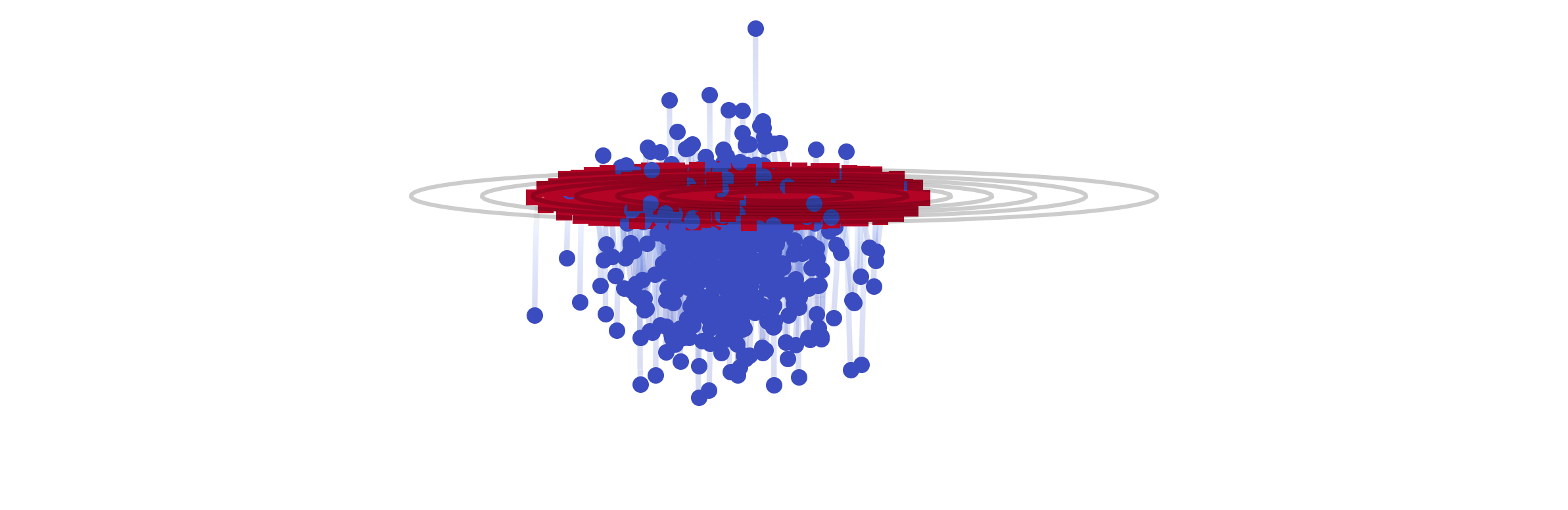}%
    \includegraphics[width=0.25\linewidth]{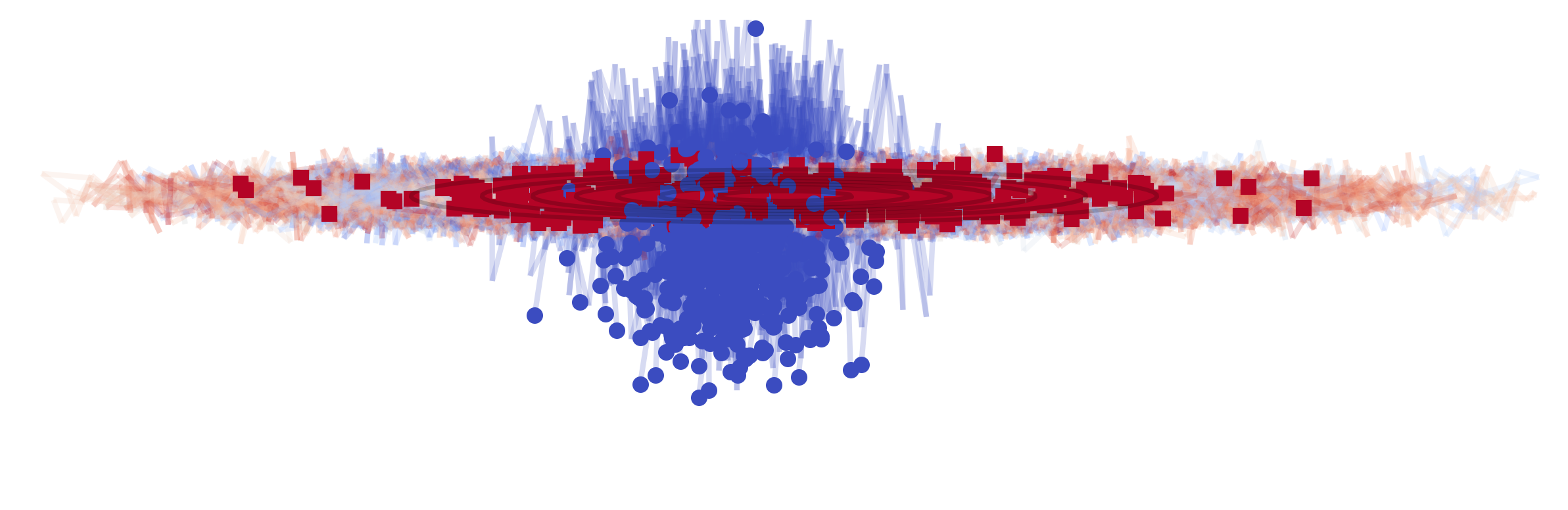}%
    \includegraphics[width=0.25\linewidth]{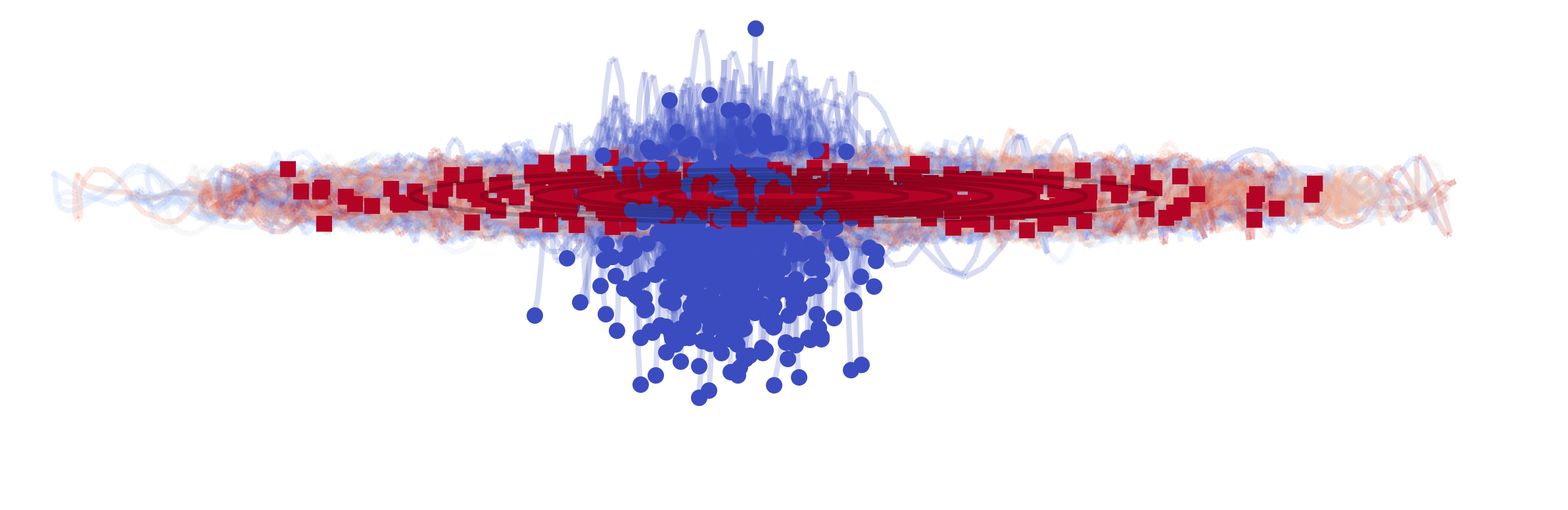}
    \subcaption{Particle trajectories of ASVGD, SVGD, with the Gaussian kernel, MALA, and ULD (from left to right) for the convex, non-Lipschitz potential $f(x, y) = \frac{1}{4}(x^4 + y^4)$ (top), the double bananas target from \cite{WL2022} (middle) and an anisotropic Gaussian target with mean $[1, 1]^{\tT}$ and covariance matrix $Q = \diag(10, 0.05)$ (bottom).}
    \end{subfigure}
    \begin{subfigure}{\textwidth}
    \includegraphics[width=0.33\linewidth]{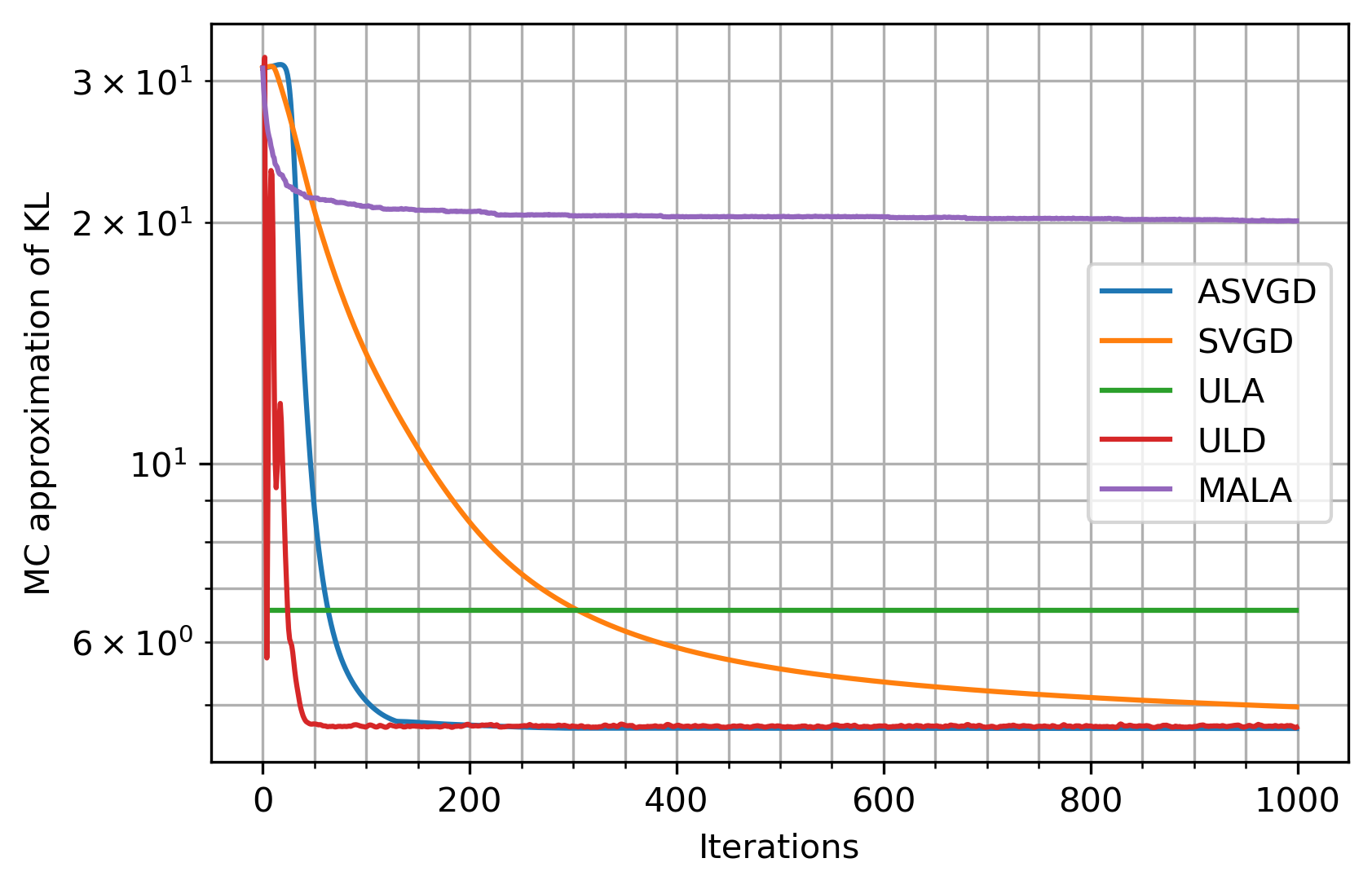}%
    \includegraphics[width=0.33\linewidth]{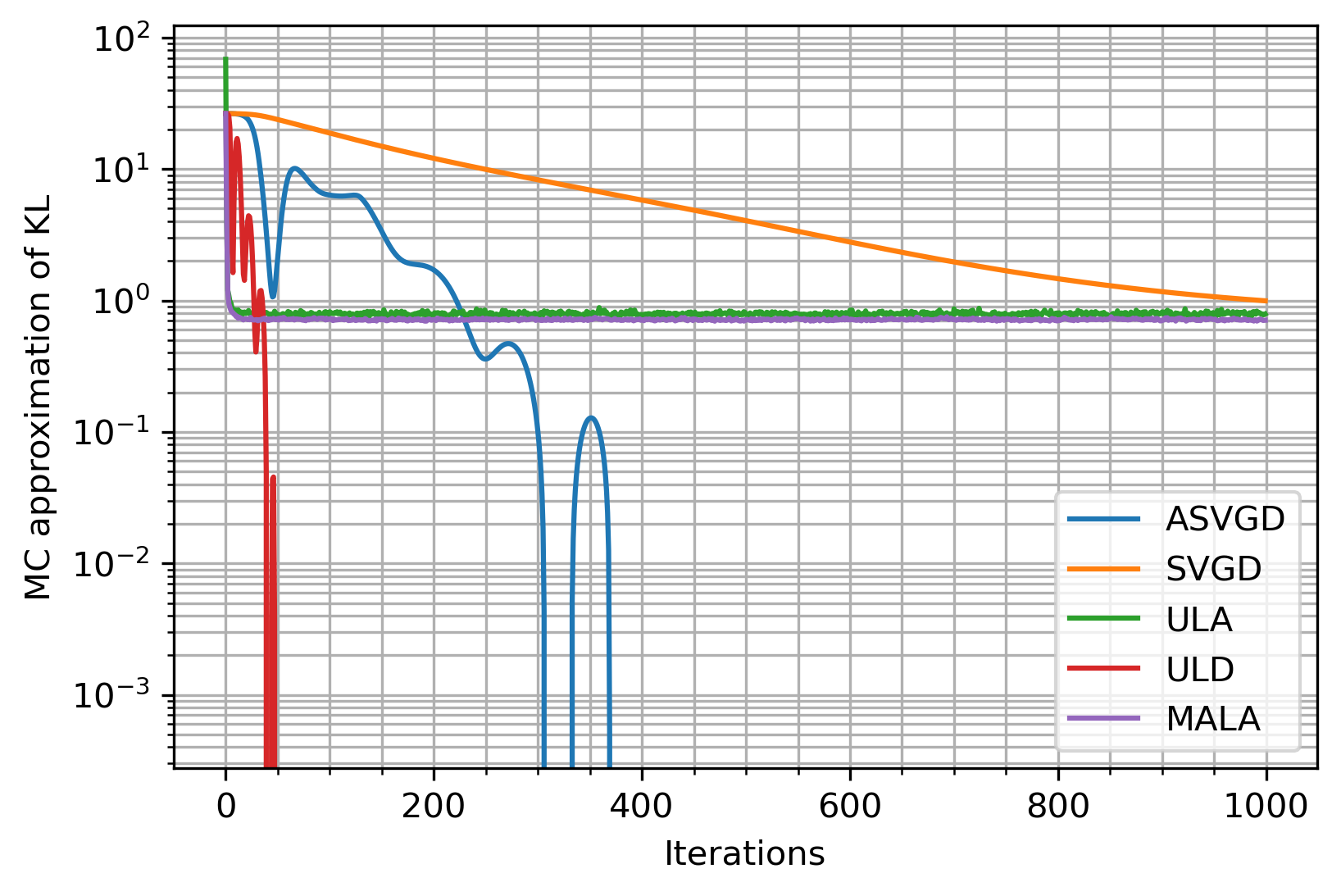}%
    \includegraphics[width=0.33\linewidth]{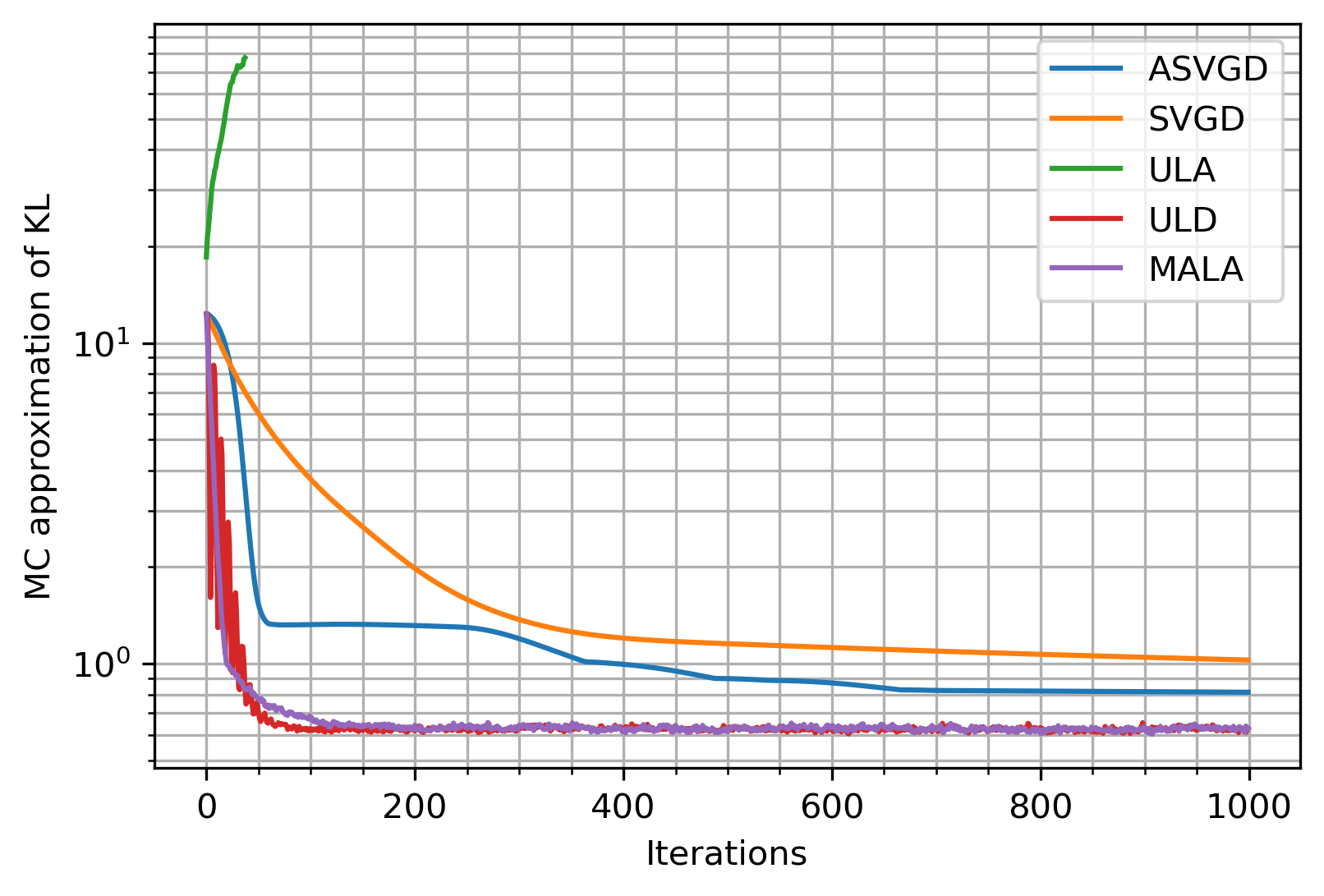}
    \subcaption{Monte-Carlo estimation of the KL divergence to the target for three targets above.} 
    \end{subfigure}
    \caption{Comparing ASVGD to other sampling algorithms.
    For the double bananas target, we choose a constant high damping $\beta = 0.985$.
    We draw the initial particles from unit normal distributions with means $[0, 5]^{\tT}$, $[0, 7]^{\tT}$, and $[0, 0]^{\tT}$, respectively.}
    \label{fig:nonLip}
\end{figure}%

\subsection{Bayesian Neural Networks}

We use the same setting as in \cite{LW2016}, which in turn relies heavily on \cite{HA2015}.
Given $N \in \N$ data points $D \coloneqq ((x_i, y_i))_{i = 1}^{N} \subset (\R^k \times \R)^N$ that are assumed to be generated by the forward model $y_i = f(x_i) + \eps_i$, where $\eps_i \sim \NN(0, \gamma^{-1})$ for some noise parameter $\gamma > 0$, which is drawn from a $\Gamma(1, 0.1)$ distribution, we want to find weights $W$ such that $f$ is well approximated by a neural network with weights $W$.
Let $X \coloneqq (x_i)_{i = 1}^{N} \in \R^{k \times N}$ and $y \coloneqq (y_i)_{i = 1}^{N} \in \R^N$.
By Bayes' rule, the posterior distribution of the weights given the data is
\begin{equation*}
    p(W, \gamma, \lambda \mid D)
    = \frac{p(W \mid \lambda) \prod_{i = 1}^{N} p(y_i \mid x_i, W, \gamma) p(\lambda) p(\gamma)}{p(y \mid X)},
\end{equation*}
where $\lambda > 0$ is a precision constant, drawn from $\Gamma(1, 0.1)$, and $p(y \mid X)$ is the intractable normalization constant.
We assume that the prior likelihood $p(W \mid \lambda)$ of the weights follows a Gaussian distribution, and that $p(y_i \mid x_i, W, \gamma) \sim \NN(y_i, \gamma^{-1})$.

We can then make predictions $\hat{y}$ for a new input $\hat{x}$ using the \enquote{predictive distribution}
\begin{equation*}
    p(\hat{y} \mid \hat{x}, D)
    = \int p(\hat{y} \mid \hat{x}, W, \gamma) p(W, \gamma, \lambda \mid D) \d{\gamma} \d{\lambda} \d{W}.
\end{equation*}

We now show that under equal parameters, ASVGD improves upon SVGD on this Bayesian neural network task, since it has a lower error and higher likelihood.
The wall clock computation times for both SVGD and ASVGD are very similar, with ASVGD taking slightly longer.
We choose the step size $10^{-4}$, a batch size of 100, 50 hidden neurons, a training ratio of 90\%, AdaGrad for choosing the step size, median heuristic for choice of kernel width and Wasserstein regularization parameter $\eps = .1$.

\begin{table}[H]
    \centering
    \resizebox{\textwidth}{!}{
    \begin{tabular}{l|cc|cc|cc}
    \hline
    \multicolumn{1}{c|}{} & \multicolumn{2}{c|}{RMSE} & \multicolumn{2}{c|}{Log-likelihood} & \multicolumn{2}{c|}{time (seconds)} \\
    \hline
    Dataset & ASVGD & SVGD & ASVGD & SVGD & ASVGD & SVGD \\
    \hline
    Concrete & $\bm{8.862_{\pm 0.107}}$ & $9.208_{\pm 0.099}$ & $\bm{-3.560_{\pm 0.012}}$ & $-3.636_{\pm 0.010}$ & $11.588_{\pm 0.05}$  & $\bm{10.609_{\pm 0.06}}$ \\
    Energy   & $\bm{2.184_{\pm 0.019}}$ & $2.200_{\pm 0.021}$ & $\bm{-2.204_{\pm 0.009}}$ & $-2.211_{\pm 0.008}$ & $11.54_{\pm 0.045}$  & $\bm{10.560_{\pm{0.026}}}$ \\
    Housing  & $\bm{2.525_{\pm 0.031}}$ & $2.556_{\pm 0.056}$ & $\bm{-2.401_{\pm 0.009}}$ & $-2.405_{\pm 0.017}$ & $13.491_{\pm 0.045}$ & $\bm{12.334_{\pm 0.050}}$  \\
    Kin8mn   & $\bm{0.175_{\pm 0.000}}$ & $0.178_{\pm 0.001}$ & $\bm{0.322_{\pm 0.003}}$  & $0.306_{\pm 0.005}$  & $11.526_{\pm 0.032}$ & $\bm{10.652_{\pm 0.049}}$  \\
    Naval    & $\bm{0.007_{\pm 0.000}}$ & $0.007_{\pm 0.000}$ & $\bm{3.487_{\pm 0.001}}$  & $3.481_{\pm 0.002}$  & $15.026_{\pm 0.042}$ & $\bm{13.726_{\pm 0.037}}$ \\
    power    & $\bm{4.089_{\pm 0.011}}$ & $4.121_{\pm 0.015}$ & $\bm{-2.844_{\pm 0.003}}$ & $-2.854_{\pm 0.004}$ & $9.985_{\pm 0.029}$ & $\bm{9.265_{\pm 0.057}} $\\
    protein  & $\bm{5.114_{\pm 0.007}}$ & $5.163_{\pm 0.013}$       & $\bm{-3.050_{\pm 0.001}}$ & $-3.060_{\pm 0.003}$     & $12.092_{\pm 0.025}$ & $\bm{11.191_{\pm 0.034}}$ \\
    wine     & $0.223_{\pm 0.010}$      & $\bm{0.215_{\pm 0.008}}$  & $0.140_{\pm 0.036}$       & $\bm{0.171_{\pm 0.018}}$ & $13.532_{\pm 0.035}$ & $\bm{12.372_{\pm 0.041}}$\\
    \hline
    \end{tabular}}
    
    \vspace{3mm}
    
    \resizebox{\textwidth}{!}{\begin{tabular}{l|cc|cc|cc}
    \hline
    \multicolumn{1}{c|}{} & \multicolumn{2}{c|}{RMSE} & \multicolumn{2}{c|}{Log-likelihood} & \multicolumn{2}{c|}{time (seconds)} \\
    \hline
    Dataset & ASVGD & SVGD & ASVGD & SVGD & ASVGD & SVGD \\
    \hline
    Concrete & $\bm{5.536_{\pm 0.060}}$ & $7.349_{\pm 0.067}$ & $\bm{-3.135_{\pm 0.016}}$ & $-3.439_{\pm 0.010}$ & $14.867_{\pm 0.040}$  & $\bm{14.001_{\pm 0.044}}$ \\
    Energy   & $\bm{0.899_{\pm 0.057}}$ & $1.950_{\pm 0.028}$ & $\bm{-1.268_{\pm 0.068}}$ & $-2.088_{\pm 0.016}$ & $14.870_{\pm 0.051}$  & $\bm{13.942_{\pm 0.035}}$ \\
    Housing  & $\bm{2.346_{\pm 0.077}}$ & $2.386_{\pm 0.048}$ & $\bm{-2.305_{\pm 0.020}}$ & $-2.343_{\pm 0.014}$ & $18.278_{\pm 0.045}$ & $\bm{17.214_{\pm 0.077}}$  \\
    Kin8mn   & $\bm{0.118_{\pm 0.001}}$ & $0.165_{\pm 0.001}$ & $\bm{0.71_{\pm 0.011}}$ & $0.384_{\pm 0.004}$ & $14.859_{\pm 0.029}$ & $\bm{14.001_{\pm 0.033}}$  \\
    Naval    & $\bm{0.005_{\pm 0.000}}$ & $0.007_{\pm 0.000}$ & $\bm{3.801_{\pm 0.01}}$ & $3.504_{\pm 0.003}$ & $20.912_{\pm 0.57}$ & $\bm{19.704_{\pm 0.05}}$ \\
    power    & $\bm{3.951_{\pm 0.005}}$ & $4.035_{\pm 0.008}$ & $\bm{-2.799_{\pm 0.001}}$ & $-2.825_{\pm 0.002}$ & $12.142_{\pm 0.053}$ & $\bm{11.435_{\pm 0.054}} $\\
    protein  & $\bm{4.777_{\pm 0.007}}$ & $4.987_{\pm 0.009}$       & $\bm{-2.983_{\pm 0.001}}$ & $-3.026_{\pm 0.002}$     & $15.616_{\pm 0.04}$ & $\bm{14.752_{\pm 0.047}}$ \\
    wine     & $\bm{0.185_{\pm 0.013}}$      & $0.191_{\pm 0.017}$  & $\bm{0.201_{\pm 0.039}}$       & $0.146_{\pm 0.046}$ & $18.293_{\pm 0.097}$ & $\bm{17.244_{\pm 0.077}}$\\
    \hline
    \end{tabular}}
    \caption{Test root mean square error (RMSE, lower is better) and log-likelihood (LL, higher is better) after 2000 iterations with (top:) 20 particles, with gradient restart and speed restart and (bottom:) 10 particles, without restarts and with constant damping $\beta = 0.95$.}
    \label{tab:BNN_RMSE_and_LL_const_damp}
\end{table}

\section{Conclusion and future directions}
In this paper, we derived and analyzed ASVGD, an accelerated variant of SVGD that incorporates momentum methods and the Stein method. In numerical analysis, we studied ASVGD with the generalized bilinear kernel. We proved that if the initial distribution and the target distribution are Gaussians, then the density function will remain in the Gaussian family.
We then gave expressions for the asymptotically optimal kernel matrix $A$ and derived the asymptotically optimal damping parameter $\alpha$.
In numerical experiments, we demonstrated that ASVGD outperforms SVGD on toy non-Gaussian targets and real-world examples, such as Bayesian neural networks.

Further directions of interest include proving convergence guarantees on the full density manifold, both in continuous and discrete time, and on the level of densities and particles, and extending the analysis to log-concave targets, including mixtures of Gaussians (this condition might have to be adapted to the kernel). 
We are also interested in finding the optimal damping parameter $\alpha$ without linearizing the ASVGD at the equilibrium, but instead based on the convexity of the energy functional with respect to the (Wasserstein-)Stein metric.
Furthermore, there are other aspects that we have not considered in this paper, namely, choosing Rényi divergences or Tsallis divergences instead of the KL divergence in ASVGD, and comparing their convergence speed. 
We are also interested in studying adaptive step size choices, choosing conformal-symplectic time discretization, and handling potential algorithm bias.  

\paragraph{Acknowledgments.}

W. Li's work is supported by the AFOSR YIP award No.~FA9550-23-10087, NSF RTG: 2038080, and NSF DMS-2245097.
The paper was finalized during a visit of W. Li at the TU Berlin support by the SPP 2298 Theoretical Foundations of Deep Learning.
V. Stein thanks his advisor, Gabriele Steidl, for her support and guidance throughout.

\clearpage

\bibliographystyle{abbrv}
\bibliography{Bibliography}

\clearpage

\begin{appendices}

\section{Additional information about Fréchet manifolds} \label{sec:Frechet}

Since $C^{\infty}(\Omega)$ is not a Banach space, we use Fréchet spaces and Fréchet manifolds instead of the more straightforward concept of Banach manifolds.

\begin{definition}[Fréchet space {\cite[p.~253]{S1989}, \cite[p.~85]{T1967}, \cite[Def.~1.1.1]{H1982}}]
    A Fréchet space is a real complete metrizable locally convex topological vector space.
\end{definition}

\begin{definition}[Fréchet differentiability {\cite[Def.~7.1.5]{S1989}, \cite[Def.~8.1]{M1980}, \cite[Def.~3.1,~3.2]{L1988}}]
    Let $V$ and $W$ be topological vector spaces and $O \subset V$ be an open subset.
    A map $f \colon O \to W$ belongs to the class $\C^1(O, W)$ if for every $x \in O$ and every $v \in V$, the limit
    \begin{equation*}
        \D f(x; v)
        \coloneqq \lim_{t \to 0} \frac{1}{t} \big( f(x + t v) - f(x) \big)
    \end{equation*}
    exists and if the resulting map $\D f \colon O \times V \to W$ is (jointly!) continuous.

    We define higher derivatives in the usual inductive fashion, see \cite[Def.~7.1.7]{S1989} and we define the class of \textit{smooth} functions as $\C^{\infty}(O, V) \coloneqq \bigcap_{k \in \N} \C^k(O, V)$.
\end{definition}

\begin{remark}
If $f \in \C^1(O, W)$, then $f$ is continuous.
\end{remark}

For the definitions of Fréchet manifold and Fréchet bundle, consult 
\cite[Def.~4.1.1,~4.3.1]{H1982}.

\begin{definition}[Fréchet topology on $\widetilde{\P}(\Omega)$]
    The  (tamely graded) Fréchet topology on $\widetilde{\P}(\Omega)$ is generated by the family of seminorms $(\| \cdot \|_k)_{k \in \N}$, where $\| f \|_{k} \coloneqq \sup_{x \in \Omega} \sum_{| \alpha | \le k} \left| \partial^{\alpha} f(x) \right|$ for $f \in \widetilde{\P}(\Omega)$.
    Hence, a sequence $(f_n)_{n \in \N} \in \widetilde{\P}(\Omega)$ converges to $f \in \widetilde{\P}(\Omega)$ if and only if $\partial^{\alpha} f_n \to \partial^{\alpha} f$ uniformly for $n \to \infty$ for all $\alpha \in \N_0^d$, where $d$ is the dimension of the manifold $\Omega$. 
\end{definition}

\begin{remark}[Tangent space to the density manifold {\cite[p.~717--718]{L1988}}]
    The tangent space $T_{\rho} \widetilde{\P}(\Omega)$ is isomorphic to $\mathcal {GX}(\Omega)$, the space of gradient vector fields, which is given meaning by the following: the space of sections of the tangent bundle of $\Omega$, $\Gamma(T \Omega)$, can be decomposed as $\mathcal{GX}(\Omega) \oplus \{ X \in \mathcal{X}(\Omega): \nabla \cdot X = 0 \}$, the latter being the Fréchet space of divergence-free (w.r.t. the volume form) vector fields and orthogonality being understood with respect to $\llangle \cdot, \cdot \rrangle_{\rho} \colon \Gamma(T \Omega) \times \Gamma(T \Omega) \to \R$, $(X, Y) \mapsto \int_{\Omega} \langle X, Y \rangle \rho$ \cite[Cor.~3.6]{L1988}. This is the Hodge decomposition, which is sometimes interpreted as a generalization of the Helmholtz decomposition \cite[p.~713]{L1988}.
\end{remark}

\begin{remark}
    We do not have the duality $\big(T_{\rho} \widetilde{\P}(\Omega)\big)^* \cong T_{\rho}^* \widetilde{\P}(\Omega)$, like in the setting of Banach manifolds, since the dual of a Fréchet space $X$ is a Fréchet space if and only if $X$ is already a Banach space.
\end{remark}

\begin{remark}[Smoothness definition for the metric tensor]
    Saying that $\rho \mapsto G_{\rho}$ is smooth means that, given any open subset $U \subset \tilde{\P}(\Omega)$, and any (smooth) vector fields $X$ and $Y$ on $U$, the real function $g(X, Y) \colon \widetilde{\P}(\Omega) \to \R$, $\rho \mapsto g_{\rho}(X_{\rho}, Y_{\rho})$ is smooth.
\end{remark}

\section{Well-known theorems and auxiliary propositions} \label{sec:well-knwon}
Since we deal with Gaussians and the Kullback-Leibler divergence, let us first state these formulas for their derivatives with respect to different parameters.

\begin{prop} \label{prop:GaussCalcationRules}
    Let $\theta = (\mu, \Sigma) \in \R^d \times \Sym_+d(D)$ and $\rho_{\theta}$ the density of a random variable $X \sim \NN(\mu, \Sigma)$.
    We have
    \begin{align*}
        \nabla_x \rho_{\theta}(x)
        & = - \rho_{\theta}(x) \Sigma^{-1}(x - \mu) \\
        \nabla_{\theta} \rho_{\theta}(x)
        & = \rho_{\theta}(x) \begin{pmatrix}
            \Sigma^{-1} (x - \mu) \\
            \frac{1}{2} \left[\Sigma^{-1} (x - \mu) (x - \mu)^{\tT} - \id\right] \Sigma^{-1}
        \end{pmatrix}
    \end{align*}
    as well as (see, e.g. \cite[Lemma~G.1]{LGBP2024})
    \begin{equation*}
        \int_{\R^d} (x - \mu)^{\tT} A (x - \mu) \rho_{\theta}(x) \d{x}
        = \tr(A \Sigma)
        \quad \text{and} \quad
        \int_{\R^d} b^{\tT} (x - \mu) A (x - \mu) \rho_{\theta}(x) \d{x}
        = A \Sigma b
    \end{equation*}
    for any $A \in \R^{d \times d}$ and $b \in \R^d$.

\end{prop}

\begin{prop}[Kullback-Leibler divergence of Gaussians and its derivative] \label{prop:functionalderivative_KL}
    For $(\mu, \Sigma), (\nu, Q) \in \R^d \times \Sym_+(d)$ we have 
    \begin{equation*}
        \KL(\NN(\mu, \Sigma) \mid \NN(\nu, Q))
        = \frac{1}{2} \left( \tr(Q^{-1} \Sigma) - d + (\nu - \mu)^{\tT} Q^{-1} (\nu - \mu) + \ln\left(\frac{\det(Q)}{\det(\Sigma)}\right)\right)
    \end{equation*}
    and thus
    \begin{equation*}
        \nabla_{(\mu, \Sigma)} \KL(\NN(\mu, \Sigma) \mid \NN(\nu, Q))
        = \begin{pmatrix}
            Q^{-1} (\mu - \nu) \\
            \frac{1}{2} \left(Q^{-1} - \Sigma^{-1}\right)
        \end{pmatrix}.
    \end{equation*}    
\end{prop}

There is a simple formula for the Moore-Penrose pseudoinverse of low-rank matrices with special structure.

\begin{lemma} \label{lemma:pseudo_inverse_rank_2}
    If $v_1 \perp v_2$ and $A = \lambda_1 v_1 v_1^{\tT} + \lambda_2 v_2 v_2^{\tT}$ for some $\lambda_1, \lambda_2 \ne 0$ and $\| v_1 \| = \| v_2 \| = 1$, then $A^{\dagger} = \frac{1}{\lambda_1} v_1 v_1^{\tT} + \frac{1}{\lambda_2} v_2 v_2^{\tT}$.
\end{lemma}

\begin{proof}
    One can easily verify that $A^{\dagger} A = A A^{\dagger} = v_1 v_1^{\tT} + v_2 v_2^{\tT}$ is symmetric and, consequently, $A A^{\dagger} A = A$ and $A^{\dagger} A A^{\dagger} = A^{\dagger}$.
\end{proof}

Lastly, we will use the following well-known proposition about the solution of the so-called Lyapunov equation.
\begin{prop}[Lyapunov equation]
    Let $P, Q \in \Sym(d)$.
    Then there exists a unique solution $X \in \Sym(d)$ of the Lyapunov equation
    \begin{equation*}
        P X + X P
        = Q.
    \end{equation*}
\end{prop}

Lastly, note that for any matrix $B \in \R^{d \times d}$, the matrix with rows $B X_t^{(j)}$ is given by $X_t B^{\tT}$ and the $(i, j)$-th entry of $P P^{\tT}$ is precisely $\langle p^{(i)}, p^{(j)} \rangle_{\R^d}$.
We thus have
\begin{equation*}
    \sum_{m, \ell = 1}^{N} \langle p^{(\ell)}, p^{(m)} \rangle K(X^{(m)}, X^{(\ell)})
    = \tr(P^{\tT} K P).
\end{equation*}

\paragraph{The Kronecker product}
The Kronecker product $\otimes$ is not commutative.
However, for any matrices $A, B$ of arbitrary size, there exist two permutation matrices $P_1, P_2$ such that $A \otimes B = P_1 (B \otimes A) P_2$.
If the dimensions of $A$ and $B$ are identical, then we can choose $P_1 = P_2$, so that $A \otimes B$ and $B \otimes A$ have the same spectrum.

\begin{prop}
    Let $A, B$ be quadratic matrices.
    The eigenvalues of $A \times B$ are all of the form $\lambda_A + \lambda_B$, where $\lambda_A$ is an eigenvalue of $A$ and $\lambda_B$ is an eigenvalue of $B$.
\end{prop}


\begin{prop}[Kronecker product and vectorization] \label{prop:Kronecker_Vectorization}
    For $B \in \R^{m \times n}$, $C \in \R^{\ell, k}$ and $V \in \R^{k \times n}$ we have
    \begin{equation*}
        (B \otimes C) \vec(V)
        = \vec(C V B^{\tT}).
    \end{equation*}    
\end{prop}

\section{Proofs and calculations}
\subsection{Mean field limits and particle approximations} \label{subsec:particleFormulation}

In this subsection, we give an informal justification for turning the coupled PDEs dealt with above into particle schemes.

First note that for differentiable functions $g \colon (0, \infty) \to \Omega$ and $f \colon (0, \infty) \times \Omega \to \R^d$ the chain rule implies
\begin{equation} \label{eq:chain_rule}
    \frac{\d}{\d t} f(t, g(t))
    = (\partial_2 f)(t, g(t)) \cdot g'(t) + (\partial_1 f)(t, g(t)).
\end{equation}
(Some authors referred to this identity as the material derivative in fluid dynamics.)
Consider a smooth solution $(\rho_t, \Phi_t)_{t > 0} \subset \widetilde{\P}(\Omega) \times T^* \widetilde{\P}(\Omega)$ satisfying $\partial_t \rho_t = - [G_{\rho_t}]^{-1}(\nabla \Phi_t)$, where $\Phi_t \in T_{\rho_t}^* \widetilde{\P}(\Omega)$ for each $t > 0$ and let $\Phi \colon (0, \infty) \times \Omega \to \R$, $(t, x) \mapsto \Phi_t(x)$.
Then, the gradient and Hessian matrix of $\Phi_t \colon \Omega \to \R$ can be rewritten as follows: $\nabla \Phi_t(x) = \partial_2 \Phi(t, x) \in \R^d$ and $\nabla^2 \Phi_t(x) = \partial_2^2 \Phi(t, x) \in \R^{d \times d}$ for $t > 0$ and $x \in \Omega$.

Let $X \colon (0, \infty) \to \Omega$, $t \mapsto X_t$ be the differentiable trajectory of a particle.
The velocity along that trajectory is $\dot{X} \colon (0, \infty) \to \Omega \times \R^d$.
Under regularity assumptions (detailed e.g. in \cite[Lemma~8.1.4]{AGS08}) the continuity equation $\partial_t \rho_t = - [G_{\rho_t}]^{-1}(\nabla \Phi_t)$ implies that the particles are driven by the velocity field, so in the case of \eqref{eq:Accelerated_Stein_Flow} we have
\begin{equation*}
    \dot{X}_t
    = \int_{\Omega} K(X_t, y) \rho_t(y) \nabla \Phi_t(y) \d{y} + \eps \nabla \Phi_t(X_t).
\end{equation*}
Now, define $V \colon (0, \infty) \to T_{X_t} \Omega \cong \R^d$, $t \mapsto V_t \coloneqq \nabla \Phi_t(X_t)$.
By \eqref{eq:chain_rule} we have
\begin{equation} \label{eq:dotVt}
    \dot{V}_t
    \coloneqq \frac{\d}{\d t} V_t
    = \frac{\d}{\d t} (\partial_2 \Phi)(t, X_t)
    = (\partial_2^2 \Phi)(t, X_t) \dot{X}_t + (\partial_1 \partial_2 \Phi)(t, X_t)
    = \nabla^2 \Phi_t(X_t) \dot{X}_t + (\partial_t \nabla \Phi_t)(X_t).
\end{equation}
By applying $\nabla$ to the second equation in \eqref{eq:Accelerated_Stein_Flow} and assuming we can exchange the order of differentiation, and evaluating at $x = X_t$, we obtain 
\begin{align*}
    \partial_t \nabla \Phi_t(X_t)
    & = - \alpha_t \nabla \Phi_t(X_t)
    - \nabla \int_{\Omega} K(X_t, y) \langle \nabla \Phi_t(y), \nabla \Phi_t(X_t) \rangle_{\R^n} \rho_t(y) \d{y} \\
    & \qquad - [\nabla \delta E(\rho_t)](X_t)
     - \frac{\eps}{2} \nabla \| \nabla \Phi_t(X_t) \|_2^2 \\
    & = - \alpha_t \nabla \Phi_t(X_t)
    - \underbrace{\int_{\Omega} \nabla_1 K(X_t, y) \langle \nabla \Phi_t(y), \nabla \Phi_t(X_t) \rangle_{\R^n} \rho_t(y) \d{y}}_{= \int_{\Omega} \nabla_1 K(X_t, y) \nabla \Phi_t(y)^{\tT} \rho_t(y) \d{y} \cdot V_t} \\
    & \qquad \underbrace{- \int_{\Omega} K(X_t, y) \nabla^2 \Phi_t(X_t) \nabla \Phi_t(y) \rho_t(y) \d{y} - \frac{\eps}{2} \nabla \| \nabla \Phi_t(X_t) \|_2^2}_{= - \nabla^2 \Phi_t(X_t) \dot{X}_t}
     - [\nabla \delta E(\rho_t)](X_t) \\
    & = - \alpha_t \nabla \Phi_t(X_t)
    - \int_{\Omega} \nabla_1 K(X_t, y) \nabla \Phi_t(y)^{\tT} \rho_t(y) \d{y} \cdot V_t 
    - \nabla^2 \Phi_t(X_t) \dot{X}_t
    - [\nabla \delta E(\rho_t)](X_t),
\end{align*}
where we used the product rule, and assumed that we can exchange integration with differentiation in the last line.
The term $ - \nabla^2 \Phi_t(X_t) \dot{X}_t$ cancels.

\subsection{Proof of \texorpdfstring{\cref{lemma:central_Lemma}}{Lemma about acceleration in the Stein flow}}  \label{subsec:proof}
\begin{proof}
    We have
    \begin{align*}
        \dot{Y}_t
        &= \int_{\Omega} (\nabla_1 K)(X_t, y) \langle X'(t), \nabla \Phi_t(y) \rangle_{\R^d} \rho_t(y) \d{y} \\
        & \qquad + \int_{\Omega} K(X_t, y) (\partial_t \nabla \Phi_t)(y) \rho_t(y) \d{y} \\
        & \qquad + \int_{\Omega} K(X_t, y) \nabla \Phi_t(y) (\partial_t \rho_t)(y) \d{y} 
        + \eps \dot{V}_t \\
        \overset{\eqref{eq:Accelerated_Stein_Flow}, \eqref{eq:Acc_Stein}}&{=} 
        \int_{\Omega} \int_{\Omega} (\nabla_1 K)(X_t, y) K(X_t, z)  \left\langle \nabla \Phi_t(z), \nabla \Phi_t(y) \right\rangle_{\R^d} \rho_t(z) \rho_t(y) \d{z} \d{y} \\
        & \qquad + \eps \int_{\Omega} (\nabla_1 K)(X_t, y) \langle V_t, \nabla \Phi_t(y) \rangle_{\R^d} \rho_t(y) \d{y} \\
        & \qquad + \int_{\Omega} K(X_t, y) \bigg(- \alpha_t \nabla \Phi_t(y) - \nabla_y \left[\int_{\Omega} K(z, y) \langle \nabla \Phi_t(z), \nabla \Phi_t(y) \rangle_{\R^n} \rho_t(z) \d{z}\right] \\
        & \qquad \qquad \qquad - \frac{\eps}{2} \nabla_y \| \nabla \Phi_t(y) \|_2^2 - \nabla \delta E(\rho_t)(y)\bigg) \rho_t(y) \d{y} \\
        & \qquad + \int_{\Omega} K(X_t, y) \nabla \Phi_t(y) \left( - \nabla \cdot \left(\rho_t(y) \left(
        \int_{\Omega} K(y, z) \rho_t(z) \nabla \Phi_t(z) \d{z} + \eps \nabla \Phi_t(y)\right)
        \right) \right) \d{y} 
        + \eps \dot{V}_t \\
        \overset{(\star)}&{=} \int_{\Omega} \int_{\Omega} (\nabla_1 K)(X_t, y) K(X_t, z)  \left\langle \nabla \Phi_t(z), \nabla \Phi_t(y) \right\rangle_{\R^d} \rho_t(z) \rho_t(y) \d{y} \d{z} \\
        & \qquad + \eps \int_{\Omega} (\nabla_1 K)(X_t, y) \langle V_t, \nabla \Phi_t(y) \rangle_{\R^d} \rho_t(y) \d{y} \\
        & \qquad - \alpha_t  \int_{\Omega} K(X_t, y) \nabla \Phi_t(y)  \rho_t(y) \d{y}
        - \int_{\Omega} K(X_t, y) \nabla \delta E(\rho_t)(y) \rho_t(y) \d{y} \\
        & \qquad - \frac{\eps}{2} \int_{\Omega} K(X_t, y) \nabla \| \nabla \Phi_t(y) \|_2^2 \rho_t(y) \d{y} \\
        & \qquad - \int_{\Omega} \int_{\Omega} K(X_t, y) \rho_t(y) \nabla_y \left[ K(z, y) \langle \nabla \Phi_t(z), \nabla \Phi_t(y) \rangle_{\R^n} \right] \rho_t(z) \d{y} \d{z} \\
        & \qquad + \int_{\Omega} \int_{\Omega} \nabla_y \left[ K(X_t, y) \langle \nabla \Phi_t(y), \nabla \Phi_t(z) \rangle_{\R^n} \right] \rho_t(y) K(y, z) \rho_t(z) \d{y} \d{z} \\
        & \qquad + \eps \int_{\Omega} \nabla_y \big[ K(X_t, y) \nabla \Phi_t(y) \big] \rho_t(y) \nabla \Phi_t(y) \d{y}
        + \eps \dot{V}_t \\
        \overset{\eqref{eq:Acc_Stein}}&{=} \int_{\Omega} \int_{\Omega} (\nabla_1 K)(X_t, y) K(X_t, z)  \left\langle \nabla \Phi_t(z), \nabla \Phi_t(y) \right\rangle_{\R^d} \rho_t(z) \rho_t(y) \d{y} \d{z} \\
        & \qquad + \eps \int_{\Omega} (\nabla_1 K)(X_t, y) \langle V_t, \nabla \Phi_t(y) \rangle_{\R^d} \rho_t(y) \d{y}
        - \alpha_t (Y_t - \eps V_t) \\
        & \qquad - \int_{\Omega} K(X_t, y) \nabla \delta E(\rho_t)(y) \rho_t(y) \d{y} \\
        & \qquad - \frac{\eps}{2} \int_{\Omega} K(X_t, y) \nabla \| \nabla \Phi_t(y) \|_2^2 \rho_t(y) \d{y} \\
        & \qquad - \int_{\Omega} \int_{\Omega} K(X_t, y) \rho_t(y) \left[ (\nabla_2 K)(z, y) \langle \nabla \Phi_t(y), \nabla \Phi_t(z) \rangle_{\R^n} \right] \rho_t(z) \d{y} \d{z} \\
        & \qquad - \int_{\Omega} \int_{\Omega} K(X_t, y) \rho_t(y) \left[ K(z, y) \nabla^2 \Phi_t(y) \nabla \Phi_t(z) \right] \rho_t(z) \d{y} \d{z} \\
        & \qquad + \int_{\Omega} \int_{\Omega} (\nabla_2 K)(X_t, y) \langle \nabla \Phi_t(y), \nabla \Phi_t(z) \rangle_{\R^n} \rho_t(y) K(y, z) \rho_t(z) \d{y} \d{z} \\
        & \qquad + \int_{\Omega} \int_{\Omega} K(X_t, y) \nabla^2 \Phi_t(y) \nabla \Phi_t(z) \rho_t(y) K(y, z) \rho_t(z) \d{y} \d{z} \\
        & \qquad + \eps \int_{\Omega} \nabla_y \big[ K(X_t, y) \nabla \Phi_t(y) \big] \rho_t(y) \nabla \Phi_t(y) \d{y}
        + \eps \dot{V}_t,
    \end{align*}
    where in $(\star)$ we use integration by parts in the third line, the divergence theorem in every component of the integral (i.e. $- \int_{\Omega} \nabla \Psi(y) \nabla \cdot \phi(y) \d{y} = \int_{\Omega} \nabla^2 \Psi(y) \phi(y) \d{y}$ for any smooth vector field $\phi \colon \Omega \to \R^d$ and smooth function $\Psi \colon \Omega \to \R$) and that
    \begin{equation*}
        \nabla_y \big(K(X_t, y) \nabla \Phi_t(y) \big)
        = (\nabla_2 K)(X_t, y) \nabla \Phi_t(y)^{\tT}
        + K(X_t, y) \nabla^2 \Phi_t(y)
    \end{equation*}
    as well as Fubini's theorem in the first and last integrals.
    
    The terms involving the Hessian of $\Phi_t$ cancel out, so putting all double integrals together yields
    \begin{equation*} 
    \begin{aligned}
        \dot{Y}_t
        & = - \alpha_t Y_t - \int_{\Omega} K(X_t, y) \nabla \delta E(\rho_t)(y) \rho_t(y) \d{y} \\
        & \quad + \int_{\Omega} \int_{\Omega} \rho_t(y) \rho_t(z) \langle \nabla \Phi_t(z), \nabla \Phi_t(y) \rangle_{\R^n} \\
        & \qquad \qquad \cdot \bigg[ K(y, z) (\nabla_2 K)(X_t, y) - K(X_t, y) (\nabla_2 K)(z, y) + (\nabla_1 K)(X_t, y) K(X_t, z) 
        \bigg] \d{y} \d{z} \\
        & \quad + \eps \bigg( \int_{\Omega} (\nabla_1 K)(X_t, y) \langle V_t, \nabla \Phi_t(y) \rangle_{\R^n} \rho_t(y) \d{y}
        + \alpha_t V_t
        - \frac{1}{2} \int_{\Omega} K(X_t, y) \nabla \| \nabla \Phi_t(y) \|_2^2 \rho_t(y) \d{y} \\
        & \quad + \int_{\Omega} \nabla_y \big[ K(X_t, y) \nabla \Phi_t(y) \big] \nabla \Phi_t(y) \rho_t(y) \d{y}
         + \dot{V}_t\bigg).
    \end{aligned}    
    \end{equation*}
    Using that $\frac{1}{2} \nabla \| \nabla \Phi \|_2^2 = \nabla^2 \Phi \nabla \Phi$, the $\eps$-term and we obtain
    \begin{equation} \label{eq:final}
    \begin{aligned}
        \dot{Y}_t
        & = - \alpha_t Y_t - \int_{\Omega} K(X_t, y) \nabla \delta E(\rho_t)(y) \rho_t(y) \d{y} \\
        & \quad + \int_{\Omega} \int_{\Omega} \rho_t(y) \rho_t(z) \langle \nabla \Phi_t(z), \nabla \Phi_t(y) \rangle_{\R^n} \\
        & \qquad \qquad \cdot \bigg[ K(y, z) (\nabla_2 K)(X_t, y) - K(X_t, y) (\nabla_2 K)(z, y) + (\nabla_1 K)(X_t, y) K(X_t, z) 
        \bigg] \d{y} \d{z} \\
        & \quad + \eps \bigg( \int_{\Omega} \left( (\nabla_1 K)(X_t, y) \langle V_t, \nabla \Phi_t(y) \rangle_{\R^n} + (\nabla_2 K)(X_t, y) \| \nabla \Phi_t(y) \|_2^2\right) \rho_t(y) \d{y}
        + \alpha_t V_t
         + \dot{V}_t \bigg).
    \end{aligned}    
    \end{equation}
\end{proof}


\subsection{Proof of \texorpdfstring{\cref{lemma:ASVGD_fully_discrete}}{a lemma}} \label{subsec:proof_full_discretization}

\begin{proof}
    \begin{enumerate}
        \item 
        \textbf{Space discretization.}
        Since $X_t \sim \rho_t$, we can simulate \eqref{eq:Accelerated_Stein_Flow_2} using $N \in \N$ particles $(X_{t}^{(j)})_{j = 1}^{N} \subset \R^d$ and their accelerations $(Y_t^{(j)})_{j = 1}^{N} \subset \R^d$ at time $t$ by replacing $\rho_{t}$ by its empirical estimate $\frac{1}{N} \sum_{j = 1}^{N} \delta_{X_{t}^{(j)}}$.
        In terms of these particles, \eqref{eq:dot_Yt} becomes
        \begin{align*}
            \dot{Y}_t^{j}
            & = - \alpha_t Y_t^{j}
            + \frac{1}{N} \sum_{i = 1}^{N} (\nabla_2 K)(X_t^{(j)}, X_t^{(i)})
            - (K_t)_{j, i} \nabla f(X_t^{(i)})
            + \frac{1}{N^2} \sum_{i, \ell = 1}^{N} \langle V_t^i, V_t^{\ell} \rangle_{\R^n} \\
            & \qquad \quad \cdot \bigg[
                (K_t)_{i, {\ell}} (\nabla_2 K)(X_t^{(j)}, X_t^{(i)}) 
                + (K_t)_{j, \ell} (\nabla_1 K)(X_t^{(j)}, X_t^{(i)})
                - (K_t)_{j, i} (\nabla_2 K)(X_t^{\ell}, X_t^{(i)})
            \bigg].
        \end{align*}
        where $(K_t)_{i, j} \coloneqq K(X_t^{(i)}, X_t^{(j)})$ for $i, j \in \{ 1, \ldots, N \}$ and $V_t^{(\ell)} \coloneqq \nabla\Phi_{t}(X_{t}^{(\ell)}) \in \R^d$ for $\ell \in \{ 1, \ldots, N \}$ for $t > 0$.
    
        \item 
        \textbf{Time discretization.} Let us consider a forward Euler discretization in time on the time interval $[0, T]$ with $T > 0$ and with time steps $0 = t_0 < t_1 < \ldots < t_M = T$.
        This yields the step sizes $\tau_k \coloneqq t_{k + 1} - t_k > 0$, $k \in \{ 0, \ldots, M - 1 \}$ and thus the update equations
        \begin{equation} \label{eq:multidim}
            \begin{cases}
                X_{t_{k + 1}}^{(j)}
                \gets X_{t_k}^{(j)} + \tau_k Y_{t_k}^{(j)}, \\
                V_{t_{k + 1}}^{(i)} = N K_{t_{k + 1}}^{\dagger} Y_{t_k} \quad \text{(this follows by rearranging \eqref{eq:Acc_Stein})} \\
                Y_{t_{k + 1}}^{(j)}
                \gets (1 - \tau_k \alpha_{t_k}) Y_{t_k}^{(j)}
                + \frac{\tau_k}{N} \sum_{i = 1}^{N} (\nabla_2 K)(X_{t_{k + 1}}^{(j)}, X_{t_{k + 1}}^{(i)}) - K(X_{t_{k + 1}}^{(j)}, X_{t_{k + 1}}^{(i)}) \nabla f(X_{t_{k + 1}}^{(i)}) \\
                \qquad \qquad + \frac{\tau_k}{N^2} \sum_{i, \ell = 1}^{N} \langle V_{t_{k + 1}}^{(i)}, V_{t_{k + 1}}^{(\ell)} \rangle
                \bigg[ K(X_{t_{k + 1}}^{(i)}, X_{t_{k + 1}}^{(\ell)}) (\nabla_2 K)(X_{t_{k + 1}}^{(j)}, X_{t_{k + 1}}^{(i)}) \\
                \qquad \qquad + K(X_{t_{k + 1}}^{(j)}, X_{t_{k + 1}}^{(\ell)})
                (\nabla_1 K)(X_{t_{k + 1}}^{(j)}, X_{t_{k + 1}}^{(i)}) 
                - K(X_{t_{k + 1}}^{(j)}, X_{t_{k + 1}}^{(i)})
                (\nabla_2 K)(X_{t_{k + 1}}^{(\ell)}, X_{t_{k + 1}}^{(i)})
                \bigg].
            \end{cases}
        \end{equation}
        for $j \in \{ 1, \ldots, N \}$ and $k \in \{ 0, \ldots, M - 1 \}$.
        
        Like in \cite{SBC2016}, we will discretize the damping $\alpha_t$ by replacing $1 - \tau_k \alpha_{t_k}$ by $\alpha_k$, which is $\frac{k - 1}{k + 2}$ or a constant, which is the right asymptotic choice for the correct time rescaling.
        
        Plugging in the definition of the matrices in the statement of the lemma, this concludes the derivation.
    \end{enumerate}
\end{proof}

\subsection{Simplification of the particle scheme for two kernels}
\paragraph{Gaussian kernel}
The most commonly used kernel in applications is the Gaussian kernel $K_{\sigma}(x, y) \coloneqq \exp\left(-\frac{1}{2 \sigma^2} \| x - y \|_2^2\right)$ with bandwidth $\sigma^2 > 0$.
We now show how the updates from \cref{lemma:ASVGD_fully_discrete} simplify for this kernel.

\begin{example}[Energy and interaction term]
    For the Gaussian kernel, in \eqref{eq:dot_Yt}, the term in square brackets in the interaction term simplifies to
    \begin{equation*}
        \frac{1}{\sigma^2} K(x, y) \big( K(y, z) (x - z) + K(x, z) (y - x)\big).
    \end{equation*}
    
    This simplification is not possible for other radial kernels.

    Since the partial derivative of the kernel is a linear term times the kernel, the energy term in \eqref{eq:dot_Yt} becomes
    \begin{equation*}
        \frac{1}{\sigma^2} \int_{\Omega} \left( x - y - 2 \sigma^2 \nabla f(y) \right) K(x, y) \rho_t(y) \d{y}.
    \end{equation*}
    For all particles, we aggregate this into a matrix:
    \begin{equation*}
        \frac{1}{\sigma^2 N} \left[ (\diag(K \1_N) - K) X - 2 \sigma^2 K \nabla f(X) \right].
    \end{equation*}
\end{example}

\begin{remark}
    Note that for a distance matrix $W \in \R^{N \times N}$, the matrix $\diag(W \1_N) - W$ is called the \textit{distance Laplacian} \cite{AH2013}, which we think of as a generalization of the graph Laplacian.
\end{remark}

In comparison, the non-accelerated Stein variational gradient descent scheme \cite{LW2016} has the update
\begin{equation*}
    X_{t_{k + 1}}^{i}
    = X_{t_k}^{i} + \frac{\tau_k}{N} \sum_{j = 1}^{N} (X_{t_k}^{i} - X_{t_k}^j - \nabla f(X_{t_k}^{j})) K(X_{t_k}^j, X_{t_k}^{i}),
\end{equation*}
which yields the following algorithm.

\begin{algorithm}[H]
    \caption{Stein variational gradient descent with Gaussian kernel.} \label{alg:TODO}
    \KwData{Gaussian kernel width $\sigma > 0$, number of particles $N \in \N$, number of steps $M \in \N$, step sizes $(\tau_k)_{k = 0}^{M} \subset (0, \infty)$}
    \KwResult{Matrix $X_{t_{M}}$, whose rows are particles particles that approximate the target distribution $\pi \sim \exp(-f)$.}
    \For{$k=0,\ldots, M$}{
    $X_{t_{k + 1}} \gets X_{t_k} + \frac{\tau}{N} (\diag(K_{t_k} \1_N) - K_{t_k}) X_{t_k} - \frac{\tau}{N \sigma^2} K_{t_k} \nabla f(X_{t_k})$}
\end{algorithm}

\paragraph{The generalized bilinear kernel}
The simplest positive definite kernel one can think of is $K(x, y) \coloneqq x^{\tT} y$.
As was motivated in \cref{sec:Bilinear_Analysis}, considering the kernel $K(x, y) \coloneqq x^{\tT} A y$ for some $A \in \Sym_+(\R; d)$ instead, can improve the asymptotic convergence rate of (A)SVGD.
However, this kernel is not integrally strictly positive definite (ISPD), since for any non-zero $\mu \in \mathcal{M}(\Omega)$ we have $\int_{\Omega} \int_{\Omega} K(x, y) \d{\mu}(x) \d{\mu}(y) = \| A^{\frac{1}{2}} \E[\mu] \|_2^2$, which is zero if and only if $\E[\mu] \in \ker(A^{\frac{1}{2}})$.
Since much of the SVGD theory builds on ISPD kernels \cite{DNS2023,NR2023}, we instead consider $K(x, y) \coloneqq x^{\tT} A y + a$, for $a > 0$, which is ISPD.
Since multiplying $K$ by a constant is equivalent to dividing the step size of SVGD by that constant, we fix $a = 1$.

We now show how the updates from \cref{lemma:ASVGD_fully_discrete} simplify for this kernel.

Since $\nabla_1 K(x, y) = A y $ and $\nabla_2 K(x, y) = A x$, the update \eqref{eq:dot_Yt} becomes
\begin{equation*} \label{eq:dot_Yt_bilinear}
\begin{aligned}
    \dot{Y}_t
    & = - \alpha_t Y_t + \int_{\Omega} \big(A X_t - (X_t^{\tT} A y + 1) \nabla f(y) \big) \rho_t(y) \d{y}
    + \int_{\Omega} \int_{\Omega} \rho_t(y) \rho_t(z) \langle \nabla \Phi_t(z), \nabla \Phi_t(y) \rangle_{\R^n} \\
    & \qquad \cdot \left[ (y^{\tT} A z + 1) A X_t + (X_t^{\tT} A z + 1) A y - (X_t^{\tT} A y + 1) A z \right] \d{y} \d{z} \\
    & = - \alpha_t Y_t + A X_t  - \int_{\Omega} (X_t^{\tT} A y + 1) \nabla f(y) \rho_t(y) \d{y} \\
    & \qquad + A X_t \int_{\Omega} \int_{\Omega} \langle \nabla \Phi_t(z), \nabla \Phi_t(y) \rangle_{\R^n} (y^{\tT} A z + 1) \rho_t(y) \d{y} \rho_t(z) \d{z}    
\end{aligned}
\end{equation*}
where we used the fact that the second and third summands in the square brackets cancel each other out, since one can exchange the roles of $y$ and $z$.

For the bilinear kernel, the kernel matrix at time $t$,
\begin{equation*}
    K_{t}
    \coloneqq \left( K\left(X_{t}^{(i)}, X_{t}^{(j)}\right) \right)_{i, j = 1}^{N}
    = X_t A X_{t}^{\tT} + a \1_{N \times N}
    = X_t A^{\frac{1}{2}} (X_t A^{\frac{1}{2}})^{\tT} + a \1_N \1_N^{\tT},
\end{equation*}
is symmetric positive semidefinite.

Rewrite the updates \eqref{eq:multidim} in matrix form we obtain \cref{algo:Particle_ASGD_Bilinear}

\begin{algorithm}[H]
    \caption{Accelerated Stein variational gradient descent particle algorithm with generalized bilinear kernel.} \label{algo:Particle_ASGD_Bilinear}
\KwData{Kernel parameter matrix $A \in \Sym_+(d)$, number of particles $N \in \N$, number of steps $M \in \N$, step sizes $(\tau_k)_{k = 0}^{M} \subset (0, \infty)$}
\KwResult{Matrix $X_{t_{M}}$, whose rows are particles particles that approximate the target distribution $\pi \sim \exp(-f)$.}
\For{$k=0,\ldots, M - 1$}{
        $X_{t_{k + 1}}
        \gets X_{t_k} + \tau_k Y_{t_k}$\;
        $V_{t_{k + 1}}
        \gets N K_{t_{k + 1}}
        ^{\dagger} Y_{t_k}$\;
        $Y_{t_{k + 1}}
        \gets \alpha_{t_k} Y_{t_k}
        + \sqrt{\tau_k}(1 + N^{-2}  \tr( V_{t_{k + 1}}^{\tT} K_{t_{k + 1}} V_{t_{k + 1}})) X_{t_{k + 1}} A
        - \frac{\sqrt{\tau_k}}{N} K_{t_{k + 1}} \nabla f(X_{t_{k + 1}})$\;
    }
\end{algorithm}
Here, the function $\nabla f$ is applied to every row of $X_t$ separately, and $\dagger$ denotes the pseudoinverse.

For comparison, we mention the non-accelerated SVGD algorithm for the generalized bilinear kernel.

\begin{algorithm}[H]
\caption{Stein variational gradient descent particle algorithm with generalized bilinear kernel.} \label{alg:SVGD_bilinear}
\KwData{Kernel parameter matrix $A \in \Sym_+(d)$, number of particles $N \in \N$, number of steps $M \in \N$, step sizes $(\tau_k)_{k = 0}^{M} \subset (0, \infty)$}
\KwResult{Matrix $X_{t_{M}}$, whose rows are particles particles that approximate the target distribution $\pi \sim \exp(-f)$.}
    \For{$k=0,\ldots, M - 1$}{
        $X_{t_{k + 1}}
        \gets X_{t_k} + \frac{\tau_k}{N} \left( K_{t_k} \nabla f(X_{t_k}) + N X_{t_k} A \right)$
    }
\end{algorithm}

\textbf{Gradient restart.}
Analogously to the construction in \cite[Sec.~5.2]{WL2020}, we also try out the gradient restart, that is, if $- \partial_t E(\rho_t) = - g_{\rho_t}(\partial_t \rho_t, G_{\rho_t}^{-1}[\delta E(\rho_t)]) < 0$, we set $Y_k$ to 0.

Using integration by parts and the score-avoiding trick from before, we obtain
\begin{align*}
    - g_{\rho_t}(\partial_t \rho_t, G_{\rho_t}^{-1}[\delta E(\rho_t)])
    & = \langle \partial_t \rho_t, \delta E(\rho_t) \rangle_{L^2(\R^d)} \\
    & = \left\langle \rho_t(\cdot) \int_{\R^d} K(\cdot, y) \rho_t(y) \nabla \Phi_t(y) \d{y}, \nabla \partial E(\rho_t) \right\rangle \\
    & = \int_{\R^d} \nabla \Phi_t(y) \int_{\R^d} \big[ K(x, y) \nabla f(x) - \nabla_2 K(y, x) \big] \rho_t(x) \d{x} \rho_t(y) \d{y} \\
    & = \frac{1}{N^2} \sum_{i, j = 1}^{N} \langle V_j, K(X_i, X_j) \nabla f(X_i) - (\nabla_2 K)(X_j, X_i) \rangle_{\R^d}.
\end{align*}


\subsection{Proof of \texorpdfstring{\cref{lemma:Gaussian-Stein_metric_tensor}}{a lemma}} \label{subsec:proof_Gauss_Stein_tensor} 
\begin{proof}
\begin{enumerate}
    \item     
    We fix $\theta = (\mu, \Sigma) \in \Theta$ and define the scalar product on $T_{\theta} \Theta \times T_{\theta} \Theta$ by
    \begin{equation*}
        \langle (\nu_1, S_1), (\nu_2, S_2) \rangle_{\theta}
        \coloneqq \tr(S_1 S_2) + \nu_1^{\tT} \nu_2.
    \end{equation*}
    The (trivial) tangent bundle of $\Theta$ is $T\Theta
        = \Theta \times \R^d \times \Sym(d)
        \cong \Theta \times \R^{\frac{d(d+1)}{2}}$.
    The pushforward of the inclusion is
    \begin{equation*}
        \d \phi_{\theta}
        = (\phi_*)_{\theta} \colon T_{\theta} \Theta \to T_{\rho_{\theta}} \widetilde{\P}(\R^d), \qquad
        \xi \mapsto \bigg[ x \mapsto \langle \xi, \nabla_{\theta} \rho_{\theta}(x) \rangle_{\theta} \bigg].
    \end{equation*}
    We can form the following commutative diagrams using $\phi_*$ and its pullback $\phi^*$:
    \begin{figure}[H]
        \centering
        \begin{tikzcd}[sep = huge]
            \substack{T_{\theta} \Theta \\ (\nu, S)}\ar[r, "(\phi_*)_{\theta}"] \ar[d, "\tilde{G}_{\theta}"'] \ar[dr, "\Xi_{\theta}"]
            & \substack{T_{\rho_{\theta}} \widetilde{\P}(\R^d) \\ \left\langle (\nu, S), \nabla_{\theta} p_{\theta}(\cdot) \right\rangle} \ar[d, "G_{\rho_{\theta}}"] \\
            \substack{T_{\theta}^* \Theta \\ (m_2, \frac{1}{2} M_1)} 
            & \substack{T_{\rho_{\theta}}^* \widetilde{\P}(\R^d) \\ \Xi_{\theta}(\nu, S) = [\frac{1}{2} (\cdot - \mu)^{\tT} M_1 (\cdot - \mu) + \langle m_2, \cdot \rangle + C]} \ar[l, "(\phi^*)_{\rho_{\theta}}"]
        \end{tikzcd}
        \caption{
        The involved tangent and cotangent spaces, and a typical element contained in them each displayed beneath them.
        We will show that $(\nu, S)$ and $(m_2, M_1)$ are related by $\nu = M_1 \Sigma A \mu + K(\mu, \mu) m_2$ and $S = 2 \Sym\left(\Sigma A\big( \Sigma M_1 + \frac{1}{K(\mu, \mu)}(\nu - M_1 \Sigma A \mu) m_2^{\tT}\big)\right)$.
        Here, $[\cdot]$ denotes the equivalence class in $C^{\infty}(\R^d) / \R$.}
        \label{fig:commute_pushforward}
    \end{figure}
    Note that pullback $\phi^* \colon \Gamma(T^* \widetilde{\P}(\R^d)) \to \Gamma(T^* \Theta)$ is defined on each tangent space via
    \begin{equation*}
        \phi^*(\alpha)|_{T_{\theta} \Theta}(\xi)
        \coloneqq \alpha|_{T_{\rho_{\theta}} \widetilde{\P}(\R^d)}(\phi_*(\xi)), \qquad \theta \in \Theta, \xi \in T_{\theta} \Theta, \alpha \in \Gamma(T^* \widetilde{\P}(\R^d)).
    \end{equation*}
    Evaluating at $\rho_{\theta}$, $\theta \in \Theta$ thus yields $(\phi^*)_{\rho_{\theta}} \colon T_{\rho_{\theta}}^* \widetilde{\P}(\R^d) \to T_{\theta}^* \Theta$, defined by
    \begin{equation*}
        \Phi \mapsto \bigg[ T_{\theta} \Theta \to \R, \qquad \xi \mapsto \langle \Phi, \big((\phi_*)_{\theta}(\xi)\big) \rangle_{L^2(\R^d)} \bigg], 
    \end{equation*}
    since we identified $T_{\rho_{\theta}}^* \widetilde{\P}(\R^d) \cong \C^{\infty}(\R^d) / \R$ in duality with $T_{\rho_{\theta}} \widetilde{\P}(\R^d)$ using the $L^2(\R^d)$ scalar product, so that
    \begin{equation*}
        (\phi^*)_{\rho_{\theta}}(\Phi)
        = \langle \Phi, \nabla_{\theta} \rho_{\theta} \rangle_{L^2(\R^d)}
        = \begin{pmatrix} 
            \langle \Phi, \nabla_{\mu} \rho_{\theta} \rangle_{L^2(\R^d)} \\
            \langle \Phi, \nabla_{\Sigma} \rho_{\theta} \rangle_{L^2(\R^d)}
        \end{pmatrix}
        \in \R^d \times \Sym(d),
    \end{equation*}
    where we integrate each component against $\Phi$.
    This map can be simplified using integration by parts: by \cref{prop:GaussCalcationRules} we have
    \begin{equation*}
        \int_{\R^d} \Phi(x) \nabla_{\mu} p_{\theta}(x) \d{x}
        = - \int_{\R^d} \Phi(x) \nabla_x p_{\theta}(x) \d{x}
        = \int_{\R^d} \nabla \Phi(x) p_{\theta}(x) \d{x}.
    \end{equation*}
    
    \item 
    Since the right diagram commutes in \figref{fig:commute_pushforward}, we have $\tilde{G}_{\theta} = (\phi^*)_{\rho_{\theta}} \circ G_{\rho_{\theta}} \circ (\phi_*)_{\theta}$.
    Since we only have a closed form for $G_{\rho_{\theta}}^{-1}$, we proceed as follows.
    Fix $(\nu, S) \in T_{\theta} \Theta$ and set $\Xi_{\theta}(\nu, S) \coloneqq \big(G_{\rho_{\theta}} \circ (\phi_*)_{\theta}\big)(\nu, S) \in T_{\rho_{\theta}}^* \widetilde{\P}(\R^d)$.
    Then we have
    \begin{equation} \label{eq:helper1}
        G_{\rho_{\theta}}^{-1}[\Xi_{\theta}(\nu, S)] = (\phi_*)_{\theta}(\nu, S).
    \end{equation}
    Aiming to find a general expression for $\tilde{G}_{\theta}$, we now to guess appropriate ansatz for $\Xi_{\theta}(\nu, S)$ from \eqref{eq:helper1}.
    Writing out \eqref{eq:helper1} yields
    \begin{equation*}
        - \nabla \cdot \left((\rho_{\theta}(x) \int_{\R^d} K(x, y) \rho_{\theta}(y) \nabla [\Xi_{\theta}(\nu, S)](y) \d{y}\right)
        = \langle (\nu, S), \nabla_{\theta} \rho_{\theta}(x) \rangle_{\theta} \qquad \forall x \in \R^d.
    \end{equation*}
    For easy of notation, we set $\Psi \colon \R^d \to \R^d$, $x \mapsto \int_{\R^d} K(x, y) \rho_{\theta}(y) \nabla [\Xi_{\theta}(\nu, S)](y) \d{y}$.
    The product rule for the left-hand side together with \cref{prop:GaussCalcationRules} yields all for $x \in \R^d$ that
    \begin{align*}
        - \nabla \cdot \left(\rho_{\theta}(x) \Psi(x)\right)
        & = \langle \Psi(x), \nabla \rho_{\theta}(x) \rangle_{\R^d}
        + \rho_{\theta}(x) \nabla \cdot \Psi(x) \\
        & = \rho_{\theta}(x) \left( \Psi(x)^{\tT} \Sigma^{-1}(x - \mu) - \nabla \cdot \Psi(x) \right).
    \end{align*}
    On the other hand, by \cref{prop:GaussCalcationRules}, the right-hand side is equal to
    \begin{align*}
        \langle (\nu, S), \nabla_{\theta} \rho_{\theta}(x) \rangle_{\theta}
        & = \rho_{\theta}(x) \left[ \tr\left(\frac{1}{2} \left[\Sigma^{-1} (x - \mu) (x - \mu)^{\tT} - \id\right] \Sigma^{-1} S\right) 
        + \nu^{\tT} \Sigma^{-1} (x - \mu)\right] \\
        & = \rho_{\theta}(x) \left[ \left(\frac{1}{2} (x - \mu)^{\tT} \Sigma^{-1} S + \nu^{\tT} \right) \Sigma^{-1} (x - \mu) - \frac{1}{2} \tr(\Sigma^{-1} S) \right].
    \end{align*}
    Since $\rho_{\theta}(x) > 0$ for all $x \in \R^d$, we can divide by it, and we have reduced \eqref{eq:helper1} to
    \begin{equation} \label{eq:helper2}
        \left(\frac{1}{2} (x - \mu)^{\tT} \Sigma^{-1} S + \nu^{\tT} \right) \Sigma^{-1} (x - \mu) - \frac{1}{2} \tr(\Sigma^{-1} S)
        = \Psi(x)^{\tT} \Sigma^{-1}(x - \mu) - \nabla \cdot \Psi(x).
    \end{equation}
    Since the left-hand side is quadratic in $x$, we want to choose $\Xi_{\theta}(\nu, S)$ such that $\Psi$ is linear.
    Noting that $\Psi$ only depends on the gradient $\nabla \Xi_{\theta}(\nu, S)$ and that the kernel is linear in each argument, we should choose the most simple ansatz, namely, that $\nabla \Xi_{\theta}(\nu, S)$ is a linear function, i.e., $[\nabla \Xi_{\theta}(\nu, S)](y) = M_1 (y - \mu) + m_2$, for $M_1 \in \Sym(d)$, $m_2 \in \R^d$.
    Then $\Psi$ is linear:
    \begin{equation} \label{eq:computing_Psi}
    \begin{aligned}
        \Psi(x)
        & = \int_{\R^d} \rho_{\theta}(y) (M_1 y + m_2 - M_1 \mu) y^{\tT} \d{y} A x
        + \int_{\R^d} \rho_{\theta}(y) (M_1 y + m_2 - M_1 \mu) \d{y} \\
        & = \left( M_1 \int_{\R^d} \rho_{\theta}(y) y y^{\tT} \d{y} + (m_2 - M_1 \mu) \left(\int_{\R^d} \rho_{\theta}(y) y \d{y}\right)^{\tT} \right) A x
        + \left( M_1 \mu + m_2 - M_1 \mu\right) \\
        & = \big(M_1 \Sigma + m_2 \mu^{\tT}\big) A x + m_2
    \end{aligned}
    \end{equation}
    and, consequently, 
    \begin{equation*}
        \nabla \cdot \Psi
        = \tr\big( \big(M_1 \Sigma + m_2 \mu^{\tT}\big) A \big)
        = \tr(M_1 \Sigma A) + \mu^{\tT} A m_2.
    \end{equation*}
    Hence, for this ansatz, the right-hand side of \eqref{eq:helper2} becomes
    \begin{equation*}
        x^{\tT} A \big(\Sigma M_1 + \mu m_2^{\tT} \big) \Sigma^{-1} (x - \mu)
        + m_2^{\tT} \Sigma^{-1} (x - \mu)
        - \tr(M_1 \Sigma A) - \mu^{\tT} A m_2.
    \end{equation*}
    Adding a productive zero, replacing the matrix of the quadratic term by its symmetrized version, and collecting like terms in this expression yields our final expression for the left-hand side of \eqref{eq:helper2}:
    \begin{align*}
        & \ \frac{1}{2} (x - \mu)^{\tT} \left[ A \big(\Sigma M_1 + \mu m_2^{\tT} \big) \Sigma^{-1} 
        +  \Sigma^{-1} \big( M_1 \Sigma + m_2 \mu^{\tT} \big) A
        \right] (x - \mu) \\
        & \qquad + \left[ \mu^{\tT} A \big(\Sigma M_1 + \mu m_2^{\tT} \big) + m_2^{\tT} \right] \Sigma^{-1} (x - \mu) - \tr(M_1 \Sigma A) - \mu^{\tT} A m_2.
    \end{align*}
    Now, both sides of \eqref{eq:helper2} are quadratic functions in $x$ of the form $(x - \mu)^{\tT} B_1 (x - \mu) + b_2^{\tT} (x - \mu) + b_3$ for $B_1 \in \Sym(d)$, $b_2 \in \R^d$, $b_3 \in \R$ and we can compare these coefficients.
    This yields 
    \begin{equation*}
        \begin{cases}
            B_1
            = \frac{1}{2} \Sigma^{-1} S \Sigma^{-1}
            \overset{!}{=} \frac{1}{2} A \big( \Sigma M_1 + \mu m_2^{\tT} \big) \Sigma^{-1}
            + \frac{1}{2} \Sigma^{-1} \big( M_1 \Sigma + m_2 \mu^{\tT} \big) A \\
            b_2
            = \Sigma^{-1} \nu 
            \overset{!}{=} \Sigma^{-1} (M_1 \Sigma + \mu m_2^{\tT}) A \mu + \Sigma^{-1} m_2 \\
            b_3
            = - \frac{1}{2} \tr(\Sigma^{-1} S)
            \overset{!}{=} - \tr(M_1 \Sigma A) - \mu^{\tT} A m_2.
        \end{cases}
    \end{equation*}
    In the second equation, we multiply by $\Sigma$ from the left on both sides to obtain the following simplified system:
    \begin{equation} \label{eq:reduced_system}
        \begin{cases}
            \Sigma^{-1} S \Sigma^{-1}
            = A \big( \Sigma M_1 + m_2 \mu^{\tT} \big) \Sigma^{-1}
            + \Sigma^{-1} \big( M_1 \Sigma + m_2 \mu^{\tT} \big) A \\
            \nu 
            = M_1 \Sigma  A \mu + K(\mu, \mu) m_2 \\
            \frac{1}{2} \tr(\Sigma^{-1} S)
            = \tr(M_1 \Sigma A) + \mu^{\tT} A m_2.
        \end{cases}
    \end{equation}
    The first equation in \eqref{eq:reduced_system} can be rewritten as
    \begin{equation} \label{eq:Lyapunov}
        S - \Sigma A \mu m_2^{\tT} - m_2 \mu^{\tT} A \Sigma
        = \Sigma A \Sigma M_1 + M_1 \Sigma A \Sigma.
    \end{equation}
    This can be rewritten as $P M_1 + M_1 P = Q$, where $P, Q \in \Sym(d)$ are given by
    \begin{equation*}
        P \coloneqq \Sigma A \Sigma, \qquad
        Q \coloneqq S - \Sigma A \mu m_2^{\tT} - m_2 \mu^{\tT} A \Sigma,
    \end{equation*}
    such that there exists a unique solution $M_1 \in \Sym(d)$.
    The second equation in \eqref{eq:reduced_system} can be solved for an $M_1$-dependent expression of $m_2$: we have
    \begin{equation} \label{eq:m2}
        m_2
        = \frac{1}{K(\mu, \mu)} (\nu - M_1 \Sigma A \mu)
    \end{equation}
    We can check that this already implies that the third equation from \eqref{eq:Lyapunov} is fulfilled:
    \begin{align*}
        \frac{1}{2} \tr(\Sigma^{-1} S)
        & = \frac{1}{2} \tr\big(\Sigma^{-1} (\Sigma A \Sigma M_1 + M_1 \Sigma A \Sigma + \Sigma A \mu m_2^{\tT} + m_2 \mu^{\tT} A \Sigma) \big) \\
        & = \frac{1}{2} \tr(A \Sigma M_1 + M_1 \Sigma A)
        + \frac{1}{2} \tr(A \mu m_2^{\tT} + m_2 \mu^{\tT} A)
        = \tr(A \Sigma M_1) + \mu^{\tT} A m_2,
    \end{align*}
    as desired.
    (This is in line with the fact that the equality we considered is in the tangent space of $\widetilde{\P}(\R^d)$, which consists of functions integrating to 0, so their constant terms must be equal anyway.)

    \item
    We can now finally compute $\tilde{G}_{\theta}$ and show that
    \begin{equation*}
        \tilde{G}_{\theta}^{-1}\left(m_2, \frac{1}{2} M_1\right)
        = \left( \nu, S \right).
    \end{equation*}
    Noting that $x \mapsto \frac{1}{2} (x - \mu)^{\tT} M_1 (x - \mu) + m_2^{\tT} x$ is an antiderivative of $\nabla \Xi_{\theta}(\nu, S)$, we have
    \begin{equation*}
        \tilde{G}_{\theta}(\nu, S)
        = (\phi^*)_{\rho_{\theta}}(\Xi_{\theta}(\nu, S))
        = (\phi^*)_{\rho_{\theta}}\left(\left[y \mapsto \frac{1}{2} (y - \mu)^{\tT} M_1 (y - \mu) + m_2^{\tT} y\right]\right),
    \end{equation*}
    where $M_1, m_2$ are uniquely determined by \eqref{eq:Lyapunov} and \eqref{eq:m2} and $[\cdot]$ denotes the equivalence class in $\C^{\infty}(\R^d) / \R$.

    The first entry of $\tilde{G}_{\theta}(\nu, S)$ is given by
    \begin{equation*}
        \int_{\R^d} ( M_1 (x - \mu) + m_2 ) p_{\theta}(x) \d{x}
        = M_1 \int_{\R^d} (x - \mu) p_{\theta}(x) \d{x} + m_2
        = m_2.
    \end{equation*}
    The second one is
    \begin{align*}
        \langle \Phi, \nabla_{\Sigma} \rho_{\theta} \rangle_{L^2(\R^d)}
        & = \frac{1}{2} \Sigma^{-1} \int_{\R^d} \left(\frac{1}{2} (x - \mu)^{\tT} M_1 (x - \mu) + m_2^{\tT} x \right) \\
        & \qquad \qquad \cdot \left( (x - \mu) (x - \mu)^{\tT} \Sigma^{-1} - \id\right) \rho_{\theta}(x) \d{x}.
    \end{align*}
    We can separate the integral into two summands, the first one being
    \begin{align*}
        I_1
        & \coloneqq \frac{1}{2} \int_{\R^d} \left( (x - \mu) (x - \mu)^{\tT} \Sigma^{-1} - \id\right) (x - \mu)^{\tT} M_1 (x - \mu) \rho_{\theta}(x) \d{x} \\
        & = \frac{1}{2} \int_{\R^d} (x - \mu) (x - \mu)^{\tT} M_1 (x - \mu) (x - \mu)^{\tT} \rho_{\theta}(x) \d{x} \Sigma^{-1} \\
        & \qquad - \frac{1}{2} \int_{\R^d} (x - \mu)^{\tT} M_1 (x - \mu) \rho_{\theta}(x) \d{x} \cdot \id \\
        \overset{(\star)}&{=} \frac{1}{2} \left( \Sigma \int_{\R^d} \left[2 (x_j - \mu_j) M_1 (x - \mu) + (x - \mu)^{\tT} M_1 (x - \mu) \cdot e_j\right] \rho_{\theta}(x) \right)_{j = 1}^{d} \Sigma^{-1} \\
        & \qquad - \frac{1}{2} \int_{\R^d} (x - \mu)^{\tT} M_1 (x - \mu) \rho_{\theta}(x) \d{x} \cdot \id \\
        \overset{(\star)}&{=} \left( \Sigma M_1 \int_{\R^d} (x - \mu) (x_j - \mu_j) \rho_{\theta}(x) \d{x} + \frac{1}{2} \Sigma \int_{\R^d} (x - \mu)^{\tT} M_1 (x - \mu) \rho_{\theta}(x)  \cdot e_j \right)_{j = 1}^{d} \Sigma^{-1} \\
        & \qquad - \frac{1}{2} \int_{\R^d} (x - \mu)^{\tT} M_1 (x - \mu) \rho_{\theta}(x) \d{x} \cdot \id \\
        & = \frac{1}{2} \left(2 \Sigma M_1 \Sigma  + \tr(M_1 \Sigma) \Sigma \right) \Sigma^{-1}
        - \frac{1}{2} \tr(\Sigma M_1) \cdot \id \\
        & = \Sigma M_1 + \frac{1}{2} \tr(M_1 \Sigma) \cdot \id - \frac{1}{2} \tr(\Sigma M_1) \cdot \id
        = \Sigma M_1,
    \end{align*}
    where in $(\star)$ we have used Stein's Lemma \cite[Lemma~6.20]{FVRS22},
    \begin{equation*}
        \E[(X - \mu) g(X)] = \Sigma \E[\nabla g(X)], \qquad g \colon \R^d \to \R \text{ differentiable}, \ X \sim \NN(\mu, \Sigma),
    \end{equation*}
    for the functions $g_j(x) \coloneqq (x - \mu)^{\tT} M_1 (x - \mu) (x_j - \mu_j)$, which have the gradients $\nabla g_j(x) = 2 (x_j - \mu_j) M_1 (x - \mu) + (x - \mu)^{\tT} M_1 (x_j - \mu_j)$ (where $e_j$ denotes the $j$-th unit vector), together, for the first summand and \cref{prop:GaussCalcationRules} for the second summand.
    To calculate $\int_{\R^d} (x - \mu) (x_j - \mu_j) \rho_{\theta}(x) \d{x}$, we again use Stein's lemma, this time with $g_j(x) \coloneqq e_j^{\tT}(x - \mu)$.

    The second summand of the second term of $\tilde{G}_{\theta}(\nu, S)$ is
    \begin{align*}
        I_2
        & \coloneqq \int_{\R^d} m_2^{\tT} x \big( (x - \mu) (x - \mu)^{\tT} \Sigma^{-1} - \id) \rho_{\theta}(x) \d{x} \\
        & = \int_{\R^d} (x - \mu) (x - \mu)^{\tT} \Sigma^{-1} m_2^{\tT} x \rho_{\theta}(x) \d{x}
        - m_2^{\tT}\int_{\R^d} x \rho_{\theta}(x) \d{x} \\
        \overset{(\star)}&{=} \Sigma \int_{\R^d} \left( \Sigma^{-1} (x - \mu) m_2^{\tT} + m_2^{\tT} x \Sigma^{-1} \right) \rho_{\theta}(x) \d{x} - m_2^{\tT} \mu \\
        & = 0,
    \end{align*}
    where in $(\star)$ we used Stein's lemma with $g(x) = (x - \mu)^{\tT} \Sigma^{-1} m_2^{\tT} x$, whose gradient is given by $\nabla g(x) = \Sigma^{-1} (x - \mu) m_2^{\tT} + m_2^{\tT} x \Sigma^{-1}$, for all $x \in \R^d$.

    In summary, the second entry of $\tilde{G}_{\theta}(\nu, S)$ is given by
    \begin{equation*}
        \frac{1}{2} \Sigma^{-1}(I_1 + I_2)
        = \frac{1}{2} M_1.
    \end{equation*}

    \item 
    We now derive the expression for the Riemannian metric.
    By definition of the pullback metric, we have for $\xi, \eta \in T_{\theta} \Theta$ that
    \begin{align*}
        \tilde{g}_{\theta}(\xi, \eta)
        & = g_{\rho_{\theta}}\big(\d \phi_{\theta}(\xi), \d \phi_{\theta}(\eta)\big)
        = \int_{\R^d} \big[\Xi_{\theta}(\xi)\big](x), G_{\rho_{\theta}}^{-1}\big[ \Xi_{\theta}(\eta)\big](x) \d{x} \\
        & = \int_{\R^d} \left\langle \nabla \big[\Xi_{\theta}(\xi)\big](x), p_{\theta}(x) \int_{\R^d} (x^{\tT} A y + 1) \nabla\big[ \Xi_{\theta}(\eta)\big](y) p_{\theta}(y) \d{y} \right\rangle \d{x} \\
        & = \int_{\R^d} \left\langle M_1 (x - \mu) + m_2, p_{\theta}(x) \int_{\R^d} (x^{\tT} A y + 1) \big( \widetilde{M}_1 (y - \mu) + \widetilde{m}_2 \big) p_{\theta}(y) \d{y} \right\rangle \d{x} \\
        & = \int_{\R^d} \left\langle M_1 (x - \mu) + m_2, \widetilde{M}_1 \Sigma A x + K(x, \mu) \widetilde{m}_2 \right\rangle p_{\theta}(x) \d{x} \\
        & = \tr(M_1 \widetilde{M}_1 \Sigma A \Sigma) + (\widetilde{m}_2^{\tT} M_1 + m_2^{\tT} \widetilde{M}_1) \Sigma A \mu
        + K(\mu, \mu) m_2^{\tT} \widetilde{m}_2,
    \end{align*}
    where we used integration by parts in the second line, and \eqref{eq:computing_Psi} in the fourth line.

    The associated tangent-cotangent isomorphism is indeed $\tilde{G}_{\theta}$, since by plugging in the formulas from above we obtain
    \begin{equation*}
        \left\langle \tilde{G}_{\theta}(\xi), \eta \right\rangle
        = \left\langle (m_2, \frac{1}{2} M_1), \eta \right\rangle
        = \tilde{g}_{\theta}(\xi, \eta).
    \end{equation*}
\end{enumerate}
\end{proof}

\subsection{Proof of \texorpdfstring{\cref{thm:3.1}}{a theorem}} \label{subsec:Thm3.1_Proof}

\begin{proof}
    \begin{enumerate}

    \item 
    We prove this statement analogously to \cite{LGBP2024}.
    Consider $\tilde{E} \colon \Theta \to \R$, $\theta \mapsto E(\rho_{\theta}) = \KL(\NN(\theta) \mid \NN(b, Q))$.
    Then by \cref{lemma:Gaussian-Stein_metric_tensor}, the gradient flow is
    \begin{equation} \label{eq:SVGD_in_Gaussians}
        \begin{aligned}
            \dot{\theta_t}
            & = - G_{\theta_t}^{-1}(\nabla_{\theta_t} \tilde{E}(\theta_t)) \\
            & = \begin{pmatrix}
                (\id - Q^{-1} \Sigma_t) A \mu_t
                - K(\mu_t, \mu_t) Q^{-1}(\mu_t - b) \\
                2 \Sym\big(\Sigma_t A(Q
                - \Sigma_t
                - \mu_t (\mu_t - b)^{\tT}) Q^{-1}\big)
            \end{pmatrix}. 
        \end{aligned}
    \end{equation}

    \item 
    The right-hand side of \eqref{eq:SVGD_in_Gaussians} is continuously differentiable with respect to $\theta_t$ and hence locally Lipschitz, so the Picard-Lindelöff theorem asserts the existence of a solution on $[0, \eps)$ for some $\eps > 0$.
    The global existence and the fact that Gaussians are preserved follows exactly as in \cite[p.~39--40]{LGBP2024}.

    We now prove that $\rho_t \weak \rho^*$.
    For $t > 0$ we have by \eqref{eq:Stein_gradient_flow_Gaussians}
    \begin{align*}
        \dot{\tilde{E}}(\mu_t, \Sigma_t)
        & = \frac{1}{2} \tr\left( (Q^{-1} - \Sigma_t^{-1}) \dot{\Sigma}_t\right)
        + (\mu_t - b)^{\tT} Q^{-1} \dot{\mu}_t \\
        & = \frac{1}{2} \tr\bigg( (Q^{-1} - \Sigma_t^{-1})
        \left( 2 \Sym\big(\Sigma_t A (Q - \Sigma_t - \mu_t (\mu_t - b)^{\tT}) Q^{-1}\big)\right) \\
        & \qquad + (\mu_t - b)^{\tT} Q^{-1} \left( (\id - Q^{-1} \Sigma_t) A \mu_t
        - K(\mu_t, \mu_t) Q^{-1}(\mu_t - b) \right) \\
        & = \frac{1}{2} \tr\left( 2 Q^{-1} \Sym(\Sigma_t A)
            - 2 Q^{-2} \Sigma_t A \Sigma_t
            - 2 \Sym\big(Q^{-2} \Sigma_t A \mu_t (\mu_t - b)^{\tT}\big)\right) \\
        & \qquad + \frac{1}{2}\tr\left(
            - 2 A 
            + 2 \Sym\big(A \Sigma_t Q^{-1}\big)
            + 2 \Sym\big(A \mu_t (\mu_t - b)^{\tT} Q^{-1}\big) \right) \\
        & \qquad + 
        (\mu_t - b)^{\tT} Q^{-1} (\id - Q^{-1} \Sigma_t) A \mu_t
        - (\mu_t - b)^{\tT} Q^{-2} (\mu_t - b) \cdot K(\mu_t, \mu_t) \\
        & = \tr\left( 2 Q^{-1} \Sym(\Sigma_t A)
            - Q^{-2} \Sigma_t A \Sigma_t\right)
            - 2 (\mu_t - b)^{\tT} Q^{-2} \Sigma_t A \mu_t - \tr(A)  \\
        & \qquad + 2
            (\mu_t - b)^{\tT} Q^{-1} A \mu_t
            - (\mu_t - b)^{\tT} Q^{-2} (\mu_t - b) \mu_t^{\tT} A \mu_{t}
            - (\mu_t - b)^{\tT} Q^{-2} (\mu_t - b)\\
        & = - \tr( (B_t + C_t)^{\tT} (B_t + C_t))
        - (\mu_t - b)^{\tT} Q^{-2} (\mu_t - b)
        \le 0,
    \end{align*}
    (in fact, the inequality is strict whenever $\mu_t \ne b$) where
    \begin{equation*}
        B_t 
        \coloneqq (Q^{-1} - \Sigma_t^{-1}) \Sigma_t A^{\frac{1}{2}}
        \qquad \text{and} \qquad
        C_t
        \coloneqq Q^{-1} (\mu_t - b) \mu_t^{\tT} A^{\frac{1}{2}}.
    \end{equation*}
    As long as $\mu_t \ne b$, the functional $\tilde{E}$ thus is strictly decreasing.
    As in \cite[Proof~of~Thm.~3.1]{LGBP2024}, since $t \mapsto \dot{\tilde{E}}(\mu_{t}, \Sigma_{t})$ is uniformly continuous, Barb\u{a}lat's Lemma \cite[Thm.~1]{FW2016} shows that $\dot{\tilde{E}}(\mu_t, \Sigma_t) \to 0$ and thus $\mu_t \to b$ and $\Sigma_t \to Q$ for $t \to \infty$, and thus the claimed weak convergence follows.

    \textbf{Convergence rates.}
    Set $\eta_t \coloneqq Q^{-1}(\mu_t - b)$ and $S_t \coloneqq (Q^{1} - \Sigma_t^{1})\Sigma_t$.
    Analogously to \cite[p.~41]{LGBP2024}, we then have
    \begin{align*}
        - \dot{\tilde{E}}(\mu_t, \Sigma_t)
        & = \tr\left(A (S_t + \eta_t \mu_t^{\tT})^{\tT} (S_t + \eta_t \mu_t^{\tT})\right) + \eta_t^{\tT} \eta_t \\
        & = \begin{pmatrix}
            \vec(B_t) \\ \eta_t
        \end{pmatrix}^{\tT}
        \begin{pmatrix}
            A \otimes \id_{d \times d}         & (A \mu_t) \otimes \id_{d \times d} \\
            (A \mu_t) \otimes \id_{d \times d} & K(\mu_t, \mu_t) \id_{d \times d} 
        \end{pmatrix}   
        \begin{pmatrix}
            \vec(B_t) \\ \eta_t
        \end{pmatrix} \\
        & = \begin{pmatrix}
            \vec(B_t) \\ \eta_t
        \end{pmatrix}^{\tT}
        \begin{pmatrix}
            A \otimes \id_{d \times d}     & (A b) \otimes \id_{d \times d} \\
            (A b) \otimes \id_{d \times d} & K(b, b) \id_{d \times d} 
        \end{pmatrix}   
        \begin{pmatrix}
            \vec(B_t) \\ \eta_t
        \end{pmatrix} 
        + o(\| S_t \| + \| \eta_t \|)
    \end{align*}
    for $t \to \infty$ (since $\mu_t \to b$ and $\Sigma_t \to Q$).
    In the second line we used $\tr(B^{\tT} C) = \vec(B)^{\tT} \vec(C)$, $\tr(A X^{\tT} X) \vec(X)^{\tT} (A \otimes \id) \vec(X)$, $\vec(S + \eta \mu^{\tT}= \vec(S) + \mu \otimes \eta$, and the mixed product property of the Kronecker product.  

    On the other hand (see \cite[p.~41]{LGBP2024}, we have
    \begin{equation*}
        \widetilde{E}(\mu_t, \Sigma_t)
        = \frac{1}{4} \begin{pmatrix}
            \vec(B_t) \\ \eta_t
        \end{pmatrix}^{\tT}
        \begin{pmatrix}
            I_{d^2 \times d^2}  & 0 \\
            0                   & 2 Q
        \end{pmatrix}   
        \begin{pmatrix}
            \vec(B_t) \\ \eta_t
        \end{pmatrix} + o(S_t), \qquad t \to \infty.
    \end{equation*}
    Aiming to use Grönwall's inequality, we thus want to prove that for all $\eps > 0$, there exists a $T > 0$ such that $- \dot{\tilde{E}}(\mu_t, \Sigma_t) \ge 4 (\gamma - \eps) \tilde{E}(\mu_t, \Sigma)$ for all $t > T$, where $\gamma > 0$ is given in the theorem.
    Hence, it suffices to show
    \begin{equation*}
        \begin{pmatrix}
            A \otimes \id_{d \times d}     & (A b) \otimes \id_{d \times d} \\
            (A b) \otimes \id_{d \times d} & K(b, b) \id_{d \times d} 
        \end{pmatrix}
        \succeq \gamma \begin{pmatrix}
            I_{d^2 \times d^2}  & 0 \\
            0                   & 2 Q
        \end{pmatrix},
    \end{equation*}
    which by conjugation with $P \coloneqq \diag\big( \id_{d \times d}, (2 Q)^{-\frac{1}{2}}\big)$ is equivalent to
    \begin{equation*}
        B \coloneqq \begin{pmatrix}
            A \otimes \id_{d \times d} & \frac{1}{\sqrt{2}} (A b) \otimes Q^{-\frac{1}{2}} \\
            \frac{1}{\sqrt{2}} (b^{\tT} A) \otimes Q^{-\frac{1}{2}} & \frac{1}{2} K(b, b) Q^{-1}
        \end{pmatrix}
        \succeq \gamma \id_{(d^2 + d) \times (d^2 + d)},
    \end{equation*}
    so $\gamma$ is the smallest eigenvalue of the matrix on the left, as claimed.
    By Grönwall's lemma, we have $\tilde{E}(\mu_t, \Sigma_t) \in O(\exp\left(- 4 (\gamma - \eps) t\right)$ as $t \to \infty$ for all $\eps > 0$.

    For $d = 1$, the eigenvalues of $B$ are
    \begin{equation*}
        \lambda_{\pm}(B)
        = \frac{A}{2} + \frac{1 + A b^2 \pm \sqrt{(1 + A b^2 + 2 A Q)^2 - 8 A Q}}{4 Q}.
    \end{equation*}
    We now provide a lower bound on $\gamma$, just like in \cite[pp.~41--42]{LGBP2024}.
    First, let $b = 0$.
    Then, the matrix $B$ is diagonal and its smallest eigenvalue is given by $\gamma = \min\left\{ \lambda_{\min}(A), \frac{1}{\lambda_{\max}(Q)}\right\}$.
    For $b \ne 0$, consider for $u > 0$ the matrix
    \begin{equation*}
        B_u
        \coloneqq \begin{pmatrix}
            \frac{1}{1 + u} A \otimes \id_{d \times d} &  \frac{1}{\sqrt{2}} (A b) \otimes Q^{-\frac{1}{2}} \\
            \frac{1}{\sqrt{2}} (b^{\tT} A) \otimes Q^{-\frac{1}{2}} & \frac{1 + u}{2}(b, b) Q^{-1}.
        \end{pmatrix}
    \end{equation*}
    The upper left block is positive definite and its Schur complement is exactly zero, so that $B_u$ positive semidefinite.
    Rearranging, we thus obtain that
    \begin{equation*}
        B \succeq \diag\left( \frac{u}{1 + u} A \otimes \id_{d \times d}, \frac{1}{2}(1 - u b^{\tT} A b) Q^{-1}\right)
        \eqqcolon \Omega_u.
    \end{equation*}
    The smallest eigenvalue of $\Omega_u$ is
    \begin{equation*}
        \lambda_{\min}(\Omega_u)
        = \min\left\{ \frac{u}{1 + u} \lambda_{\min}(A), \frac{1 - u b^{\tT} A b}{2 \lambda_{\max}(Q)}\right\}.
    \end{equation*}
    Note that the two functions of $u$ over which the minimum is taken are strictly increasing and strictly decreasing, respectively, hence
    \begin{align*}
        u^*
        & \coloneqq \argmax_{u > 0} \lambda_{\min}(\Omega_u) \\
        & = \frac{\sqrt{(2 \lambda_{\min}(A) \lambda_{\max}(Q) + b^{\tT} A b - 1)^2 + 4 b^{\tT} A b} - (2 \lambda_{\min}(A) \lambda_{\max}(Q) + b^{\tT} A b - 1)}{2 b^{\tT} A b}.
    \end{align*}
    is equal to their intersection point.
    For this choice of $u = u^*$ we have
    \begin{align*}
        \gamma
        & \ge \lambda_{\min}(\Omega_{u^*})
        = \frac{u^*}{1 + u^*} \lambda_{\min}(A) \\
        & = \frac{2}{K(b, b) + 2 \lambda_{\min}(A) \lambda_{\max}(Q) + \sqrt{\big(K(b, b) + 2 \lambda_{\min}(A) \lambda_{\max}(Q)\big)^2 - 8 \lambda_{\min}(A) \lambda_{\max}(Q)}} \\
        & > \frac{1}{K(b, b) + 2 \lambda_{\min}(A) \lambda_{\max}(Q)}.
    \end{align*}

    \item
    The validity of this statement follows like the proof of \cite[Thm.~3.2]{LGBP2024}: since $\Sigma_0$ and $Q$ are symmetric positive definite matrices that commute, they are simultaneously diagonalizable.
    Hence, there exists an orthogonal matrix $V \in \R^{d \times d}$ and two diagonal matrices $D_Q = (q_1, \ldots, q_d)$ and $D_{\Sigma_0} = (s_1, \ldots, s_d)$ such that $Q = V D_Q V^{\tT}$ and $\Sigma_0 = V D_{\Sigma_0} V^{\tT}$.

    We posit the ansatz $\Sigma_t = V \diag(\lambda_k(t)) V^{\tT}$.
    Then we have $\dot{\lambda}_k(t) = 2 \lambda_k(t) (a_k) - 2 \lambda_k(t)^2 \frac{1}{q_k})$ for all $k \in \{ 1, \ldots, d \}$, where $A = \diag(a_1, \ldots, a_d)$.
    This quadratic ODE is a logistic equation \cite[Sec.~1.3]{HS1998} with rate $2 a_k$ and carrying capacity $\tilde{\lambda}_k$, which appears in mathematical ecology.
    The unique solution is $\lambda_k(t)^{-1} = \frac{1}{q_k} + e^{- 2 a_k t} \left( \frac{1}{s_k} - \frac{1}{q_k} \right)$.
    Hence, $\Sigma_t^{-1} = V \diag(\lambda_k(t)^{-1}) V^{\tT} =  Q^{-1} + e^{- 2 t A}(\Sigma_0^{-1} - Q^{-1})$

    \item 
    The Hamiltonian is
    \begin{gather*}
        H \colon \big(\R^d \times \Sym_+(d)\big) \times \big(\R^d \times \Sym(d)\big) \to \R, \\
        \big((\mu, \Sigma), (\nu, S)\big)
        \mapsto \frac{1}{2} \left\langle (\nu, S), G_{(\mu, \Sigma)}^{-1}(\nu, S) \right\rangle_{\R^d \times \Sym(d)} + \KL\big(\NN(\mu, \Sigma) \mid \NN(b, Q)\big).
    \end{gather*}
    Hence, 
    \begin{equation} \label{eq:Hamiltonian_ASVGD_Gaussians}
    \begin{aligned}
        H((\mu, \Sigma), (\nu, S))
        & = \frac{1}{2} \nu^{\tT} (2 S \Sigma A \mu + K(\mu, \mu) \nu)
        + \frac{1}{2} \tr(S \cdot 2 \Sym(\Sigma A (2 \Sigma S + \mu \nu^{\tT}))) \\
        & \qquad + \frac{1}{2} \left( \tr(Q^{-1} \Sigma) - d + (\mu - b)^{\tT} Q^{-1} (\mu - b) + \ln\left(\frac{\det(Q)}{\det(\Sigma)}\right)\right).
    \end{aligned}
    \end{equation}
    Using \cref{prop:functionalderivative_KL}, we obtain
    \begin{equation*}
        \partial_{(\mu, \Sigma)} H((\mu, \Sigma), (\nu, S))
        =  \begin{pmatrix}
            2 A \Sigma S \nu + A \mu \| \nu \|_2^2 + Q^{-1} (\mu - b) \\
            2 S \nu \mu^{\tT} A + 4 \Sym(S^2 \Sigma A) + \frac{1}{2} (Q^{-1} - \Sigma^{-1})
        \end{pmatrix},
    \end{equation*}
    and
    \begin{equation*}
        \partial_{(\nu, S)} H((\mu, \Sigma), (\nu, S))
        = \begin{pmatrix}
            2 S \Sigma A \mu + K(\mu, \mu) \nu \\
            \nu \mu^{\tT} A \Sigma
            + \Sym\big(\Sigma A (2 \Sigma S + \mu \nu^{\tT})\big)
            + 2 \Sym(\Sigma A \Sigma S)
        \end{pmatrix}.
    \end{equation*}
    Next, we show that $\Sigma_t$ remains positive definite for all $t > 0$.
    First, we show that the Hamiltonian \eqref{eq:Hamiltonian_ASVGD_Gaussians} is decreasing in time along the damped Hamiltonian flow \eqref{eq:ASVGD_Gaussians}.

    For the sake of brevity, we write $q \coloneqq (\mu, \Sigma)$ and $p \coloneqq (\nu, S)$ and $H(q, p) \coloneqq K(q, p) + P(q)$.
    Then, for any $t > 0$ we have
    \begin{align*}
        \frac{\d}{\d t} H(q_t, p_t)
        & = \partial_{q_t} H(q_t, p_t) \dot{q}_t
        + \partial_{p_t} H(q_t, p_t) \dot{p}_t \\
        & = \partial_{q_t} H(q_t, p_t) \partial_{p_t} H(q_t, p_t)
        + \partial_{p_t} H(q_t, p_t) (- \alpha_t p_t - \partial_{q_t} H(q_t, p_t)) \\
        & = - \alpha_t p_t \partial_{p_t} K(q_t, p_t).
    \end{align*}
    Since $K$ is homogeneous of degree two in $p$, Euler's theorem implies that $p \partial_p K(q, p) = 2 K(q, p)$, so that the last line becomes $- 2 \alpha_t K(q, p)$.
    Since $\alpha_t \ge 0$ and the kinetic part $K$ is nonnegative, this proves dissipation.
    Exactly as in \cite[Proof~of~Thm.~3.5]{LGBP2024}, it now follows that the smallest eigenvalue of $\Sigma_t$ is bounded below by a positive constant.
    
    The fact that the solution of the Hamiltonian dynamics in the Gaussian submanifold coincides with the solution of the Hamiltonian dynamics on the full density manifold follows again like in \cite[p.~44]{LGBP2024}.

    \end{enumerate}
\end{proof}

\subsection{Proof of \texorpdfstring{\cref{lemma:optimalA_1D}}{lemma 5.1}} \label{subsec:ProofLemma5.1}
\begin{proof}
    Consider the gradient flow \eqref{eq:SVGD_in_Gaussians}.
    Because - to the best of our knowledge - even in one dimension this coupled system of polynomial ODEs does not admit a closed form, we now linearize this system at equilibrium (compare \cite[p.~605,~Subsec.~10.3]{OS2018}).     
    That is, we consider $\mu_t = b + \eps \Delta \mu$ and $\Sigma_t = Q + \eps \Delta \Sigma$ for small $\eps \in \R$, motivated by the fact that $\mu_t \to b$ and $\Sigma_t \to Q$ for $t \to \infty$, and ignore all higher order terms.
    We get
    \begin{align*}
        \eps \begin{pmatrix}
            \dot{\Delta \mu} \\
            \dot{\Delta \Sigma}
        \end{pmatrix}
        = \begin{pmatrix}
            - \eps Q^{-1} \Delta \Sigma A b
            - \eps b^{\tT} A b Q^{-1} \Delta \mu
            - \eps a Q^{-1} \Delta \mu \\
            - \eps Q A \Delta \Sigma Q^{-1} 
            - \eps Q A b (\Delta \mu)^{\tT} Q^{-1}
            - \eps Q^{-1} \Delta \Sigma A Q
            - \eps Q^{-1} \Delta \mu b^{\tT} A Q
        \end{pmatrix}
        + O(\eps^2)
    \end{align*}
    Thus, ignoring higher order terms, omitting the $\Delta$ for clarity, and dividing by $\eps \ne 0$ yields the linearized system
    \begin{equation*}
        \begin{pmatrix}
            \dot{\mu} \\
            \dot{\Sigma}
        \end{pmatrix}
        = \begin{pmatrix}
            - Q^{-1} \left( \Sigma A b  
            + K(b, b) \mu\right) \\
            - 2 \Sym \big(Q A (\Sigma + b \mu^{\tT}) Q^{-1}\big)
        \end{pmatrix}.
    \end{equation*}
    By the properties of the Kronecker product from \cref{prop:Kronecker_Vectorization}, we can rewrite this as the linear system
    \begin{equation} \label{eq:linearizedSystem}
        \begin{pmatrix}
            \dot{\mu} \\
            \vec\big(\dot{\Sigma}\big)
        \end{pmatrix}
        = - B_A
        \begin{pmatrix}
            \mu \\ \vec(\Sigma)
        \end{pmatrix}
    \end{equation}
    with the system matrix
    \begin{equation*}
        B_A
        \coloneqq \begin{pmatrix}
            K(b, b) Q^{-1}      &  (b^{\tT} A) \otimes Q^{-1} \\
            (Q A b) \oplus Q^{-1} & Q^{-1} \oplus(Q A)
        \end{pmatrix} \in \R^{(d^2 + d) \times (d^2 + d)}.
    \end{equation*}
    which becomes block-diagonal for $b = 0$, and we set $A \oplus B \coloneqq A \otimes B + B \otimes A$.

    \begin{enumerate}
        \item 
        In \textit{one dimension}, i.e. $d = 1$, the linearized system \eqref{eq:linearizedSystem} simplifies to
        \begin{equation*}
            \begin{pmatrix}
                \dot{\mu} \\
                \dot{\Sigma}
            \end{pmatrix}
            = - \begin{pmatrix}
                K(b, b) Q^{-1}    & Q^{-1} A b \\
                2 A b               & 2 A
            \end{pmatrix}
            \begin{pmatrix}
                \mu \\
                \Sigma
            \end{pmatrix}.
        \end{equation*}
        The two eigenvalues of the (non-normal, but diagonalizable) matrix
        \begin{equation*}
            B_A 
            = \begin{pmatrix}
                Q^{-1} (A b^2 + 1)  & Q^{-1} A b \\
                2 A b               & 2 A
            \end{pmatrix}
        \end{equation*}
        are
        \begin{equation} \label{eq:EVals}
            \lambda_{\pm} \colon (0, \infty) \to \R, \qquad 
            A \mapsto \frac{1}{2 Q} \left[ (2 A Q + A b^2 + 1) \pm \sqrt{(2 A Q + A b^2 + 1)^2 - 8 A Q}\right].
        \end{equation}
        The eigenvalues are real, because the expression in the square root is nonnegative, since this is equivalent to saying that $(\sqrt{2 A Q} - 1)^2 \ge - A b^2$.
        Clearly, we have $\lambda_{+} \ge \lambda_{-} > 0$.

        We now distinguish two cases.
        For $b \ne 0$, the eigenvalue functions are differentiable in $A$ on $(0, \infty)$ with $\lambda_{\pm}'(A) > 0$ for all $A > 0$, and $\lim_{A \searrow 0} \lambda_{-}(A) = 0$ and $\lim_{A \searrow 0} \lambda_{+}(A) = \frac{1}{Q}$.
        Hence for $b \ne 0$, the set $\argmax_{A > 0} \lambda_-(A)$ is empty.
    
        As detailed in \cref{remark:convergence_rate_kappa}, we want to minimize 
        \begin{align*}
            \frac{\lambda_+(B_A)}{\lambda_-(B_A)}
            & = \frac{(2 Q + b^2) A + 1 + \sqrt{(2 A Q + A b^2 + 1)^2 - 8 A Q}}{(2 Q + b^2) A + 1 -\sqrt{(2 A Q + A b^2 + 1)^2 - 8 A Q}}.
        \end{align*}
        
        We have
        \begin{align*}
            \frac{\d}{\d A} \frac{\lambda_+(B_A)}{\lambda_-(B_A)}
            = \frac{8 Q (A (b^2 + 2 Q) - 1) \big((A (b^2 + 2 Q) + 1)^2 - 8 A Q\big)^{-\frac{1}{2}}}{(A b^2 - \sqrt{(A (b^2 + 2 Q) + 1)^2 - 8 A Q} + 2 A Q + 1)^2}.
        \end{align*}
        Setting this quantity to zero and solving for $A$ yields
        \begin{equation*}
            A = \frac{1}{b^2 + 2 Q}.
        \end{equation*}
        For $b = 0$, the eigenvalues \eqref{eq:EVals} simplify to 
        \begin{equation*}
            \lambda_{\pm}(A)
            = \mp \left| A - \frac{1}{2 Q} \right| + A + \frac{1}{2 Q}
            = \begin{cases}
                \frac{1}{Q}, & \text{if } \mp \left(A - \frac{1}{2 Q}\right) \ge 0, \\
                2 A, & \text{else.}
            \end{cases}.
        \end{equation*}
        Hence,
        \begin{equation*}
            \argmax_{A > 0} \min\big(\lambda_{-}(A), \lambda_+(A)\big)
            = \argmax_{A > 0} \lambda_{-}(A)
            = \left[\frac{1}{2 Q}, \infty \right),
        \end{equation*}
        and for $A^* = \frac{1}{2 Q}$ we get
        \begin{equation*}
            \lambda_{-}(A^*)
            = \frac{1}{Q},
        \end{equation*}
        and thus the continuous time rate $\| x(t) - x^* \|_2 \le C_1 e^{- \frac{1}{Q} t}$.
        Furthermore, 
        \begin{equation*}
            \frac{\lambda_+(A)}{\lambda_-(A)}
            = \begin{cases}
                \frac{1}{2 Q A},    & \text{if } A < \frac{1}{2 Q}, \\
                2 Q A,              & \text{else.}
            \end{cases}
        \end{equation*}
        has the unique minimizer $A = \frac{1}{2 Q}$ with minimal value $\kappa_2(B_A) = 1$.
    
        For the choice $A^* = \frac{1}{2 Q + v^2}$, the optimal step size $h^* = \frac{2}{\lambda_{\min}(B_{A^*}) + \lambda_{\max}(B_{A^*})} = \frac{2}{\frac{1}{Q} + \frac{1}{Q}} = Q$.

        \item 
        If $A$ and $Q$ are simultaneously diagonalizable, i.e. $A = V D_A V^{-1}$ and $Q = V D_Q V^{-1}$ with $D_A = \diag(a_i)$ and $D_Q = \diag(q_i)$, and $b = 0$, then
        \begin{equation*}
            B_A
            = \begin{pmatrix}
                Q^{-1} & 0 \\
                0 & Q^{-1} \oplus (Q A)
            \end{pmatrix}
        \end{equation*}
        and with $U \coloneqq \diag(V, V \otimes V)$, the diagonal matrix
        \begin{equation*}
            \tilde{B}_A
            \coloneqq U^{-1} B_A U
            = \begin{pmatrix}
                D_Q^{-1} & 0 \\
                0 & D_Q^{-1} \oplus (D_Q D_A)
            \end{pmatrix},
        \end{equation*}
        has the positive eigenvalues $\frac{1}{q_i}$ and $\frac{q_i a_i}{q_j} + \frac{q_j a_j}{q_i}$ for $i, j \in \{ 1, \ldots, d \}$.
        We have the $A$-independent lower bound
        \begin{equation*}
            \frac{\lambda_{\max}(B_A)}{\lambda_{\min}(B_A)}
            \ge \frac{\frac{1}{\lambda_{\min}(Q)}}{\frac{1}{\lambda_{\max}(Q)}}
            = \frac{\lambda_{\max}(Q)}{\lambda_{\min}(Q)}.
        \end{equation*}
        If we choose $A = \frac{1}{2} Q^{-1}$, then $a_i = \frac{1}{2 q_i}$, so the eigenvalues become $\frac{1}{q_i}$ and $\frac{1}{2 q_j} + \frac{1}{2 q_i}$, and this lower bound is achieved, thus achieving the lowest condition number.
    \end{enumerate} 
\end{proof}

\begin{remark}[Justifying the linearization at equilibrium]
    As mentioned in \cite[p.~591]{N2018}, the asymptotic stability of the linearized system is preserved, and solutions of the non-linear and the linear system close to the equilibrium point have the same qualitative stability features.
\end{remark}


\subsection{Proof of \texorpdfstring{\cref{thm:optimal_damping}}{a lemma}} \label{subsec:proof_optimal_damping}

\begin{proof}
        Linearizing \eqref{eq:ASVGD_Gaussians} at equilibrium using $\mu_t \approx b + \Delta t \tilde{\mu}_t$, $\Sigma_t \approx Q + \Delta t \tilde{\Sigma}_t$, $\nu_t \approx \Delta t \tilde{\nu}_t$, and $S_t \approx \Delta t \tilde{S}_t$, yields
        \begin{equation*}
            \begin{cases}
                \Delta t \dot{\tilde{\mu}}_t
                = 2 \Delta t \tilde{S}_t (Q + \Delta t \tilde{\Sigma}_t) A (b + \Delta t \tilde{\mu}_t)
                + \big((b + \Delta t \tilde{\mu_t})^{\tT} A (b + \Delta t \tilde{\mu_t}) + 1\big) \Delta t \tilde{\nu}_t, \\
                \Delta t \dot{\tilde{\Sigma}}_t
                = \Sym\big((Q + \Delta t \tilde{\Sigma}_t) A \big(2 (Q + \Delta t \tilde{\Sigma}_t) \Delta t \tilde{S}_t
                + (b+ \Delta t \tilde{\mu}_t) \Delta t \tilde{\nu}_t^{\tT}\big)\big) \\
                \qquad \qquad + \Delta \tilde{\nu}_t (b + \Delta t \tilde{\mu}_t)^{\tT} A (Q + \Delta t \tilde{\Sigma}_t) + 2 \Sym((Q + \Delta t \tilde{\Sigma}_t) A (Q + \Delta t \tilde{\Sigma}_t) \Delta t \tilde{S}_t), \\
                \Delta t \dot{\tilde{\nu}}_t
                = - \alpha_t \Delta t \tilde{\nu}_t
                - 2 A (Q + \Delta t \tilde{\Sigma}_t) \Delta t \tilde{S}_t \Delta t \tilde{\nu}_t
                - A (b + \Delta t \tilde{\mu}_t) \| \Delta t \tilde{\nu}_t \|_2^2
                - Q^{-1} \Delta t \tilde{\mu}_t, \\
                \Delta t \dot{\tilde{S}}_t
                = - \alpha_t \Delta t \tilde{S}_t
                - 2 \Delta t \tilde{S}_t \Delta t \tilde{\nu}_t (b + \Delta t \tilde{\mu}_t)^{\tT} A
                - 4 \Sym\big( (\Delta t \tilde{S}_t)^2 (Q + \Delta t \tilde{\Sigma}_t) A \big) \\
                \qquad \qquad - \frac{1}{2}\big(Q^{-1} - (Q + \Delta t \tilde{\Sigma}_t)^{-1}\big).
            \end{cases}
        \end{equation*}
        Dividing by $\Delta t$ and removing the tildes for better readability yields
        \begin{equation*}
            \begin{cases}
                \dot{\mu}_t
                = 2 S_t (Q + \Delta t \Sigma_t) A (b + \Delta t \mu_t)
                + \big((b + \Delta t \mu_t)^{\tT} A (b + \Delta t \mu_t) + 1\big) \nu_t, \\
                \dot{\Sigma}_t
                = \Sym\bigg( (Q + \Delta t \Sigma_t) A \big( 2(Q + \Delta t \Sigma_t) S_t + (b + \Delta t \mu_t) \nu_t^{\tT}\big)\bigg) \\
                \qquad \qquad + \nu_t (b + \Delta t \mu_t)^{\tT} A (Q + \Delta t \Sigma_t) + 2 \Sym\big((Q + \Delta t \Sigma_t) A (Q + \Delta t \Sigma_t) S_t\big), \\
                \dot{\nu}_t
                = - \alpha_t \nu_t
                - 2 A (Q + \Delta t \Sigma_t) S_t \Delta t \nu_t
                - A (b + \Delta t \mu_t) \Delta t \| \nu_t \|_2^2
                - Q^{-1} \mu_t, \\
                \dot{S}_t
                = - \alpha_t S_t
                - 2 \tilde{S}_t \Delta t \tilde{\nu}_t (b + \Delta t \tilde{\mu}_t)^{\tT} A
                - 4 \Sym\big( \Delta t (\tilde{S}_t)^2 (Q + \Delta t \tilde{\Sigma}_t) A \big) \\
                \qquad \qquad - \frac{1}{2 \Delta t}\big(Q^{-1} - (Q + \Delta t \tilde{\Sigma}_t)^{-1}\big).
            \end{cases}
        \end{equation*}
        Sending $\Delta t \to 0$ yields
        \begin{equation} \label{eq:linearized_system}
            \begin{cases}
                \dot{\mu}_t
                = 2 S_t Q A b + K(b, b) \nu_t, \\
                \dot{\Sigma}_t
                = \nu_t b^{\tT} A Q
                + \Sym\big( Q A ( 2 Q S_t + b \nu_t^{\tT})\big)
                + 2 \Sym(Q A Q S_t), \\
                \dot{\nu}_t
                = - \alpha_t \nu_t 
                - Q^{-1} \mu_t \\
                \dot{S}_t
                = - \alpha_t S_t - \frac{1}{2} Q^{-1} \Sigma_t Q^{-1}.
            \end{cases}
        \end{equation}
        We rewrite the linearized system \eqref{eq:linearized_system} as
        \begin{equation*}
            \partial_t \begin{pmatrix}
                \mu_t \\
                \vec(\Sigma_t) \\
                \nu_t \\
                \vec(S_t)
            \end{pmatrix}
            = - B_{A, Q, \alpha}
            \begin{pmatrix}
                \mu_t \\
                \vec(\Sigma_t) \\
                \nu_t \\
                \vec(S_t)
            \end{pmatrix}
        \end{equation*}
        with
        \begin{equation} \label{eq:AVGD_Gaussians_linearized}
            B_A
            \coloneqq - \begin{pmatrix}
                0           & 0                                     & - K(b, b) \id                                    & - 2 (b^{\tT} A Q) \otimes \id \\
                0           & 0                                     & - \frac{3}{2} (Q A b) \otimes \id - \frac{1}{2} \id \otimes Q A b     & - 2 Q A Q \oplus \id \\
                Q^{-1}      & 0                                     & \alpha_t \id                                                          & 0 \\
                0           & \frac{1}{2} Q^{-1} \otimes Q^{-1}     & 0                                                                     & \alpha_t \id \\
            \end{pmatrix}
        \end{equation}
        Since $\mu_0 = b = 0$, we strike the first and third rows and columns, to obtain the matrix
        \begin{equation*}
            B_{A, Q, \alpha}
            \coloneqq \begin{pmatrix}
                0 & - 2 Q A Q \oplus \id \\
                \frac{1}{2} Q^{-1} \otimes Q^{-1}                                                                          & \alpha \id \\
            \end{pmatrix}
            \in \R^{2 d^2 \times 2 d^2}.
        \end{equation*}
        Since $A = V D_A V^{\tT}$ and $Q = V D_Q V^{\tT}$ for an orthogonal matrix $V \in \R^{d \times d}$ and invertible diagonal matrices $D_A, D_Q \in \R^{d \times d}$ with entries $(a_i)_{i = 1}^{d}$ and $(q_i)_{i = 1}^{d}$ respectively, we have by the mixed product property that
        \begin{equation*}
            \tilde{V}^{-1} B_{A, Q, \alpha} \tilde{V} 
            = \begin{pmatrix}
                0                                       & - 2 (D_Q^2 D_A) \oplus \id_{d \times d} \\
                \frac{1}{2} D_Q^{-1} \otimes D_Q^{-1}   & \alpha \id_{d^2 \times d^2},
            \end{pmatrix}
            \eqqcolon \tilde{B}_{A, Q, \alpha},
        \end{equation*}
        where $\tilde{V} \coloneqq \diag(V \otimes V, V \otimes V)$, so we can equivalently calculate the eigenvalues of the matrix $\tilde{B}_{A, Q, \alpha}$.
        Since for $\lambda \in \mathbb{C}$, the matrix $\tilde{B}_{Q, A, \alpha} - \lambda \id_{d^2 \times d^2}$ is a $2 \times 2$ block matrix with blocks of equal size and the two lower blocks commute, we have by the mixed product property of the Kronecker product that
        \begin{align*}
            \det(\tilde{B}_{A, Q, \alpha} - \lambda \id_{2 d^2 \times 2 d^2})
            & = \det\left( \lambda(\lambda - \alpha) \id_{d^2 \times d^2} + (D_Q D_A) \oplus D_Q^{-1}\right).
        \end{align*}
        Since, $(D_Q D_A) \oplus D_Q^{-1}$ is a diagonal matrix with entries $\mu_{i, j} \coloneqq \frac{q_i}{q_j} a_i + \frac{q_j}{q_i} a_j$ for $i, j \in \{ 1, \ldots, d \}$, it follows that
        \begin{align*}
            \det(\tilde{B}_{A, Q, \alpha} - \lambda \id_{2 d^2 \times 2 d^2})
            & = \prod_{i, j = 1}^{d} \left(\lambda (\lambda - \alpha) + \mu_{i, j} \right).
        \end{align*}
        Hence, the eigenvalues of the matrix $B_{A, Q, \alpha}$ are
        \begin{equation*}
            \lambda_{i, j}^{\pm}(A, Q, \alpha)
            \coloneqq \frac{\alpha}{2} \pm \frac{1}{2} \sqrt{\alpha^2 - 4 \mu_{i, j}}, \qquad i, j \in \{ 1, \ldots, d \}.
        \end{equation*}
        The choice of $\alpha$ giving the best convergence rate is
        \begin{equation*}
            \alpha^*
            \coloneqq \argmax_{\alpha > 0} \min_{\lambda \in \sigma(B_{A, Q, \alpha})} \Re(\lambda)
            = 2 \sqrt{\min_{i, j \in \{ 1, \ldots, d \}} \mu_{i, j}}
            = 2 \sqrt{2 \lambda_{\min}(A)}.
        \end{equation*}
        We used that $\Re(\lambda_{i, j}^{\pm}(Q, A, \alpha)) = \frac{\alpha}{2} \pm \frac{1}{2} \sqrt{\max(\alpha^2 - 4 \mu_{i, j}, 0)}$ is monotonically increasing in $\mu_{i, j}$, and that $\lambda_{i, j}^{+} \ge \lambda_{i, j}^{-}$.
        For this choice of $\alpha = \alpha^*$, we obtain the \enquote{condition number}        
        \begin{equation} \label{eq:tildekappaB}
            \tilde{\kappa}(B_{A, Q, \alpha^*})
            \coloneqq \frac{\max\limits_{\lambda \in \sigma(B_{A, Q, \alpha^*})} | \lambda |}{\min\limits_{\lambda \in \sigma(B_{A, Q, \alpha^*})} | \lambda |}
            = \sqrt{\frac{\max\limits_{1 \le i, j \le d} \mu_{i, j}}{2 \lambda_{\min}(A)}}.
        \end{equation}
        Here, we used that $(\alpha^*)^2 \le 4 \mu_{i, j}$ for all $i, j \in \{ 1, \ldots, d \}$ and that
        \begin{equation*}
            | \lambda_{i, j}^{\pm}(Q, A, \alpha) |
            = \begin{cases}
                \lambda_{i, j}^{\pm}(Q, A, \alpha), & \text{if } \alpha^2 \ge 4 \mu_{i, j}, \\
                \mu_{i, j}, & \text{else.}
            \end{cases}
        \end{equation*}
        The expression \eqref{eq:tildekappaB} is 1-homogeneous in $A$, so $\min_{A > 0} \tilde{\kappa}(B_{A, Q, \alpha^*}) = \kappa(Q) + \frac{1}{\kappa(Q)}$ is achieved for $A^* = \theta \id_{d \times d}$ for any $\theta > 0$.
        This then yields
        \begin{equation*}
            \alpha^*
            = \sqrt{8 \theta}
            \qquad \text{and} \qquad
            \tilde{\kappa}(B_{A^*, Q, \alpha^*})
            = \sqrt{\frac{1}{2}\left(\kappa(Q) + \frac{1}{\kappa(Q)}\right)},
        \end{equation*}
        and so $\tilde{\kappa}(B_{A^*, Q, \alpha^*})$ is minimal (that is, equal to one) if and only if $\kappa(Q)$ is minimal (that is, equal to one).

        The stepwise contraction factor when using the optimal step size
        \begin{equation*}
            h^*
            = \frac{2}{\max\limits_{\lambda \in \sigma(B_{A, Q, \alpha^*})} | \lambda | + \min\limits_{\lambda \in \sigma(B_{A, Q, \alpha^*})} | \lambda |}   
            = \frac{2}{\sqrt{\max\limits_{1 \le i, j \le d} \mu_{i, j}} + \sqrt{2 \lambda_{\min}(A)}}
        \end{equation*}
        in the explicit Euler scheme is
        \begin{equation*}
            \rho^*
            \coloneqq \frac{\tilde{\kappa}(B_{A^*, Q, \alpha^*}) - 1}{\tilde{\kappa}(B_{A^*, Q, \alpha^*}) + 1}
            = \frac{\sqrt{\frac{1}{2}\left(\kappa(Q) + \frac{1}{\kappa(Q)}\right)} - 1}{\sqrt{\frac{1}{2}\left(\kappa(Q) + \frac{1}{\kappa(Q)}\right)} + 1}
            = \frac{\sqrt{\kappa(Q)^2 + 1} - \sqrt{2 \kappa(Q)}}{\sqrt{\kappa(Q)^2 + 1} + \sqrt{2 \kappa(Q)}}
            < 1.
        \end{equation*}
        Note that for all $x > 1$ we have $\sqrt{x} > \sqrt{\frac{1}{2}\left(x + \frac{1}{x}\right)}$ and thus $\rho^* < \frac{\sqrt{\kappa(Q)} - 1}{\sqrt{\kappa(Q)} + 1}$.
\end{proof}

\end{appendices}

\end{document}